\documentclass[11pt,twoside]{preprint}

\usepackage{hyperref}
\usepackage{breakurl}

\usepackage{times}
\usepackage{microtype}
\usepackage{amsmath}
\usepackage{amssymb}
\usepackage{tikz}
\usepackage{wasysym}
\usepackage{centernot}

\usepackage{mathrsfs}
\usepackage{mhequ}
\usepackage{mhenvs}
\usepackage{mhsymb}
\usepackage[top=4cm, bottom=3.5cm, left=3.5cm, right=3.2cm]{geometry}

\colorlet{symbols}{blue!90!black}
\def\symbol#1{\textcolor{symbols}{#1}}
\def\1{\mathbf{\symbol{1}}}
\def\X{\symbol{X}}
\def\sXi{\symbol{\Xi}}
\def\id{\mathrm{id}}
\def\emptyset{\mathop{\centernot\ocircle}}
\def\Wick#1{{:} #1 {:}}

\def\ex{\mathrm{ex}}

\colorlet{testcolor}{green!60!black}

\usetikzlibrary{calc}
\usetikzlibrary{shapes.misc}
\usetikzlibrary{shapes.symbols}
\usetikzlibrary{shapes.geometric}
\usetikzlibrary{snakes}
\usetikzlibrary{decorations}
\usetikzlibrary{decorations.markings}

\tikzset{
	root/.style={circle,fill=testcolor,inner sep=0pt, minimum size=2mm},
	dot/.style={circle,fill=symbols,draw=symbols,inner sep=0pt,minimum size=0.5mm},
	eps/.style={circle,fill=white,draw=symbols,inner sep=0pt,minimum size=1mm},
	int/.style={circle,fill=black,draw=black,inner sep=0pt,minimum size=0.7mm},
	var/.style={circle,fill=black!10,draw=black,inner sep=0pt, minimum size=2mm},
	dotred/.style={circle,fill=black!50,inner sep=0pt, minimum size=2mm},
	generic/.style={semithick,shorten >=1pt,shorten <=1pt},
	dist/.style={ultra thick,draw=testcolor,shorten >=1pt,shorten <=1pt},
	testfcn/.style={ultra thick,testcolor,shorten >=1pt,shorten <=1pt,<-},
	testfcnx/.style={ultra thick,testcolor,shorten >=1pt,shorten <=1pt,<-,
		postaction={decorate,decoration={markings,mark=at position 0.6 with {\drawx}}}},
	keps/.style={semithick,shorten >=1pt,shorten <=1pt,densely dashed,->},
	kprimex/.style={semithick,shorten >=1pt,shorten <=1pt,densely dashed,->,
		postaction={decorate,decoration={markings,mark=at position 0.4 with {\drawx}}}},
	kernel/.style={semithick,shorten >=1pt,shorten <=1pt,->},
	multx/.style={shorten >=1pt,shorten <=1pt,
		postaction={decorate,decoration={markings,mark=at position 0.5 with {\drawx}}}},
	kernelx/.style={semithick,shorten >=1pt,shorten <=1pt,->,
		postaction={decorate,decoration={markings,mark=at position 0.4 with {\drawx}}}},
	kernel1/.style={->,semithick,shorten >=1pt,shorten <=1pt,postaction={decorate,decoration={markings,mark=at position 0.45 with {\draw[-] (0,-0.1) -- (0,0.1);}}}},
	kernel2/.style={->,semithick,shorten >=1pt,shorten <=1pt,postaction={decorate,decoration={markings,mark=at position 0.45 with {\draw[-] (0.05,-0.1) -- (0.05,0.1);\draw[-] (-0.05,-0.1) -- (-0.05,0.1);}}}},
	kernelBig/.style={semithick,shorten >=1pt,shorten <=1pt,decorate, decoration={zigzag,amplitude=1.5pt,segment length = 3pt,pre length=2pt,post length=2pt}},
	rho/.style={dotted,semithick,shorten >=1pt,shorten <=1pt},
	renorm/.style={shape=circle,fill=white,inner sep=1pt},
	labl/.style={shape=rectangle,fill=white,inner sep=1pt},
	xi/.style={circle,fill=symbols!10,draw=symbols,inner sep=0pt,minimum size=1.2mm},
	xix/.style={crosscircle,fill=symbols!10,draw=symbols,inner sep=0pt,minimum size=1.2mm},
	xib/.style={circle,fill=symbols!10,draw=symbols,inner sep=0pt,minimum size=1.6mm},
	xibx/.style={crosscircle,fill=symbols!10,draw=symbols,inner sep=0pt,minimum size=1.6mm},
	not/.style={circle,fill=symbols,draw=symbols,inner sep=0pt,minimum size=0.5mm},
cumu2n/.style={inner sep=3pt},
cumu2/.style={draw=red!80,fill=red!40},
cumu2b/.style={draw=blue!80,fill=blue!40},
cumu2nv/.style={inner sep=3pt},
cumu2v/.style={draw=red!80,fill=white,very thick},
cumu3/.style={regular polygon, regular polygon sides=3,draw=red!80,rounded corners=3pt,fill=red!40,minimum size=5mm},
cumu4/.style={regular polygon, regular polygon sides=4,draw=red!80,rounded corners=3pt,fill=red!40,minimum size=7mm},
cumu5/.style={regular polygon, regular polygon sides=5,draw=red!80,rounded corners=3pt,fill=red!40,minimum size=7mm},
	>=stealth,
	}
\makeatletter
\def\DeclareSymbol#1#2#3{\expandafter\gdef\csname MH@symb@#1\endcsname{\tikz[baseline=#2,scale=0.15,draw=symbols]{#3}}}
\def\<#1>{\csname MH@symb@#1\endcsname}
\makeatother

\DeclareSymbol{X}{-2.4}{\node[dot] {};}
\DeclareSymbol{1}{0}{\draw[white] (-.5,0) -- (.5,0); \draw (0,0)  -- (0,1.5) node[dot] {};}
\DeclareSymbol{2}{0}{\draw (-0.5,1.5) node[dot] {} -- (0,0) -- (0.5,1.5) node[dot] {};}
\DeclareSymbol{3}{0}{\draw (0,0) -- (0,1.5) node[dot] {}; \draw (-.7,1.3) node[dot] {} -- (0,0) -- (.7,1.3) node[dot] {};}
\DeclareSymbol{E3}{0}{\draw (0,0) -- (0,1.5) node[dot] {}; \draw (-.7,1.3) node[dot] {} -- (0,0) -- (.7,1.3) node[dot] {}; \node[eps] (0,0) {};}

\DeclareSymbol{E5E4}{1}{
\draw (.5,2) node[dot] {} -- (1.5,1.5) -- (2.5,2) node[dot] {};
\draw (1.5,1.5) -- (1.5,2.7) node[dot] {}; \draw (0.8,2.5) node[dot] {} -- (1.5,1.5) -- (2.2,2.5) node[dot] {};
\draw (0,0) -- (0,1.5) node[dot] {}; \draw (-.7,1.3) node[dot] {} -- (0,0) -- (.7,1.3) node[dot] {};
\draw (-1.2,0.7) node[dot] {} -- (0,0);
\draw (0,0) to[bend right=50] (1.5,1.5); 
\node[eps] at (0,0) {};\node[eps] at (1.5,1.5) {};
}
\DeclareSymbol{E4E4}{1}{
\draw (.5,2) node[dot] {} -- (1.5,1.5) -- (2.5,2) node[dot] {};
\draw (1,2.6) node[dot] {} -- (1.5,1.5) -- (2,2.6) node[dot] {};
\draw (0,0) -- (0,1.5) node[dot] {}; \draw (-.7,1.3) node[dot] {} -- (0,0) -- (.7,1.3) node[dot] {};
\draw (-1.2,0.7) node[dot] {} -- (0,0);
\draw (0,0) to[bend right=50] (1.5,1.5); 
\node[eps] at (0,0) {};\node[eps] at (1.5,1.5) {};
}
\DeclareSymbol{3E4}{1}{
\draw (1.5,1.5) -- (1.5,2.7) node[dot] {}; \draw (0.8,2.5) node[dot] {} -- (1.5,1.5) -- (2.2,2.5) node[dot] {};
\draw (0,0) -- (0,1.5) node[dot] {}; \draw (-.7,1.3) node[dot] {} -- (0,0) -- (.7,1.3) node[dot] {};
\draw (-1.2,0.7) node[dot] {} -- (0,0);
\draw (0,0) to[bend right=50] (1.5,1.5); 
\node[eps] at (0,0) {};
}

\DeclareSymbol{3E3}{1}{
\draw (1.5,1.5) -- (1.5,2.7) node[dot] {}; \draw (0.8,2.5) node[dot] {} -- (1.5,1.5) -- (2.2,2.5) node[dot] {};
\draw (0,0) -- (0,1.5) node[dot] {}; \draw (-.7,1.3) node[dot] {} -- (0,0) -- (.7,1.3) node[dot] {};
\draw (0,0) to[bend right=50] (1.5,1.5); 
\node[eps] at (0,0) {};
}

\DeclareSymbol{E5}{0}{\draw (0,0) -- (0,1.5) node[dot] {}; \draw (-.7,1.3) node[dot] {} -- (0,0) -- (.7,1.3) node[dot] {};
\draw (-1.2,0.7) node[dot] {} -- (0,0) -- (1.2,0.7) node[dot] {}; \node[eps] (0,0) {};}
\DeclareSymbol{E52}{-3}{\draw (0,0) -- (0,-1); \draw (1.5,-0.4) node[dot] {} -- (0,-1) -- (-1.5,-0.4) node[dot] {}; 
\draw (-1,0.5) node[dot] {} -- (0,0) -- (1,0.5) node[dot] {};
\draw (0,0) -- (0,1.2) node[dot] {}; \draw (-.7,1) node[dot] {} -- (0,0) -- (.7,1) node[dot] {};\node[eps] (0,0) {};}
\DeclareSymbol{E32}{-3}{\draw (0,0) -- (0,-1); \draw (1.5,-0.4) node[dot] {} -- (0,-1) -- (-1.5,-0.4) node[dot] {}; 
\draw (0,0) -- (0,1.2) node[dot] {}; \draw (-.7,1) node[dot] {} -- (0,0) -- (.7,1) node[dot] {};\node[eps] (0,0) {};}
\DeclareSymbol{E50}{-3}{\draw (0,0) -- (0,-1); 
\draw (-1,0.5) node[dot] {} -- (0,0) -- (1,0.5) node[dot] {};
\draw (0,0) -- (0,1.2) node[dot] {}; \draw (-.7,1) node[dot] {} -- (0,0) -- (.7,1) node[dot] {};\node[eps] (0,0) {};}

\DeclareSymbol{3E2}{-3}{\draw (0,0) -- (0,-1); \draw (1,0) node[dot] {} -- (0,-1) -- (-1,0) node[dot] {}; \draw (0,0) -- (0,1.2) node[dot] {}; \draw (-.7,1) node[dot] {} -- (0,0) -- (.7,1) node[dot] {};\node[eps] at (0,-1) {};}

\DeclareSymbol{31}{-3}{\draw (0,0) -- (0,-1) -- (1,0) node[dot] {}; \draw (0,0) -- (0,1.2) node[dot] {}; \draw (-.7,1) node[dot] {} -- (0,0) -- (.7,1) node[dot] {};}
\DeclareSymbol{30}{-3}{\draw (0,0) -- (0,-1); \draw (0,0) -- (0,1.2) node[dot] {}; \draw (-.7,1) node[dot] {} -- (0,0) -- (.7,1) node[dot] {};}
\DeclareSymbol{32}{-3}{\draw (0,0) -- (0,-1); \draw (1,0) node[dot] {} -- (0,-1) -- (-1,0) node[dot] {}; \draw (0,0) -- (0,1.2) node[dot] {}; \draw (-.7,1) node[dot] {} -- (0,0) -- (.7,1) node[dot] {};}
\DeclareSymbol{32b}{-4.3}{\draw (0,0.5) -- (0,-1); \draw (1,0) node[dot] {} -- (0,-1) -- (-1,0) node[dot] {};}
\DeclareSymbol{22}{-3}{\draw (0,0.3) -- (0,-1); \draw (1,0) node[dot] {} -- (0,-1) -- (-1,0) node[dot] {};\draw (-.7,1) node[dot] {} -- (0,0.3) -- (.7,1) node[dot] {};}
\DeclareSymbol{20}{-3}{\draw (0,0) -- (0,-1);\draw (-.7,1) node[dot] {} -- (0,0) -- (.7,1) node[dot] {};}
\DeclareSymbol{12}{-3}{\draw (0,0.3) -- (0,-1); \draw (1,0) node[dot] {} -- (0,-1) -- (-1,0) node[dot] {};\draw (-.7,1) node[dot] {} -- (0,0.3);}
\DeclareSymbol{10}{-3}{\draw (0,0.3) -- (0,-1);\draw (-.7,1) node[dot] {} -- (0,0.3);}
\DeclareSymbol{21}{-3}{\draw (0,0.3) -- (0,-1) -- (1,0) node[dot] {};\draw (-.7,1) node[dot] {} -- (0,0.3) -- (.7,1) node[dot] {};}

\newtheorem{assumption}[lemma]{Assumption}

\definecolor{darkred}{rgb}{0.9,0.1,0.1}

\def\comment#1{\ifthenelse{\isodd{\value{page}}}{\marginpar{\raggedright\scriptsize{\textcolor{darkred}{#1}}}}{\marginpar{\raggedleft\scriptsize{\textcolor{darkred}{#1}}}}}  

\newop{diag}

\def\II{\mathscr{I}}
\def\EE{\mathscr{E}}

\let\D\CD

\def\TT{\mathscr{T}}
\def\MM{\mathscr{M}}
\def\LL{\mathscr{L}}
\def\fM{\mathfrak{M}}

\def\WW{{\mathbb W}}

\def\E{{\mathbf E}}
\def\P{\mathbf{P}}
\def\T{\mathbf{T}}

\def\Cum{\mathbf{E}_c}
\def\powerset{\mathscr{P}}

\def\${|\!|\!|}
\def\l|{\left|\!\left|\!\left|}
\def\r|{\right|\!\right|\!\right|}

\def\s{\mathfrak{s}}
\def\K{\mathfrak{K}}
\def\RR{\mathfrak{R}}

\def\sym{\mathrm{sym}}

\def\PPi{\boldsymbol{\Pi}}
\def\Ren{\mathscr{R}}

\def\Deltap{\Delta^{\!+}}

\begin{document}

\title{Regularity structures and the dynamical $\Phi^4_3$ model}
\author{Martin Hairer}
\institute{The University of Warwick, \email{M.Hairer@Warwick.ac.uk}}

\maketitle

\begin{abstract}
We give a concise overview of the
theory of regularity structures as first exposed in \cite{Regular}.
In order to allow to focus on the conceptual aspects of the theory, many proofs are
omitted and statements are simplified. In order to provide both motivation and focus, 
we concentrate on the study of solutions to the stochastic quantisation 
equations for the Euclidean $\Phi^4_3$ quantum field theory which can be obtained
with the help of this theory. In particular, we sketch the proofs of how one can 
show that this
model arises quite naturally as an idealised limiting object for several classes
of smooth models.
\end{abstract}

\tableofcontents

\section{Introduction}

The purpose of these notes is to give a short informal introduction to the main
concepts of the theory of regularity structures, focusing as an example on the construction
and approximation of the dynamical $\Phi^4_3$ model.
They expand and complement the notes \cite{Brazil} which focus more on the
general theory and the ideas of the proofs of the main abstract results.
Here, we instead focus mainly on the construction of the dynamical
$\Phi^4_3$ model, as well as on the way in which the abstract theory can be used to obtain
a number of rather non-obvious approximation results for this model.

The theory of ``regularity structures'', introduced in \cite{Regular}, 
unifies various flavours of the theory of (controlled)
rough paths (including Gubinelli's theory of controlled rough paths \cite{Max},
as well as his branched rough paths \cite{Trees}), as well as the usual
Taylor expansions. While it has its conceptual 
roots in Lyons's theory of rough paths \cite{Lyons},
its main advantage is that it is no longer tied to the
one-dimensionality of the time parameter, which makes it also suitable for the description of 
solutions to stochastic \textit{partial} differential equations, rather than just stochastic ordinary 
differential equations. This broader scope requires a theory that admits more flexibility than the
theory of rough paths, so we will see that the underlying algebraic structure (which in the 
case of the theory of rough paths is always given by the tensor algebra endowed with the concatenation
and shuffle products) is problem-dependent and enters as a parameter of the theory. 
While the exposition of these notes aims to be mostly self-contained, none of the proofs 
will be given in detail,
instead we will only sketch the main arguments.

The main achievement of the theory of
regularity structures is that it allows to give a meaning, as well as a robust 
approximation theory, to ill-posed stochastic 
PDEs that arise naturally when trying to describe the crossover regime between two
universality classes for various system of statistical mechanics. One
example of such an equation is the KPZ
equation arising as a natural model for one-dimensional interface motion in the crossover
regime between the Edwards-Wilkinson and the KPZ universality classes 
\cite{KPZOrig,MR1462228,KPZSurvey,KPZ}:
\begin{equ}[e:KPZ]\tag{KPZ}
\d_t h = \d_x^2 h + (\d_x h)^2 + \xi - C\;.
\end{equ}
Another example is the dynamical $\Phi^4_3$ model arising for example in the stochastic quantisation of Euclidean
quantum field theory \cite{ParisiWu,MR815192,AlbRock91,MR2016604,Regular}. However, it also
arises as a description of the crossover regime for the dynamic of phase coexistence models 
near their critical point between the mean-field theory and the ``Wilson-Fisher'' renormalisation 
fixed point \cite{MR1661764,HendrikIsing}. This model can be written formally as
\begin{equ}[e:Phi4]\tag{$\Phi^4$}
\d_t \Phi = \Delta \Phi  + C \Phi - \Phi^3 + \xi\;.
\end{equ}
In both of these examples, $\xi$ denotes space-time white noise, namely the generalised
Gaussian random field such that
\begin{equ}[e:defnoise]
\E \xi(t,x)\xi(s,y) = \delta(t-s)\delta(y-x)\;.
\end{equ}
Furthermore, $C$ is a constant (which will actually turn out to be
infinite in some sense!), and we consider these equations
on bounded tori. 
In the case of the dynamical $\Phi^4_3$ model, the spatial
variable has dimension $3$, while it has dimension $1$ in the case of the KPZ equation. 
Why are these equations problematic? As one can guess from \eqref{e:defnoise}, typical
realisations of the noise $\xi$ are not functions, but rather irregular 
space-time distributions. As a matter of fact, it follows immediately from
\eqref{e:defnoise} that the law of $\xi$ is invariant under the substitution
$\xi(t,x) \mapsto \lambda^{{d\over 2} + 1} \xi(\lambda^2 t, \lambda x)$, which
correctly suggests that its samples typically belong to the parabolic $\CC^\alpha$
spaces only for $\alpha < -1 - {d\over 2}$. (We write $\CC^\alpha$ as a shorthand
for the parabolic Besov space $B^\alpha_{\infty,\infty,\mathrm{loc}}$, which coincides with the
usual parabolic $\CC^\alpha$ spaces for positive non-integer values of $\alpha$.
It is in some sense the largest space of distributions that is 
invariant under the scaling $\phi(\cdot) \mapsto \lambda^{-\alpha} \phi(\lambda^{-1} \cdot)$,
see for example \cite{Bourgain}.)

As a consequence, even the solution $u$ to the simplest
parabolic stochastic PDE, the stochastic
heat equation $\d_t u = \Delta u + \xi$, is quite irregular. As a consequence of the above,
combined with classical parabolic Schauder estimates, one
can show that if the spatial variable has dimension $d$, then the solution $u$ belongs
to $\CC^\alpha$ (again in the parabolic sense) for every $\alpha < {2-d \over 2}$, but \textit{not} 
for $\alpha = {2-d\over 2}$. In particular, one expects solutions
$h$ to \eqref{e:KPZ} to be of class ``just about not'' H\"older-${1\over 2}$.
This begs the following question: if typical solutions $h$ are nowhere differentiable,
what does the nonlinearity $(\d_x h)^2$ mean? Similarly, we see that one needs to take $\alpha < 0$
in dimensions $2$ or higher, which means that in these dimensions
the solutions to the stochastic heat equation are no longer functions themselves, but only make
sense as space-time distributions. As a consequence, it is not clear at all what
the meaning of the nonlinearity $\Phi^3$ is in the dynamical $\Phi^4_3$ model.

This is the type of question that will be addressed in this article. It is 
hopeless to try to build a consistent theory allowing to multiply any two
Schwartz distributions, as already pointed out by Schwartz himself \cite{MR0064324}.
It is possible to extend the space of Schwartz distributions to some
larger class of objects that do form an algebra (and such that the product
of a Schwartz distribution with a smooth function has its usual meaning), and this
has been explored by Colombeau \cite{MR701451}. However, while this frameworks does allow one to
treat stochastic PDEs \cite{AlbRuss}, it is not clear at all what the meaning of the
resulting solutions are. In fact, there are very strong hints that solutions built using
such a theory are not ``physically relevant'': in the case of the sine-Gordon model
they fail to ``see'' the Kosterlitz-Thouless transition \cite{MR634447,Falco}, while the theory
of regularity structures does break down precisely at that transition \cite{Hao}, as expected for 
a theory designed to treat ``subcritical'' models.

Instead, the direction pursued here is to exploit as much as possible a priori information on
the model at hand. Instead of building a general theory allowing to multiply any two distributions, 
we are only interested in multiplying those distributions that could potentially arise in the
right hand side of the equation under consideration. The theory of regularity structures provides the tools
and techniques to design such ``purpose-built'' spaces in a systematic way, as well as the objects required 
to encode the various renormalisation procedures arising in these problems. 
While a full exposition
of the theory is well beyond the scope of this short introduction, 
we aim to give a concise overview to most of its concepts. In most cases,
we will only state results in a rather informal way and give some ideas as to 
how the proofs work, focusing on conceptual rather than technical issues. 
For precise statements and complete proofs of most of 
the results exposed here, we refer to the articles \cite{Regular,HaoShen,Jeremy,Weijun}.
The type of well-posedness results that can be proven using the techniques surveyed in this
article include the following.

\begin{theorem}\label{theo:construction}
Let $\xi_\eps = \rho_\eps * \xi$ denote the convolution of space-time 
white noise with a
compactly supported smooth mollifier $\rho$ that is scaled by $\eps$ in the spatial direction(s)
and by $\eps^2$ in the time direction. Denote by $\Phi_\eps$ the solution to
\begin{equ}
\d_t \Phi_\eps = \Delta \Phi_\eps +  C_\eps \Phi_\eps - \Phi_\eps^3 + \xi_\eps \;,\label{e:Phi43Regular}
\end{equ}
on the three-dimensional torus.
Then, there exist choices of constants $C_\eps$  diverging as $\eps \to 0$, as
well as a process $\Phi$ such that $\Phi_\eps \to \Phi$ in probability. 
Furthermore, if $C_\eps$ is suitable chosen, then $\Phi$ does not depend on the
mollifier $\rho$.
\end{theorem}

\begin{remark}
Very similar results have recently been obtained by slightly different techniques.
In \cite{CC13}, the authors use the theory of paracontrolled distributions
introduced in \cite{GIP12}. This theory is a kind of Fourier space analogue to the 
theory of regularity structures, but seems for the moment restricted to first-order
expansions. In \cite{Antti} on the other hand, the author uses a variant of Wilson's renormalisation
group ideas. 
\end{remark}

\begin{remark}
We made an abuse of notation, since the space-time white noise
appearing in the equation for $h_\eps$ is on $\R \times \T^1$, while the
one appearing in the equation for $\Phi_\eps$ is on $\R \times \T^3$.
Similarly, the mollifier $\rho_\eps$ is of course different for the two equations.
\end{remark}

\begin{remark}
In both cases,  convergence is in probability in $\CC^\alpha$ for $0 < \alpha < {1\over 2}$ in the
case of the KPZ equation and $-{2\over 3} < \alpha < -{1\over 2}$ in the case of the $\Phi^4_3$ model.
This requires the corresponding initial conditions to also belong to the relevant $\CC^\alpha$
spaces.
\end{remark}

It is also possible to show that various natural approximation schemes converge to the 
dynamical $\Phi^4_3$ model. Take for example a function $\theta \mapsto V_\theta$ taking values
in the space of even polynomials on $\R$ of some fixed degree $2m$ and consider the
equation
\begin{equ}[e:modelV]
\d_t \Phi_{\eps,\theta} = \Delta \Phi_{\eps,\theta} - \eps^{-3/2} V_\theta'(\sqrt \eps \Phi_{\eps,\theta}) + \xi_\eps \;,
\end{equ}
with $\xi_\eps$ as above. This particular scaling arises naturally when rescaling a weakly nonlinear
model, see \cite{Weijun} for more details. Denote by $\mu = \CN(0,C)$ the centred Gaussian measure with
covariance
\begin{equ}[e:defCC]
C = \|\rho \star P\|_{L^2}^2\;, 
\end{equ}
where $P$ denotes the $3$-dimensional heat kernel and $\rho$ is the mollifier appearing in the
definition of $\xi_\eps$. Note that since $d = 3$, this quantity is indeed finite. (In dimensions 
$d \le 2$ the heat kernel fails to be square integrable at large scales.)
We then define the ``effective potential'' $\scal{V_\theta} = \mu \star V_\theta$ and
assume that $\theta \mapsto \scal{V_\theta}$ exhibits a pitchfork bifurcation at the origin
at $\theta = 0$. We also normalise the solution in such a way that
$\scal{V_0}^{(4)}(0) = 6$, which guarantees that $\scal{V_0}'(u) = u^3 + \CO(u^5)$ for $u \ll 1$. 
The main result of \cite{Weijun}, which builds on the analogous results
obtained in \cite{Jeremy} for the KPZ equation then reads

\begin{theorem}\label{theo:universal}
In the above setting, there exist values $a > 0$ and $b \in \R$, such that if one sets 
$\theta = a \eps \log \eps + b \eps$, then the solution to \eqref{e:modelV} converges
as $\eps \to 0$ to the process $\Phi$ built in Theorem~\ref{theo:construction}.
\end{theorem}

\begin{remark}
Concerning the initial condition, one needs to consider a sequence $\Phi_0^{(\eps)}$ 
of smooth initial conditions converging to $\Phi_0 \in \CC^\alpha$ with 
$-{1\over 2} - {1\over 4m} < \alpha < -{1\over 2}$ in a suitable sense.
This is because if we chose a fixed initial condition in $\CC^{-{1/ 2}}$ say, 
\eqref{e:modelV} may fail to even possess local solutions for fixed $\eps > 0$.
\end{remark}

We can also consider approximations of the type given in Theorem~\ref{theo:construction}, but with
$\xi_\eps$ a non-Gaussian approximation to white noise. In this case, one typically
has to add additional counterterms in order to obtain the same limit. Consider for example
an approximation of the type
\begin{equ}[e:defxieps]
\xi_\eps(t,x) = \eps^{-5/2}\eta(t/\eps^2, x/\eps)\;,
\end{equ}
for a stationary process $\eta$ on $\R \times \R^3$ which admits moments
of all orders and has finite dependence in the sense that $\sigma_K$ and $\sigma_{\bar K}$
are independent whenever
\begin{equ}
\inf_{z\in K \atop \bar z \in \bar K} |z-\bar z| \ge 1\;.
\end{equ}
Here, we wrote $\sigma_K$ for the $\sigma$-algebra generated by all evaluation maps
 $\eta(z)$ for $z \in K$.\footnote{Since we always consider equations on bounded tori,
we actually need to consider a suitably periodised version of $\xi_\eps$, 
see \cite[Assumption~2.1]{HaoShen} for more details.}
It is then possible to show that there exist suitable choices of constants $C_\eps^{(i)}$
such that solutions to 
\begin{equ}[e:solNonGauss]
\d_t \Phi_\eps = \Delta \Phi_\eps + C_\eps^{(1)} + C_\eps^{(2)} \Phi_\eps - \Phi_\eps^3 + \xi_\eps \;,
\end{equ}
converge as $\eps \to 0$ to the process $\Phi$ built in Theorem~\ref{theo:construction}.
Details of the proof can be found in \cite{HaoShen} for the case of the KPZ equation.
It is shown in \cite[Remark~6.6]{HaoShen} that these constants are of the type
\begin{equ}[e:constNonGauss]
C_\eps^{(1)} = {C^{(1,1)} \over \eps^{3/2}} + {C^{(1,2)} \over \eps^{1/2}}\;,\qquad
C_\eps^{(2)} = {C^{(2,1)} \over \eps} + c \log \eps + C^{(2,2)}\;,
\end{equ}
where the $C^{(i,j)}$ depend on the details of the process $\eta$, but $c$ is universal.
We will provide a short sketch of the argument showing how these constants appear in
Section~\ref{sec:CLT} below.

Another very nice way in which the dynamical $\Phi^4$ model appears is
as the crossover regime for the Glauber dynamic in an Ising-Kac model \cite{HendrikIsing}.
Here, one considers an Ising model, but with the usual nearest-neighbour Hamiltonian replaced by 
a Hamiltonian of the type
\begin{equ}
H = \sum_{x,y} K_\gamma(|x-y|) \sigma_x \sigma_x \;,
\end{equ}
where $K_\gamma(r) = \gamma^{-d} K(\gamma r)$ for some small value $\gamma$, and $K$ is a positive, smooth, compactly
supported function. On is then interested 
in the simultaneous limit $\gamma \to 0$ and $N \to \infty$ (where $N$ is the side length of 
a discrete torus on which $\sigma$ lives) at (or sufficiently near) the critical temperature.
In dimension $d = 2$, it was shown in \cite{HendrikIsing} that by suitably choosing $N$ and the 
inverse temperature as a function of $\gamma$ and considering the Glauber dynamic on suitable long time scales,
one recovers the dynamical $\Phi^4_2$ model in the limit. A similar result is of course conjectured to hold for
$d=3$. We will not consider any discrete approximation of this type in these notes but refer instead to
\cite{AjayHendrik} for a recent review of the techniques involved in the proof of such a result. 

One of the main insights used in the proof of the type of results mentioned in this introduction 
is that while solutions to equations like \eqref{e:KPZ} or \eqref{e:Phi4}
may appear to be very rough (indeed, they are not even functions in the case of \eqref{e:Phi4}), they
can actually be considered to be smooth, provided that one looks at them in the right way. 
To understand what we mean by this, it is worth revisiting the very notion of ``regularity''. 
The way we usually measure regularity is by asking how well a given function can be approximated by polynomials.
More precisely, we say that a function $F\colon \R^d \to \R$ is of class $\CC^\gamma$ with $\gamma > 0$ 
if, for every
point $x \in \R^d$ we can find a polynomial $P_x$ of degree $\lfloor \gamma\rfloor$ such that
$|F(y) - P_x(y)| \lesssim |y-x|^\gamma$ for $y$ close to $x$. The idea now is to consider spaces
of ``regular'' functions / distributions, but where ``regularity'' is measured by local proximity not
to polynomials, but to linear combinations of some other ``basis functions'' that are specific 
to the problem at hand. In particular, the objects that play the role of polynomials are allowed
to be random themselves, and they are also allowed to be distributions rather than just functions. 
One might now hope (and this is indeed the case) that if we are given a consistent product rule on
these basic objects and if we have two distributions that are sufficiently ``smooth'' in the sense that
they are locally described by linear combinations of these objects (a ``local specification''), 
then there exists a unique distribution, which we interpret as the product, whose local specification
is given by the products of the local specifications of each factor. This then allows to formulate 
equations like \eqref{e:KPZ} or \eqref{e:Phi4} as fixed point problems in some of these spaces
of ``smooth'' functions / distributions, and to build a local solution theory with many 
nice properties very similar to those one would have for the corresponding deterministic problems
with smooth inputs.

The remainder of these notes is organised as follows. First, in Section~\ref{sec:Holder}, we
provide a fresh look at the definitions of the classical spaces $\CC^\gamma$ and we
describe a natural generalisation, as well as a situation where such a generalisation is useful.
In Section~\ref{sec:defs}, we then give the general definition of a regularity structure
and of the associated analogues to the spaces $\CC^\gamma$. We also formulate the reconstruction 
theorem which is fundamental to the theory and we give a simple application.
We then provide some operational tools in Sections~\ref{sec:prod} and \ref{sec:Schauder}.
Finally, we show in Sections~\ref{sec:app} and \ref{sec:Phi} how to apply the theory to
stochastic PDEs in general and the dynamical $\Phi^4_3$ model in particular.

\subsection*{Acknowledgements}

{\small
Financial support from the Leverhulme trust through a leadership award
and from the ERC through a consolidator award is gratefully acknowledged.
}

\section{Another look at smooth functions}
\label{sec:Holder}

Before we turn to a more precise description of how this idea of replacing Taylor polynomials
by purpose-built functions / distributions is implemented in practice, let us
first have a slightly different look at the definition of the classical H\"older spaces. For the
sake of this discussion, let us fix some exponent $\gamma \in (1,2)$. The space $\CC^\gamma(S^1)$ then
consists of those functions $f$ that are continuously differentiable and such that their
derivative $f'$ satisfies $|f'(t) - f'(s)| \lesssim |t-s|^{\gamma-1}$. A natural seminorm on
$\CC^\gamma$ is then given by
\begin{equ}[e:normalphabasic]
\|f\|_\gamma = \sup_{s\neq t} {|f'(t) - f'(s)| \over |t-s|^{\gamma-1}}\;,
\end{equ}
where $f'$ denotes the derivative of $f$. This definition is more complicated than
it may appear at first sight since it is given in terms of $f'$, rather than the function $f$ itself, 
and $f'$ has to be computed first. Instead, we would like to consider an element of $\CC^\gamma$ as
a \textit{pair} of functions $(f,f')$, endowed with the norm
\begin{equ}[e:normalpha]
\|(f,f')\|_\gamma = \sup_{s\neq t} {|f'(t) - f'(s)| \over |t-s|^{\gamma-1}} \vee \sup_{s\neq t} 
{|f(t) - f(s) - f'(s)(t-s)| \over |t-s|^{\gamma}}\;,
\end{equ}
where $a \vee b$ denotes the maximum of $a$ and $b$.
It is not difficult to verify that if $f'$ is the derivative of $f$, 
then as a consequence of the identity
\begin{equ}
f(t) - f(s) - f'(s)(t-s) = \int_s^t (f'(r) - f'(s))\,dr\;,
\end{equ}
the seminorm \eqref{e:normalpha} is indeed
equivalent to \eqref{e:normalphabasic}. Furthermore, the finiteness
of the second part of the expression \eqref{e:normalpha} \textit{forces} $f$ to be 
differentiable with derivative $f'$, since 
otherwise that supremum would be infinite. 
In this point of view, we therefore do not have to impose \textit{a priori} that $f'$ is the
derivative of $f$, this simply turns out to be a side-effect of having $\|(f,f')\|_\gamma < \infty$.
Furthermore, viewing the pair $(f,f')$ as our primary data has the advantage that we do not need
to compute any additional data (the derivative) 
from it in order to express its norm.
At this stage, one may worry that the notation quickly becomes cumbersome when considering
higher degrees of regularity. Indeed, for $\gamma \in (2,3)$ we could come up with a similar norm,
but this time we would have to consider triples $(f,f',f'')$, etc. Instead, we prefer to rewrite 
an element of $\CC^\gamma$ as a single function, but taking values in $\R \oplus \R \approx \R^2$, 
with the first component equal to $f$ and the second component equal to $f'$. In order to be able to easily
distinguish these components, we denote by $\1$ the basis vector corresponding to the first component
and by $\X$ the basis vector for the second component, so that our pair $(f,f')$ can be written
as
\begin{equ}
F(t) = f(t)\1 + f'(t)\X\;.
\end{equ}
At this stage, this is nothing but a change of notation.
However, this notation already suggests a very natural 
product rule for elements in $\CC^\gamma$: 
we \textit{postulate} that $\1 \cdot \1 = \1$, $\1\cdot \X = \X\cdot \1 = \X$, and $\X\cdot\X = 0$,
and we define $\bigl(F\cdot G\bigr)(t) = F(t)\cdot G(t)$. With this definition, one has
\begin{equ}
\bigl(F\cdot G\bigr)(t) = f(t)g(t)\1 + \bigl(f'(t)g(t) + f(t)g'(t)\bigr)\X\;,
\end{equ}
i.e.\ the component multiplying $\X$ is ``automatically'' given by what the Leibniz rule
suggests should be its correct expression. 

Each of the basis vectors $\tau \in \{\1, \X\}$ also comes with a natural ``degree''
(or ``homogeneity'') $|\tau|$ given by $|\1| = 0$ and $|\X| = 1$. 
There is an analytical meaning to this homogeneity: the vector $\1$ also quite
naturally represents the constant function $1$ and the vector $\X$ represents the
monomial of degree $1$ in a Taylor expansion around some base point. In this way,
we can view $F$ as a function taking values in the space of Taylor polynomials of
degree $1$.
This suggests the introduction
of a family of linear maps $\{\Pi_s\}_{s \in S^1}$ which associate to each basis
vector the corresponding Taylor monomial:
\begin{equ}[e:defPi]
\bigl(\Pi_s \1\bigr)(t) = 1\;,\qquad
\bigl(\Pi_s \X\bigr)(t) = t-s\;.
\end{equ}
With this notation, given $F \in \CC^\gamma$, the function $t \mapsto \bigl(\Pi_s F(s)\bigr)(t)$
is nothing but its first-order Taylor expansion at $s$. Furthermore, 
the degree of a vector $\tau \in \{\1,\X\}$ is precisely the order at which the map
$t \mapsto \bigl(\Pi_s\tau\bigr)(t)$ vanishes near $t = s$. The maps $\Pi_s$ yield an easy
way to recover the ``actual'' real-valued function described by the vector-valued function $F$:
setting
\begin{equ}[e:RFsimple]
\bigl(\CR F\bigr)(t) \eqdef \bigl(\Pi_t F(t)\bigr)(t)\;,
\end{equ}
we see that in this case $\CR F$ is nothing but the component of $F$ multiplying $\1$, which
is indeed the real-valued function $f$ represented by $F$.
One crucial property of the maps $\Pi_s$ is that varying $s$ can also be achieved
by composing them with an adequate linear transformation. Indeed, setting
\begin{equ}[e:basicGamma]
\Gamma_{st} \1 = \1\;,\qquad
\Gamma_{st} \X = \X + (s-t)\1\;,
\end{equ}
it is straightforward to verify that one has the identities
\begin{equ}[e:algebraic]
\Pi_t = \Pi_s \Gamma_{st} \;,\qquad \Gamma_{st}\Gamma_{tu} = \Gamma_{su}\;.
\end{equ}

With this notation,
and making a slight abuse of notation by also using $\{\1,\X\}$ for the dual basis, 
the norm \eqref{e:normalpha} can be rewritten in a more concise and much
more natural way as 
\begin{equ}[e:defalphaNorm]
\|F\|_\gamma = \sup_{\tau \in \{\1,\X\}} {|\scal{\tau,F(t) - \Gamma_{ts}F(s)}| \over |t-s|^{\gamma-|\tau|}}\;.
\end{equ}
It is now immediate to generalise these definitions to 
the case of $\gamma \in \R_+ \setminus \N$ by introducing additional basis vectors
$\X^k$, postulating that $\X^0 = \1$ and $\X^k \cdot \X^\ell = \X^{k+\ell}$, 
and extending the definition of the maps $\Gamma_{st}$ to these additional vectors
by imposing that it satisfies the multiplicative property
\begin{equ}
\Gamma_{st}(\tau \bar \tau) = (\Gamma_{st}\tau)(\Gamma_{st} \bar \tau)\;.
\end{equ}
One can then verify that \eqref{e:algebraic} still holds provided that 
one sets $\bigl(\Pi_s \X^k\bigr)(t) = (t-s)^k$.
This is a consequence of the fact that first changing the base point $s$ 
and then multiplying two monomials is the same as first multiplying them
and then changing the base point.

Note now that the definition of the seminorm 
\eqref{e:defalphaNorm} is extremely robust and does not refer anymore to 
any of the details that make a Taylor polynomial a polynomial. 
For example, for $\alpha \in (0,1)$ we could fix an $\alpha$-H\"older 
continuous (in the usual sense) function $W$ and, instead of \eqref{e:defPi}, we could
set
\begin{equ}
\bigl(\Pi_s \1\bigr)(t) = 1\;,\qquad
\bigl(\Pi_s \X\bigr)(t) = W(t)-W(s)\;.
\end{equ}
Since $\bigl(\Pi_s \X\bigr)(t)$ now only vanishes at order $\alpha$ near $t=s$,
this strongly suggests that we should set $|\X| = \alpha$. 
The identity \eqref{e:algebraic} is still satisfied if we change the definition
of $\Gamma_{st}$ into
\begin{equ}[e:gammaalt]
\Gamma_{st} \1 = \1\;,\qquad
\Gamma_{st} \X = \X + \bigl(W(s)-W(t)\bigr)\1\;.
\end{equ}
With these definitions, the norm \eqref{e:defalphaNorm} is still very natural, 
but its meaning now is that, setting $F(t) = f(t) \1 + f'(t)\X$ as before, 
the function $f'$ is H\"older continuous of order $\gamma-\alpha$ and the function
$f$ is such that
\begin{equ}[e:idenf]
f(t) = f(s) + f'(s) \bigl(W(t) - W(s)\bigr) + \CO(|t-s|^\gamma)\;.
\end{equ}
In particular, if $\gamma > \alpha$, it no longer implies means that $f$ is 
of class $\CC^\gamma$, only that its increments ``look like'' some multiple of 
those of $W$, up to a remainder of order $\gamma$.
Furthermore, $f'$ is now no longer equal to the derivative of $f$. 

\begin{remark}\label{rem:unique}
Consider the interesting case $0 < \alpha < \gamma < 1$ for the situation just described. Then it is 
no longer necessarily true that the ``derivative'' $f'$ is uniquely determined
by $f$ if all we know is that $\|F\|_\gamma < \infty$. If for example we take $W(t) = \cos(t)$
which is certainly $\alpha$-H\"older continuous, although it is of course much
more than that, then the identity \eqref{e:idenf} simply forces 
$f$ to be $\gamma$-H\"older continuous and puts no restriction whatsoever on $f'$.
If, on the other hand, $W$ happens to be ``nowhere $\gamma$-H\"older'' in the sense
that for every $s \in S^1$ one can find a sequence $t_n \to s$ such that
$|W(t_n) - W(s)|/|t_n -s|^\gamma \to \infty$, then \eqref{e:idenf} determines $f'$
uniquely. A quantitative version of this statement, together with an application,
can be found in \cite{Natesh}. 
\end{remark}

The functions of class ``$\CC^\gamma$ with respect to some function $W$'' just described in \eqref{e:idenf}
play a prominent role in the 
theory of controlled rough paths \cite{Lyons,Max,LyonsStFlour}, so let us see how these
definitions are useful there.
The setting is the following: we want to provide a robust solution theory for a
controlled differential equation of the type 
\begin{equ}[e:SDE]
dY = f(Y)\,dW(t)\;,
\end{equ}
 where $W \in \CC^\alpha$ is a rather rough function
(say a typical sample path for an $m$-dimensional Brownian motion).
In general, we allow for $W$ to take values in $\R^m$ and $Y$ in $\R^n$ for
arbitrary integers $n$ and $m$, so we cannot solve \eqref{e:SDE} by
simply setting $Y(t) = Z(W(t))$ with 
$\bar Y$ the solution to the ODE $\dot Z = f(Z)$.
It is a classical result by Young \cite{Young} that the Riemann-Stieltjes
integral $(Y,W) \mapsto \int_0^\cdot Y\,dW$ makes sense as a continuous map 
from $\CC^\alpha \times \CC^\alpha$
into $\CC^\alpha$ if and only if $\alpha > {1\over 2}$. 
As a consequence, ``na\"\i ve'' approaches to 
a pathwise solution to \eref{e:SDE} with $\int_0^\cdot f(Y)\,dW$
interpreted in Young's sense are bound to fail if $W$ has the regularity of Brownian 
motion, since the modulus of continuity $\omega(h)$
of the latter behaves no better than $\omega(h) = \sqrt{2h \log 1/h}$.
A fortiori, it will fail for example if $W$ is a typical sample 
path of fractional Brownian motion with Hurst parameter $H < {1\over 2}$.

In order to break through this barrier, the main idea is to exploit the a priori ``guess'' that solutions to \eref{e:SDE} should ``look like $W$ at small scales'', i.e.\ in order to try to define an integral $\int_0^\cdot Z\,dW$ for any two
elements $Z$ and $W$ of some function space, we only consider those $Z$'s
that could plausibly arise as $Z = f(Y)$ for $Y$ the solution to \eqref{e:SDE}.
More precisely, one would naturally expect the solution $Y$ to satisfy 
\begin{equ}[e:expY]
Y_t = Y_s + Y'_s W_{s,t} + \CO(|t-s|^{2\alpha})\;,
\end{equ}
for some function $Y'$,
where we wrote $W_{s,t}$ as a shorthand for the increment $W_t - W_s$. 
As a matter of fact, one would
expect to have such an expansion with the specific choice $Y' = f(Y)$. 
As a consequence, one would also expect a similar bound to \eqref{e:expY}
for $Z = f(Y)$ with $Z' = f'(Y)$.
In other words, one would make the \textit{a priori} guess that solutions to \eqref{e:SDE}
satisfy $\|Y\1 + Y'\X\|_{2\alpha} < \infty$, provided that we define
$\|\cdot\|_{2\alpha}$ as in \eqref{e:defalphaNorm} with $|\X| = \alpha$ and
$\Gamma_{st}$ as in \eqref{e:gammaalt} (with the caveat that in the 
multidimensional case one should now introduce symbols $\X_i$ for
$i\in\{1,\ldots,m\}$, interpret \eqref{e:gammaalt} as a ``vector'' identity,
take $Y'$ matrix-valued, etc).

It now remains to provide a coherent construction of the integral 
$\int Z\,dW$ for such functions (or rather pairs of functions $(Z,Z')$).
The main idea is to then simply \textit{postulate} the values of the
integrals 
\begin{equ}[e:defXX]
\WW_{s,t} =: \int_s^t W_{s,r}\otimes dW_r\;,
\end{equ}
instead of trying to compute them from $W$ (which is doomed to failure). For a
two-parameter function $\WW$ to be a ``reasonable'' candidate for the right hand side
of \eqref{e:defXX}, one would like it to satisfy Chen's relations
\begin{equ}[e:constr]
\WW_{s,t} - \WW_{s,u} -\WW_{u,t} =  W_{s,u}\otimes W_{u,t}\;,
\end{equ}
since these follow from the requirements that 
$\int_s^t c\,dW_r = c W_{s,t}$ for any $c,s,t \in \R$ and that the sum of two
integrals with the same integrand over adjacent intervals equals the integral over the
union of these intervals.
By simple scaling, it is also natural to impose the analytic 
bound 
\begin{equ}[e:analytic]
|\WW_{s,t}| \lesssim |t-s|^{2\alpha}\;.
\end{equ}
One can then exploit this additional data to give a coherent definition of expressions 
of the type $\int Z\,dW$,
provided that the path $W$ is ``enhanced'' with its iterated integrals $\WW$ and 
$Z$ is a ``controlled path''
of the type \eref{e:expY}. Indeed, it suffices to set
\begin{equ}
\int_0^t Z_s\,dW_s \eqdef \lim_{|\CP| \to 0} \sum_{[s,u] \in \CP} \bigl(Z_s\,W_{s,u}
+ Z'_s\, \WW_{s,u}\bigr)\;,
\end{equ}
where $\CP$ denotes a partition of $[0,t]$ (interpreted as a collection of closed
intervals) and $|\CP|$ denotes the length of the longest interval in $\CP$.
It is a fact that the analytical bounds given above, together with \eqref{e:constr},
guarantee that this limit exists as soon as $\alpha > {1\over 3}$ and that,
if we set 
\begin{equ}
Y_t = \int_0^t Z_s\,dW_s \1 + Z_t\X\;,
\end{equ}
then one has again $\|Y\|_{2\alpha} < \infty$, thus allowing to formulate \eqref{e:SDE}
as a well-behaved fixed point problem.
See for example \cite{Max} or the lecture notes
\cite{Book} for a more detailed exposition.

\section{The basic theory of regularity structures}
\label{sec:defs}

Let us now set up in more detail a general framework in which one can define
``H\"older-type'' spaces as above.
Our first ingredient is a vector space $T$ that contains the
coefficients of our ``Taylor-like'' expansion at each point. In the previous example,
this space was given by the linear span of $\1$ and $\X$. In general, it is natural to 
postulate that $T$ is an arbitrary graded vector space $T = \bigoplus_{\alpha \in A} T_\alpha$,
for some set $A$ of possible ``homogeneities''. For example, in the case of the usual Taylor expansion,
it is natural to take for $A$ the set of natural numbers and to have $T_\ell$ contain the coefficients corresponding
to the derivatives of order $\ell$. In the case of controlled rough paths however, 
we have already seen that it is natural
to take $A = \{0,\alpha\}$, to have again $T_0$ contain the value of the function 
$Y$ at any time $s$,
and to have $T_\alpha$ contain its ``Gubinelli derivative'' $Y'_s$. This reflects the fact that in that case the 
``monomial'' $t \mapsto (\Pi_s \X)(t) = W_{s,t}$ only vanishes at order $\alpha$ near $t = s$, while the usual monomials $t \mapsto (t-s)^\ell$
vanish at integer order $\ell$. In general, we only assume that each $T_\alpha$ is a real Banach space,
although in many examples of interest these spaces will be finite-dimensional.

This however isn't the full algebraic structure describing Taylor-like expansions. Indeed, 
we have already seen that a crucial characteristic of
Taylor expansions is that an expansion around some point $x_0$ can be 
re-expanded around any other point $x_1$, namely simply by writing
\begin{equ}[e:TaylorExp]
(x-x_0)^m = \sum_{k+\ell = m} \binom{m}{k} (x_1 - x_0)^k\cdot (x-x_1)^\ell\;.
\end{equ}
(In the case when $x \in \R^d$, $k$, $\ell$ and $m$ denote multi-indices and $k! = k_1!\ldots k_d!$.) 
In general, we have seen in both of our examples that there are linear maps $\Gamma_{st}$
transforming the coefficients of an expansion around $t$ into the coefficients of the 
same ``polynomial'', but this time expanded around $s$.

What is a natural abstraction of this fact?
In view of the above examples, it is natural to impose that any such ``reexpansion map''
$\Gamma_{st}$ has the property that
if $\tau \in T_\alpha$, then $\Gamma_{st} \tau - \tau \in \bigoplus_{\beta < \alpha} T_\beta
=: T_{<\alpha}$. In other words, 
when reexpanding a homogeneous monomial around a different point, 
the leading order coefficient remains 
the same, but lower order monomials may appear. Furthermore, one should be able to compose reexpansions,
since taking an expansion around $t$, reexpanding it around $s$ and then reexpanding the result
around a third point $r$ should be the same as reexpanding the first expansion around $r$.
In other words, it seems natural that one has the identity $\Gamma_{rs}\Gamma_{st} = \Gamma_{rt}$,
which is indeed the case for the examples we have seen so far.
These considerations can be summarised in the following definition of an 
algebraic structure which we call a \textit{regularity structure}:

\begin{definition}\label{def:reg}
Let $A \subset \R$ be bounded from below and without accumulation point, and let 
$T = \bigoplus_{\alpha \in A} T_\alpha$ be a vector space graded by $A$ such that each $T_\alpha$ is a 
Banach space. Let furthermore $G$ be a group of continuous operators on $T$ such that, for every $\alpha \in A$,
every $\Gamma \in G$, and every $\tau \in T_\alpha$, one has $\Gamma \tau - \tau \in T_{<\alpha}$.
The triple $\TT=(A,T,G)$ is called a \textit{regularity structure} with \textit{model space} $T$ and \textit{structure group} $G$.
\end{definition}

\begin{remark}
In principle, the set $A$ can be infinite. By analogy with the polynomials, it is then natural to 
consider $T$ as the set of all formal series of the form $\sum_{\alpha \in A}\tau_\alpha$, where only
finitely many of the $\tau_\alpha$'s are non-zero, endowed with the topology of
term-wise convergence. This also dovetails nicely with the
particular form of elements in $G$. In practice however we will only ever work with finite subsets
of $A$ so that the precise topology on $T$ does not matter.
\end{remark}

A regularity structure as given by Definition~\ref{def:reg} is just a kind of algebraic ``skeleton'':
it does not require any underlying configuration space and it contains
no information as to which actual functions / distributions its elements
are supposed to describe. It only becomes useful for our purpose when endowed with
a \textit{model}, which is the analytical ``flesh'' associating to 
any $\tau \in T$ and $x_0 \in \R^d$, the actual
``Taylor polynomial based at $x_0$'' represented by $\tau$. In order to 
link the algebraic description to the corresponding analytical objects, 
we want elements $\tau \in T_\alpha$ to
represent functions (or possibly distributions!) that 
``vanish at order $\alpha$'' around the given point $x_0$.

Since we would like to allow elements in $T$ to represent
distributions and not just functions, we cannot evaluate them at points and thus need a suitable 
notion of ``vanishing at order $\alpha$''. We achieve this by controlling the
size of our distributions when tested against test functions that are localised in a small region around the given point $x_0$.
Given a test function $\phi$ on $\R^d$, we write $\phi_x^\lambda$ as a shorthand for
\begin{equ}
\phi_x^\lambda(y) = \lambda^{-d} \phi\bigl(\lambda^{-1}(y-x)\bigr)\;.
\end{equ}
Given $r >0$, we also denote by $\CB_r$ the set of all smooth 
functions $\phi \colon \R^d \to \R$ such that 
$\|\phi\|_{\CC^r} \le 1$, and that are furthermore supported in the unit ball around the origin.
With this notation, and writing furthermore $\CS'(\R^d)$ for the space
of distributions (not necessarily tempered) on $\R^d$,
our definition of a model for a given regularity structure $\TT$ is as follows.

\begin{definition}\label{def:model}
Given a regularity structure $\TT$ and an integer $d \ge 1$, a \textit{model} for $\TT$ on $\R^d$ consists
of maps
\begin{equs}[2]
\Pi \colon \R^d &\to \CL\bigl(T, \CS'(\R^d)\bigr)&\qquad \Gamma\colon \R^d \times \R^d &\to G\\
x &\mapsto \Pi_x&\quad (x,y) &\mapsto \Gamma_{xy}
\end{equs}
such that $\Gamma_{xy}\Gamma_{yz} = \Gamma_{xz}$ and $\Pi_x \Gamma_{xy} = \Pi_y$. Furthermore, given $r > |\inf A|$, 
for any compact set $\K \subset \R^d$ and constant $\gamma > 0$, we assume 
that there exists a constant $C$ such that the bounds
\begin{equ}[e:bounds]
\bigl|\bigl(\Pi_x \tau\bigr)(\phi_x^\lambda)\bigr| \le C \lambda^{|\tau|} \|\tau\|_\alpha\;,\qquad
\|\Gamma_{xy}\tau\|_\beta \le C |x-y|^{\alpha-\beta} \|\tau\|_\alpha\;,
\end{equ}
hold uniformly over $\phi \in \CB_r$, $(x,y) \in \K$, $\lambda \in (0,1]$, 
$\tau \in T_\alpha$ with $\alpha \le \gamma$, and $\beta < \alpha$.
We denote by $\MM$ the space of all models for a given regularity structure.
\end{definition}

\begin{remark}
Given $\tau \in T$, we wrote $\|\tau\|_\alpha$ for the norm of its component in $T_\alpha$.
In other words, if $\tau = \bigoplus_\alpha \tau_\alpha$ with $\tau_\alpha \in T_\alpha$,
then $\|\tau\|_\alpha = \|\tau_\alpha\|_{T_\alpha}$.
\end{remark}

\begin{remark}
The identity $\Pi_x \Gamma_{xy} = \Pi_y$ reflects the fact that $\Gamma_{xy}$ is the linear map that takes
an expansion around $y$ and turns it into an expansion around $x$. The first bound in \eref{e:bounds} states
what we mean precisely when we say that $\tau \in T_\alpha$ represents a term 
of order $\alpha$.
The second bound in \eref{e:bounds} is very natural in view of both \eref{e:basicGamma} and \eref{e:gammaalt}.
It states that when expanding a monomial of order $\alpha$ around a new point at distance $h$ from the old one,
the coefficient appearing in front of lower-order monomials of order $\beta$ is of order at most $h^{\alpha - \beta}$.
\end{remark}

\begin{remark}
In many cases of interest, it is natural to scale the different directions of $\R^d$ in a different way. 
This is the case for example when using the theory of regularity structures to build solution theories for
parabolic stochastic PDEs, in which case the time direction ``counts double''.  
To deal with such  a situation, one can introduce a scaling $\s$ of $\R^d$, which is just a collection of
$d$ mutually prime strictly positive integers and one defines $\phi_x^\lambda$ in such a  way that the $i$th
direction is scaled by $\lambda^{\s_i}$. In this case,  the Euclidean distance between two points should
be replaced everywhere by the corresponding scaled distance $|x|_\s = \sum_i |x_i|^{1/\s_i}$.
See also \cite{Regular} for more details.
\end{remark}

With these definitions at hand, it is then natural to define an equivalent in this context of the space
of $\gamma$-H\"older continuous functions in the following way, which is the natural
generalisation of \eqref{e:defalphaNorm}.

\begin{definition}
Given a regularity structure $\TT$ equipped with a model $(\Pi,\Gamma)$ over $\R^d$, 
the space $\D^\gamma = \D^\gamma(\TT,\Gamma)$ is given by the set of functions $f\colon \R^d \to \bigoplus_{\alpha < \gamma} T_\alpha$
such that, for every compact set $\K$ and every $\alpha < \gamma$, the exists a constant $C$ with
\begin{equ}[e:defHolder]
\|f(x) - \Gamma_{xy} f(y)\|_{\alpha} \le C |x-y|^{\gamma-\alpha}
\end{equ}
uniformly over $x,y \in \K$.
\end{definition}

\begin{remark}
Note that, given $\TT$ and a model $(\Pi,\Gamma) \in \MM$, the corresponding space
$\D^\gamma$ is a Fr\'echet space. However, as we have already seen in
Remark~\ref{rem:unique} above, this space does in general depend crucially on the choice
of model. 
We thus have a ``total space'' $\MM \ltimes \CD^\gamma$ containing all triples of the form
$(\Pi,\Gamma,F)$ with $F \in \CD^\gamma$ based on the model $(\Pi,\Gamma)$.
This space is no longer a linear space, but it still comes with a natural topology:
the distance between $(\Pi, \Gamma, f)$ and $(\bar \Pi, \bar \Gamma, \bar f)$ is given by
the smallest constant $\rho$ such that
\begin{equs}
\|f(x) - \bar f(x) - \Gamma_{xy} f(y) + \bar \Gamma_{xy} \bar f(y)\|_{\alpha} &\le \rho |x-y|^{\gamma-\alpha}\;,\\
\bigl|\bigl(\Pi_x \tau - \bar \Pi_x \tau\bigr)(\phi_x^\lambda)\bigr| &\le \rho \lambda^{\alpha} \|\tau\|\;,\\
\|\Gamma_{xy}\tau - \bar \Gamma_{xy}\tau\|_\beta &\le \rho |x-y|^{\alpha-\beta} \|\tau\|\;,
\end{equs}
uniformly for $x,y$ in some compact set.
\end{remark}

At this point, we should pause and ask the following question: given 
$f \in \D^\gamma$ for a given regularity structure and model, what is the actual
function / distribution represented by $f$? Recall that we have seen previously
in \eqref{e:RFsimple} that it should represent $\CR f$ given by
$(\CR f)(x) = \bigl(\Pi_x f(x)\bigr)(x)$. However, this definition now no longer 
makes sense since $\Pi_x f(x)$ is a distribution in general and can therefore
not be evaluated at $x$!
The most fundamental result in the theory of regularity structures then states that,
given $f \in \D^\gamma$ with $\gamma > 0$, there exists a \textit{unique} 
distribution $\CR f$ on $\R^d$ such that, for every $x \in \R^d$, $\CR f$ 
``looks like $\Pi_x f(x)$ near $x$''. More precisely, one has

\begin{theorem}\label{theo:reconstruction}
Let $\TT$ be a regularity structure as above and let $(\Pi,\Gamma) \in \MM$ be
a model for $\TT$ on $\R^d$.
Then, there exists a unique linear map $\CR\colon \D^\gamma \to \CS'(\R^d)$ such that
\begin{equ}[e:boundRf]
\bigl|\bigl(\CR f - \Pi_x f(x)\bigr)(\phi_x^\lambda)\bigr| \lesssim \lambda^\gamma\;,
\end{equ}
uniformly over $\phi \in \CB_r$ and $\lambda$ as before, and locally uniformly in $x$.
\end{theorem}

We do not provide a proof of this result in these notes, but a concise version can be found
in \cite{Brazil}. It relies crucially on the fact that
our assumptions guarantee that the distributions $\zeta_x = \Pi_x f(x)$ vary slowly with $x$:
\begin{equ}[e:mainBound]
\bigl(\zeta_x - \zeta_y\bigr)(\phi_x^\lambda) = \Pi_x \bigl(f(x) - \Gamma_{xy} f(y)\bigr)(\phi_x^\lambda)
\lesssim \sum_{\alpha < \gamma} \lambda^\alpha \, |x-y|^{\gamma - \alpha}
\lesssim \lambda^\gamma\;,
\end{equ}
for $|y-z| \lesssim \lambda$.
In particular, the value of $\zeta_y(\phi_x^\lambda)$ varies by at most of the order $\lambda^\gamma$
as $y$ varies over the support of $\phi_x^\lambda$.
In fact, one can show more generally that given a function $x \mapsto \zeta_x$ which is 
(locally) uniformly bounded in $\CC^\alpha$ for some $\alpha < 0$ and such that 
\eqref{e:mainBound} holds for some $\gamma > 0$, 
then there exists a unique distribution $\zeta \in \CC^\alpha$ such that
$\bigl|\bigl(\zeta - \zeta_x\bigr)(\phi_x^\lambda)\bigr| \lesssim \lambda^\gamma$, locally
uniformly in $x$, see \cite[Prop.~3.25]{Regular}.

\begin{remark}
The spaces $\CD^\gamma$ are natural generalisations of H\"older spaces. One may wonder
whether there are also natural generalisations of other classical function spaces and whether
Theorem~\ref{theo:reconstruction} still applies. In \cite{Cyril}, the authors
show that approximations to the multiplicative stochastic heat equation converge
even when the initial condition is taken to be a Dirac mass. For this, it appears that
weighted H\"older-type spaces are not suitable spaces to work with.
Instead, the authors work with a generalisation of a scale of inhomogeneous Besov spaces.
It seems likely that most classical function spaces admit generalisations to the present
context. The rule of thumb regarding the reconstruction theorem is that one typically
has existence and uniqueness of the reconstruction operator as soon as the ``regularity index''
of the corresponding classical function space is positive. 
\end{remark}

\begin{remark}
Since the spaces $\D^\gamma$ depend on the choice of model $(\Pi,\Gamma)$,
the reconstruction operator $\CR$ itself also depends on that choice.
Since the aim of the theory is to provide robust approximation procedures, 
we should really view $\CR$ as a map from the total space $\MM \ltimes \CD^\gamma$ 
into $\CS'(\R^d)$. It turns out that even if viewed it in this way, 
$\CR$ is still a continuous map, although it is no longer linear.
The fact that this stronger continuity property also holds is crucial when 
showing that sequences of
solutions to mollified equations all converge to the same limiting object.
\end{remark}

\begin{remark}\label{rem:continuousModel}
In the particular case where $\Pi_x \tau$ happens to be a continuous function for every $\tau \in T$
(and every $x \in \R^d$), $\CR f$ is also a continuous function and one has the identity
\begin{equ}[e:formulaRf]
\bigl(\CR f\bigr)(x) = \bigl(\Pi_x f(x)\bigr)(x)\;.
\end{equ}
This can easily be seen by noting that in this case, if $\phi$ integrates to $1$, then as 
$\lambda \to 0$, the quantity $\bigl(\Pi_x f(x)\bigr)(\phi_x^\lambda)$ appearing in 
\eqref{e:boundRf} converges to $\bigl(\Pi_x f(x)\bigr)(x)$. Since the right hand side converges
to $0$, this implies that $\bigl(\CR f\bigr)(\phi_x^\lambda)$ converges to
$\bigl(\Pi_x f(x)\bigr)(x)$, whence \eqref{e:formulaRf} follows.
\end{remark}

\subsection{A classical result from harmonic analysis}
\label{sec:classical}

It is a classical result in harmonic analysis \cite{BookChemin} that the product
extends naturally to $\CC^{-\alpha} \times \CC^\beta$ into $\CS'(\R^d)$ if and only if 
$\beta > \alpha$. The reconstruction theorem yields a straightforward proof of the ``if''
part of this result:

\begin{theorem}\label{thm:barrier}
There is a continuous bilinear map $B \colon \CC^{-\alpha} \times \CC^\beta \to \CS'(\R^d)$
such that $B(f,g) = fg$ for any two continuous functions $f$ and $g$.
\end{theorem}

\begin{proof}
Assume from now on that $\xi \in \CC^{-\alpha}$ for some $\alpha > 0$ and that $f \in \CC^\beta$
for some $\beta > \alpha$. 
We then build a regularity structure $\TT$ in the following way. For the set $A$,
we take $A = \N \cup (\N-\alpha)$ and
for $T$, we set $T = V \oplus W$,
where each one of the spaces $V$ and $W$ is a copy of the polynomial structure in $d$ commuting variables
described in Section~\ref{sec:Holder}.
We also choose the structure group $G$ as in the polynomial structure, 
acting simultaneously on each of the two instances.

As before, we denote by $\X^k$ the canonical basis vectors in $V$. We also use the suggestive notation
``$\sXi \X^k$'' for the corresponding basis vector in $W$, but we postulate that $|\sXi \X^k| = \alpha + |k|$ 
rather than $|\sXi \X^k| = |k|$ as would usually be the case.
Given any distribution $\xi \in \CC^{-\alpha}$, we then define a model $(\Pi^\xi,\Gamma)$,
where $\Gamma$ is as in the canonical polynomial model, while $\Pi^\xi$ acts as
\begin{equ}
\bigl(\Pi^\xi_x  \X^k\bigr)(y) = (y-x)^k \;,\qquad
\bigl(\Pi^\xi_x  \sXi \X^k\bigr)(y) = (y-x)^k \xi(y)\;,
\end{equ}
with the obvious abuse of notation in the second expression. 
It is then straightforward to verify that 
$\Pi_y = \Pi_x\circ \Gamma_{xy}$ and that the relevant analytical bounds are satisfied, so that this is indeed a model.

Denote now by $\CR^\xi$ the reconstruction map associated to the model $(\Pi^\xi,\Gamma)$ and, for $f \in \CC^\beta$,
denote by $F$ the element in $\D^\beta$ given by the local Taylor expansion of $f$ of order $\beta$ at each point.
Note that even though the space $\D^\beta$ does in principle depend on the choice of model, in our situation
$F \in \D^\beta$ for any choice of $\xi$ since $\Gamma_{xy}$ is independent of $\xi$. 
It follows immediately from the definitions that the map $x \mapsto \sXi F(x)$ belongs to $\D^{\beta - \alpha}$
so that, provided that $\beta > \alpha$, one can apply the reconstruction operator to it. This suggests that the
multiplication operator we are looking for can be defined as
\begin{equ}
B(f,\xi) = \CR^\xi \bigl(\sXi F\bigr)\;.
\end{equ}
By Theorem~\ref{theo:reconstruction}, 
this is a jointly continuous map from $\CC^\beta \times \CC^{-\alpha}$ into $\CS'(\R^d)$, provided that 
$\beta > \alpha$.
If $\xi$ happens to be a smooth function, then it follows immediately from Remark~\ref{rem:continuousModel} that 
$B(f,\xi) = f(x)\xi(x)$, so that $B$ is indeed the requested continuous extension of the usual product.
\end{proof}

\begin{remark}
As a consequence of \eref{e:boundRf}, it follows that $B \colon \CC^{-\alpha} \times \CC^\beta \to \CC^{-\alpha}$.
\end{remark}

\subsection{Products}
\label{sec:prod}

One of the main purposes of the theory presented here is to give a robust way to multiply distributions
(or functions with distributions) that goes beyond the barrier illustrated by Theorem~\ref{thm:barrier}.
Provided that our functions / distributions are represented as elements in $\D^\gamma$ for some
model and regularity structure, we can multiply their ``Taylor expansions'' pointwise, provided that
we give ourselves a table of multiplication on $T$.

It is natural to consider products with the following properties. Here, given a regularity structure,
we say that a subspace
$V \subset T$ is a \textit{sector} if it is invariant under the action of the structure group $G$
and if it can furthermore be written as $V = \bigoplus_{\alpha \in A} V_\alpha$
with $V_\alpha \subset T_\alpha$.

\begin{definition}
Given a regularity structure $(T,A,G)$ and two sectors $V, \bar V \subset T$, a \textit{product}
on $(V,\bar V)$ is a bilinear map $\star \colon V \times \bar V \to T$ such that, for any $\tau \in V_\alpha$
and $\bar \tau \in \bar V_\beta$, one has $\tau \star \bar \tau \in T_{\alpha + \beta}$ and such that,
for any element $\Gamma \in G$, one has $\Gamma(\tau \star \bar \tau) = \Gamma \tau \star \Gamma \bar \tau$. 
\end{definition}

\begin{remark}
The condition that homogeneities add up under multiplication is very natural bearing in mind the case
of the polynomial regularity structure. The second condition is also very natural since it merely states that
if one reexpands the product of two ``polynomials'' around a different point, one should obtain the same result
as if one reexpands each factor first and then multiplies them together.
\end{remark}

Given such a product, we can ask ourselves when the pointwise product of an element $\D^{\gamma_1}$
with an element in $\D^{\gamma_2}$ again belongs to some $\D^\gamma$. In order to answer this question, 
we introduce the notation $\D_\alpha^\gamma$ to denote those elements $f \in \D^\gamma$ such that 
furthermore 
$f(x) \in T_{\ge\alpha}$
for every $x$. With this notation at hand, 
it is not too difficult to verify that one has the following result:

\begin{theorem}\label{theo:mult}
Let $f_1 \in \D^{\gamma_1}_{\alpha_1}(V)$, $f_2 \in \D^{\gamma_2}_{\alpha_2}(\bar V)$, and let $\star$
be a product on $(V,\bar V)$. Then, the function $f$ given by $f(x) = f_1(x) \star f_2(x)$ belongs to $\D_{\alpha}^\gamma$
with
\begin{equ}[e:formulaGamma]
\alpha = \alpha_1 +\alpha_2\;,\qquad \gamma = (\gamma_1 + \alpha_2)\wedge (\gamma_2 + \alpha_1)\;.
\end{equ} 
\end{theorem}

The proof of this result is straightforward and can be found in \cite{Regular,Brazil}.
It is clear that the formula \eref{e:formulaGamma} for $\gamma$ is optimal in general as can be seen from
the following two ``reality checks''. First, consider the case of the polynomial model
and take $f_i \in \CC^{\gamma_i}$. In this case, the truncated Taylor series $F_i$ for $f_i$ belong to $\D_0^{\gamma_i}$. 
It is clear that in this case, the product cannot be expected to have better regularity than $\gamma_1 \wedge \gamma_2$ in
general, which is indeed what \eref{e:formulaGamma} states.
The second reality check comes from the example of Section~\ref{sec:classical}.
In this case, one has $F \in \D_0^\beta$, while the constant function $x \mapsto \Xi$ belongs
to $\D_{-\alpha}^\infty$ so that, according to \eref{e:formulaGamma}, one expects their product
to belong to $\D_{-\alpha}^{\beta-\alpha}$, which is indeed the case.

\begin{remark}
In order to obtain robust approximation results, one would like to obtain
bounds on the distance between $f_1 \star f_2$ and $\bar f_1 \star \bar f_2$ in
cases where the $f_i$ belong to spaces $\CD^{\gamma_i}$ based on some
model $(\Pi,\Gamma) \in \MM$, while the $\bar f_i$ belong to spaces $\CD^{\gamma_i}$ 
based on a different model $(\bar \Pi,\bar \Gamma) \in \MM$. This can also be obtained,
but makes the proof slightly more lengthy.
\end{remark}

\begin{remark}
Even 
if both $\CR f_1$ and $\CR f_2$ happens to be continuous functions, this does
\textit{not} in general imply that $\CR(f_1 \star f_2)(x) = (\CR f_1)(x) \, (\CR f_2)(x)$!

For example, fix $\kappa < 0$ and consider the regularity structure given by $A = (-2\kappa,-\kappa, 0)$,
with each $T_\alpha$ being a copy of $\R$ given by $T_{-n\kappa} = \scal{\sXi^n}$. 
We furthermore take for $G$ the trivial group. This regularity structure comes with an obvious product
by setting $\sXi^m \star \sXi^n = \sXi^{m+n}$ provided that $m+n \le 2$.
Then a perfectly valid model is
\begin{equ}
\bigl(\Pi_x \sXi^0\bigr)(y) = 1\;,\quad 
\bigl(\Pi_x \sXi^1\bigr)(y) = 0\;,\quad \label{e:modelNonStandard}
\bigl(\Pi_x \sXi^2\bigr)(y) = c\;,
\end{equ}
for which, setting $F(x) = f(x)\sXi^0 + g(x) \sXi^1$, one has
$(\CR F)(x) = f(x)$, but $(\CR F^2)(x) = f^2(x) + c g^2(x)$.
This flexibility is crucial since the ``naive'' definition of the product is usually 
broken by renormalisation, as we will see in Section~\ref{sec:renorm} below.
\end{remark}

\subsection{Schauder estimates and admissible models}
\label{sec:Schauder}

One of the reasons why the theory of regularity structures is very successful at providing detailed
descriptions of the small-scale features of solutions to semilinear (S)PDEs is that it comes with
very sharp Schauder estimates. Recall that the classical Schauder estimates state that if $K\colon \R^d \to \R$
is a kernel that is smooth everywhere, except for a singularity at the origin that is (approximately) homogeneous
of degree $\beta - d$ for some $\beta > 0$, then the operator $f \mapsto K * f$ maps $\CC^\alpha$ into $\CC^{\alpha +\beta}$
for every $\alpha \in \R$, except for those values for which $\alpha + \beta \in \N$. (See for example \cite{MR1459795}.)

It turns out that similar Schauder estimates hold in the context of general regularity structures in the sense
that it is in general possible to build an operator $\CK \colon \D^\gamma \to \D^{\gamma+\beta}$ with the
property that $\CR \CK f = K * \CR f$. Of course, such a statement can only be true if our regularity structure
contains not only the objects necessary to describe $\CR f$ up to order $\gamma$, but also those required
to describe $K * \CR f$ up to order $\gamma + \beta$. What are these objects? At this stage, it might be useful
to reflect on the effect of the convolution of a singular function (or distribution) with $K$. 

Let us assume for a moment that $f$ is also smooth everywhere, except at some point $x_0$. It is then 
straightforward to convince ourselves that $K * f$ is also smooth everywhere, except at $x_0$. Indeed,
for any $\delta > 0$, we can write $K = K_\delta + K_\delta^c$, where $K_\delta$ is supported in a ball of radius
$\delta$ around $0$ and $K_\delta^c$ is a smooth function. Similarly, we can decompose $f$ as $f = f_\delta + f_\delta^c$,
where $f_\delta$ is supported in  a $\delta$-ball around $x_0$ and $f_\delta^c$ is smooth. Since the convolution of a 
smooth function with an arbitrary distribution is smooth, it follows that the only non-smooth component of $K * f$
is given by $K_\delta * f_\delta$, which is supported in a ball of radius $2\delta$ around $x_0$. Since $\delta$ was
arbitrary, the statement follows. By linearity, this strongly suggests that the local structure of the singularities of $K*f$ can be described completely by only using knowledge on the local structure of the singularities of $f$.
It also suggests that the ``singular part'' of the operator $\CK$ should be local, with the non-local
parts of $\CK$ only contributing to the ``regular part''.

This discussion suggests that we certainly need the following ingredients to build an operator $\CK$ with the
desired properties:
\begin{claim}
\item The canonical polynomial structure should be part of our regularity structure in order to
be able to describe the ``regular parts''.
\item We should be given an ``abstract integration operator'' $\CI$ on $T$ which describes how the 
``singular parts'' of $\CR f$ transform under convolution by $K$.
\item We should restrict ourselves to models which are ``compatible'' with the action of $\CI$ in the sense
that the behaviour of $\Pi_x \CI \tau$ should relate in a suitable way to the behaviour of $K * \Pi_x \tau$ near $x$.
\end{claim}
One way to implement these ingredients is to assume first that our model space $T$ 
contains abstract polynomials in the following sense.

\begin{assumption}\label{ass:poly}
There exists a sector $\bar T \subset T$ isomorphic to the space of abstract polynomials
in $d$ commuting variables. In other words, $\bar T_\alpha \neq 0$ if and only if $\alpha \in \N$, 
and one can find basis vectors $\X^k$ of $T_{|k|}$ such that every element $\Gamma \in G$ acts
on $\bar T$ by $\Gamma \X^k = (\X-h)^k$ for some $h \in \R^d$.
\end{assumption}

Furthermore, we assume that there exists an abstract integration operator $\CI$ with the following properties.

\begin{assumption}\label{ass:int}
There exists a linear map $\CI \colon T \to T$ such that $\CI T_\alpha \subset T_{\alpha + \beta}$,
such that $\CI \bar T = 0$, and such that, for every $\Gamma \in G$ and $\tau \in T$, one has
\begin{equ}[e:propI]
\Gamma \CI \tau - \CI \Gamma \tau \in \bar T\;.
\end{equ}
\end{assumption}

Finally, we want to consider models that are compatible with this structure for a given kernel $K$.
For this, we first make precise what we mean exactly when we said that $K$ is approximately homogeneous of 
degree $\beta - d$. 

\begin{assumption}\label{ass:kernel}
One can write $K = \sum_{n \ge 0} K_n$ where each of the kernels $K_n\colon \R^d \to \R$ is
smooth and compactly supported in a ball of radius $2^{-n}$ around the origin. Furthermore, we assume that
for every multiindex $k$, one has a constant $C$ such that the bound
\begin{equ}[e:boundKn]
\sup_x |D^k K_n(x)| \le C 2^{n(d-\beta + |k|)}\;,
\end{equ}
holds uniformly in $n$. Finally, we assume that $\int K_n(x) P(x)\,dx = 0$ for every polynomial $P$
of degree at most $N$, for some sufficiently large value of $N$.
\end{assumption}

\begin{remark}
It turns out that in order to define the operator $\CK$ on $\D^\gamma$, we will need $K$ to
annihilate polynomials of degree $N$ for some $N \ge \gamma + \beta$.
\end{remark}

\begin{remark}\label{rem:decompKernel}
The last assumption may appear to be extremely stringent at first sight. In practice, this turns out
not to be a problem at all. Say for example that we want to define an operator that represents convolution
with $P$, the fundamental solution to the heat equation. Then, $P$ can be decomposed into a sum of terms
satisfying the bound \eref{e:boundKn} with $\beta = 2$,
but it does of course not annihilate generic polynomials and it is not supported in the ball of radius $1$.

However, for any fixed value of $N>0$, it is straightforward to decompose $P$ as $P = K + \hat K$, where the
kernel $K$ is compactly supported and satisfies all of the properties mentioned above, and the kernel $\hat K$ is
smooth. Lifting the convolution with $\hat K$ to an operator from $\D^\gamma \to \D^{\gamma + \beta}$
(actually to $\D^{\bar \gamma}$ for any $\bar \gamma > 0$) is straightforward, so that we have reduced our
problem to that of constructing an operator describing the convolution by $K$.
\end{remark}

Given such a kernel $K$, we can now make precise what we meant earlier when we said that the models
under consideration should be compatible with the kernel $K$.

\begin{definition}
Given a kernel $K$ as in Assumption~\ref{ass:kernel} and a regularity structure $\TT$ satisfying
Assumptions~\ref{ass:poly} and \ref{ass:int}, we say that a model $(\Pi,\Gamma)$ is \textit{admissible} if 
the identities
\begin{equ}[e:defAdmissible]
\bigl(\Pi_x \X^k\bigr)(y) = (y-x)^k\;,\qquad
\Pi_x \CI \tau = K * \Pi_x \tau - \Pi_x \CJ(x) \tau\;,
\end{equ}
hold for every $\tau \in T$ in the domain of $\CI$. 
Here, $\CJ(x) \colon T \to \bar T$ is the linear map given on 
homogeneous elements by
\begin{equ}[e:defJ]
\CJ(x)\tau = \sum_{|k| < |\tau| + \beta} {\X^k\over k!} \int D^{(k)} K(x-y)\,\bigl(\Pi_x \tau\bigr)(dy)\;.
\end{equ}
\end{definition}

\begin{remark}\label{rem:welldef}
While $K * \xi$ is well-defined for any distribution $\xi$, it is not so clear \textit{a priori} whether
the operator $\CJ(x)$ given in \eref{e:defJ} is also well-defined. It turns out that the axioms of a model
do ensure that this is the case. The correct way of interpreting \eref{e:defJ} is by  
\begin{equ}
\CJ(x)\tau = \sum_{|k| < |\tau| + \beta} \sum_{n \ge 0} {\X^k\over k!} \bigl(\Pi_x \tau\bigr)\bigl(D^{(k)} K_n(x-\cdot)\bigr)\;.
\end{equ}
The scaling properties of the $K_n$ ensure that $2^{(\beta - |k|)n} D^{(k)} K_n(x-\cdot)$ is 
of the form $c \phi_x^\lambda$ with $\phi \in \CB_r$ (for arbitrary $r$), $c \approx 1$ and
$\lambda \approx 2^{-n}$. As a consequence of \eqref{e:bounds}, one then has 
\begin{equ}
\bigl|\bigl(\Pi_x \tau\bigr)\bigl(D^{(k)} K_n(x-\cdot)\bigr)\bigr| \lesssim 2^{(|k| - \beta - |\tau|)n}\;,
\end{equ}
so that this expression is indeed summable as long as $|k| < |\tau| + \beta$. 
\end{remark}

\begin{remark}
The above definition of an admissible model dovetails very nicely 
with our axioms defining a general model. Indeed, starting from \textit{any} regularity structure $\TT$, \textit{any}
model $(\Pi, \Gamma)$ for $\TT$, and a kernel
$K$ satisfying Assumption~\ref{ass:kernel}, it is usually possible to build a larger regularity structure $\hat \TT$
containing $\TT$ (in the ``obvious'' sense that $T \subset \hat T$ and the action of $\hat G$ on $T$ is compatible with 
that of $G$) and endowed with an abstract integration map $\CI$, as well as an admissible model $(\hat \Pi, \hat \Gamma)$ on 
$\hat \TT$ which reduces to $(\Pi,\Gamma)$ when restricted to $T$. 
See \cite{Regular} for more details.

The only exception to this rule arises
when the original structure $T$ contains some homogeneous element $\tau$ which does not represent a polynomial
and which is such that $|\tau| + \beta \in \N$. Since the bounds appearing both in the definition of a model and
in Assumption~\ref{ass:kernel} are only upper bounds, it is in practice easy to exclude such a situation by slightly
tweaking the definition of either the exponent $\beta$ or of the original regularity structure $\TT$. 
\end{remark}

With all of these definitions in place, we can finally build the operator $\CK \colon \D^\gamma\to \D^{\gamma + \beta}$
announced at the beginning of this section. Recalling the definition of $\CJ$ from \eref{e:defJ}, we set
\begin{equ}[e:defKf]
\bigl(\CK f\bigr)(x) = \CI f(x) + \CJ(x) f(x) + \bigl(\CN f\bigr)(x)\;,
\end{equ}
where the operator $\CN$ is given by
\begin{equ}[e:defN]
\bigl(\CN f\bigr)(x) = 
\sum_{|k| < \gamma + \beta} {\X^k\over k!} \int D^{(k)} K(x-y)\,\bigl(\CR f - \Pi_x f(x)\bigr)(dy)\;.
\end{equ}
Note first that thanks to the reconstruction theorem, it is possible to verify that the right hand side of 
\eref{e:defN} does indeed make sense for every $f \in \D^\gamma$ in virtually the same way as in Remark~\ref{rem:welldef}.
One has:

\begin{theorem}\label{theo:Schauder}
Let $K$ be a kernel satisfying Assumption~\ref{ass:kernel}, let $\TT = (A,T,G)$ be a regularity structure
satisfying Assumptions~\ref{ass:poly} and \ref{ass:int}, and let $(\Pi,\Gamma)$ be an admissible model for $\TT$.
Then, for every $f \in \D^\gamma$ with $\gamma \in (0,N-\beta)$ and $\gamma + \beta \not \in \N$, 
the function $\CK f$ defined in \eref{e:defKf} belongs to $\D^{\gamma + \beta}$ and
satisfies $\CR \CK f = K * \CR f$.
\end{theorem}

\begin{proof}
The complete proof of this result can be found in \cite{Regular} and will not be given here. Let us simply show that one has
indeed $\CR \CK f = K * \CR f$ in the particular case when our model consists of continuous functions so
that Remark~\ref{rem:continuousModel} applies. In this case, one has
\begin{equ}
\bigl(\CR \CK f\bigr)(x) = \bigl(\Pi_x (\CI f(x) + \CJ(x) f(x))\bigr)(x) + \bigl(\Pi_x\bigl(\CN f\bigr)(x)\bigr)(x)\;.
\end{equ}
As a consequence of \eref{e:defAdmissible}, the first term appearing in the right hand side of this expression is given by
\begin{equ}
\bigl(\Pi_x (\CI f(x) + \CJ(x) f(x))\bigr)(x) = \bigl(K * \Pi_x f(x)\bigr)(x)\;.
\end{equ}
On the other hand, the only term contributing to the second term is the one with $k = 0$ (which is always present
since $\gamma > 0$ by assumption) which then yields
\begin{equ}
\bigl(\Pi_x\bigl(\CN f\bigr)(x)\bigr)(x) = \int K(x-y)\,\bigl(\CR f - \Pi_x f(x)\bigr)(dy)\;.
\end{equ}
Adding both of these terms, we see that the expression $\bigl(K * \Pi_x f(x)\bigr)(x)$ cancels, leaving us with the desired result.
\end{proof}

\section{Application of the theory to the dynamical \texorpdfstring{$\Phi^4_3$}{Phi43} model}
\label{sec:app}

We now sketch how the theory of regularity structures can be used to obtain
the kind of convergence result stated in Theorem~\ref{theo:construction}.
We will only focus on the dynamical $\Phi^4_3$ model and ignore the corresponding result for the
KPZ equation. 
First of all, we note that while our solution $\Phi$ will be 
a space-time distribution (or rather an element of $\D^\gamma$ for some regularity structure with 
a model over $\R^4$), the ``time'' direction has a different scaling behaviour from the three ``space''
directions. As a consequence, it turns out to be effective to slightly change our definition of
``localised test functions'' by setting
\begin{equ}
\phi_{(s,x)}^\lambda (t,y) = \lambda^{-5} \phi\bigl(\lambda^{-2}(t-s), \lambda^{-1}(y-x)\bigr)\;.
\end{equ}
Accordingly, the ``effective dimension'' of our space-time is actually $5$, rather than $4$.
The theory presented above extends \textit{mutatis mutandis} to this setting. (Note in particular
that when considering the degree of a regular monomial, powers of the time variable should now be 
counted double.) With this way of measuring regularity, space-time white noise
belongs to $\CC^{-\alpha}$ for every $\alpha > {5\over 2}$.

\subsection{Construction of the associated regularity structure}

Our first step is to build a regularity structure that is sufficiently large to allow to 
reformulate \eref{e:Phi4} as a fixed point in $\D^\gamma$ for some $\gamma > 0$. 
Denoting by $P$ the heat kernel, we 
can write the solution to \eref{e:Phi4} with initial condition $\Phi_0$ as
\begin{equ}[e:intPhi4]
\Phi = P * \one_{t > 0}\bigl(\xi - \Phi^3\bigr) + P \Phi_0\;,
\end{equ}
where $*$ denotes space-time convolution and where we denote by $P \Phi_0$ 
the solution to the heat equation with initial condition $\Phi_0$. 
Here, $\one_{t > 0}$ denotes the indicator function
of the set $\{(t,x)\,:\, t>0\}$. In order to have a chance of fitting 
this into the framework described above, we first decompose the heat kernel $P$ as
$P = K + \hat K$,
where $K$ satisfies all of the assumptions of Section~\ref{sec:Schauder} with $\beta = 2$ and
the remainder $\hat K$ is smooth, see Remark~\ref{rem:decompKernel}.
For any regularity structure containing the usual
Taylor polynomials and equipped with an admissible model, is straightforward to associate to $\hat K$ an
operator $\hat \CK \colon \D^\gamma \to \D^{\gamma + 2}$ via
\begin{equ}
\bigl(\hat \CK f\bigr)(z) = \sum_{|k| < \gamma+2} {\X^k \over k!} \bigl(D^{(k)}\hat K * \CR f\bigr)(z)\;,
\end{equ}
where $z$ denotes a space-time point.
Similarly, the harmonic extension of $\Phi_0$ can be lifted to an element in $\D^\gamma$
(for any fixed $\gamma > 0$) which we denote again by $P \Phi_0$ by
considering its Taylor expansion around every space-time point. At this stage, we note that we actually
cheated a little: while $P \Phi_0$ is smooth in $\{(t,x)\,:\, t > 0, x\in \T^3\}$ and vanishes when $t < 0$,
it is of course singular on the time-$0$ hyperplane $\{(0,x)\,:\, x\in \T^3\}$. Similarly, we have the
problem that the function $\one_{t > 0}$ does not belong to any $\CD^\gamma$.
Both of these problems can be cured at once
by introducing weighted versions of the spaces $\D^\gamma$ allowing for singularities on
a given hyperplane. A precise definition of these spaces and their behaviour under multiplication and
the action of the integral operator $\CK$ can be found in \cite{Regular}. For the purpose of the informal
discussion given here, we will simply ignore this problem.

As in Section~\ref{sec:classical}, we furthermore introduce a 
new symbol $\sXi$ which will be used to represent the noise $\xi$.
This suggests that the formulation of \eref{e:Phi4} in a suitable (weighted) space
$\CD^\gamma$ should be of the form
\begin{equ}[e:abstractFull]
\Phi = \CP \one_{t>0}\bigl(\sXi - \Phi^3\bigr) + P \Phi_0\;,
\end{equ}
where we set $\CP = \CK + \hat \CK$. In view of \eref{e:defKf}, for $t>0$ this equation is of the type
\begin{equ}[e:abstract]
\Phi = \CI \bigl(\sXi - \Phi^3\bigr) + (\ldots)\;,
\end{equ}
where the terms $(\ldots)$ consist of functions that take values in the subspace 
$\bar T\subset T$ spanned by the regular Taylor monomials $\X^k$. 
In order to build a regularity structure in which \eref{e:abstract} can be formulated, 
it is natural to start with the structure given by abstract polynomials (again with the parabolic scaling
which causes the abstract ``time'' variable to have homogeneity $2$ rather than $1$), and to add
a symbol $\sXi$ to it which we postulate to have homogeneity $-{5\over 2}^{-}$, where we denote by $\alpha^-$ an
exponent strictly smaller than, but arbitrarily close to, the value $\alpha$.

We then simply add to $T$ all of the formal expressions that an application of the right hand side of \eref{e:abstract}
can generate for the description of $\Phi$, $\Phi^2$, and $\Phi^3$. The homogeneity of a given expression is 
completely determined by the rules $|\CI \tau| = |\tau| + 2$ and $|\tau \bar \tau| = |\tau| + |\tau|$.
More precisely, we consider a collection $\CU$ of formal expressions which is the
smallest collection containing $\X^k$ and $\CI(\sXi)$, and such that 
\begin{equ}[e:induction]
\tau_1,\tau_2,\tau_3 \in \CU \quad\Rightarrow\quad \CI(\tau_1\tau_2\tau_3) \in \CU\;,
\end{equ}
where it is understood that $\CI(X^k) = 0$ for every multiindex $k$.
We then set 
\begin{equ}[e:defCW]
\CW = \{\sXi\} \cup \{\tau_1\tau_2\tau_3\,:\, \tau_i \in \CU\}\;,
\end{equ}
and we define our space $T$ as the set of all linear combinations of elements in $\CW$. 
(Note that since $\1 \in \CU$, one does in particular have $\CU \subset \CW$.)
Naturally, $\CT_\alpha$ consists of those linear combinations that only involve
elements in $\CW$ that are of homogeneity $\alpha$. 
It is not too difficult to convince oneself that, 
for every $\alpha \in \R$, $\CW$ contains only
finitely many elements of homogeneity less than $\alpha$, so that each $\CT_\alpha$
is finite-dimensional.

In order to simplify expressions later, we will use the following shorthand graphical 
notation for elements of $\CW$. For $\sXi$, we simply draw a dot.
The integration map is then represented by a downfacing line and the multiplication of 
symbols is obtained by joining them at the root. For example, we have
\begin{equ}
\CI(\sXi) = \<1>\;,\quad
\CI(\sXi)^3 = \<3>\;,\quad
\CI(\sXi)\CI(\CI(\sXi)^3) = \<31>\;.
\end{equ}
Symbols containing factors of $\X$ have no particular graphical representation, so we
will for example write $\X_i \CI(\sXi)^2 = \X_i\<2>$. With this notation, the space $T$ is 
given by
\begin{equs}
T &= \langle \sXi, \<3>, \<2>, \<32>,\<1>, \<31>, \<22>, \X_i \<2>, \1, \<30>, \<21>, \ldots\rangle\;,
\end{equs}
where we ordered symbols in increasing order of homogeneity and used $\scal{\cdot}$ to denote
the linear span.
Given any sufficiently regular function $\xi$ (say a continuous space-time function), there is then
a canonical way of lifting $\xi$ to a model $\LL(\xi) = (\Pi,\Gamma)$ for $T$ by setting
\minilab{e:canonical}
\begin{equ}[e:basic]
\bigl(\Pi_x \sXi\bigr)(y) = \xi(y)\;,\qquad \bigl(\Pi_x \X^k\bigr)(y) = (y-x)^k\;,
\end{equ}
and then recursively by
\minilab{e:canonical}
\begin{equ}[e:product]
\bigl(\Pi_x \tau \bar \tau\bigr)(y) = \bigl(\Pi_x \tau\bigr)(y)\cdot \bigl(\Pi_x \bar \tau\bigr)(y)\;,
\end{equ}
as well as \eref{e:defAdmissible}. (Note that here we used $x$ and $y$ as notations for generic
space-time points in order to keep notations compact.)

\subsection{Construction of the structure group}
\label{sec:structure}

So far, we have only described the vector space $T$ arising in the construction of a regularity 
structure suitable for the analysis of the dynamical $\Phi^4_3$ model, but we have not yet 
described the corresponding structure group $G$.
The reason why a non-trivial group $G$ is needed is that in \eqref{e:defAdmissible},
as soon as $|\CI(\tau)| > 0$, $\Pi_z \CI(\tau)$ depends non-trivially on the base point $z$,
so that $G$ is needed in order to be able to enforce the algebraic relation
$\Pi_{\bar z} = \Pi_{z}\Gamma_{z\bar z}$ of Definition~\ref{def:model}, which should
be interpreted as the action of ``reexpanding'' a ``Taylor series'' around a different point.

In our case, in view of \eref{e:defAdmissible}, the coefficients of these reexpansions will naturally be
some polynomials in $x$ and in the expressions appearing in \eref{e:defJ}. This suggests that we should
define a space $T^+$ whose basis vectors consist of formal expressions of the type
\begin{equ}[e:genPolynom]
X^k \prod_{i=1}^N \II_{\ell_i}(\tau_i)\;,
\end{equ}
where $N$ is an arbitrary but finite number, the $\tau_i$ are basis elements of $T$
different from the Taylor monomials $\X^k$,
and the $\ell_i$ are $d$-dimensional multiindices satisfying $|\ell_i| < |\tau_i| + 2$. 
(The last bound is a reflection of the restriction of the summands in \eref{e:defJ} with $\beta = 2$.)
The space $T^+$ also admits a natural graded structure $T^+ = \bigoplus T^+_\alpha$ by setting
\begin{equ}
|\II_{\ell}(\tau)| = |\tau| + 2 - |\ell|\;,\qquad |X^k| = |k|\;,
\end{equ}
and by postulating that the degree of a product is the sum of the degrees. Unlike in the case of $T$ however,
elements of $T^+$ all have strictly positive homogeneity, except for the empty product $\one$ which we postulate
to have degree $0$.

To any given admissible model $(\Pi,\Gamma)$, it is then natural to associate linear maps
$f_x \colon T^+ \to \R$ by $f_x (X^k) = (-x)^k$, $f_x(\sigma \bar \sigma) = f_x(\sigma) f_x(\bar \sigma)$, and
\begin{equ}[e:deffx]
f_x (\II_{\ell}\tau) = -\sum_{|k+\ell| < |\tau|+2} {(-x)^k\over k!} \int D^{(\ell+k)} K(x-y)\, \bigl(\Pi_x \tau\bigr)(dy)\;.
\end{equ}
The minus signs are here purely by convention and serve to make some expressions simpler later on.\footnote{This definition differs from that in \cite{Regular} by a simple change of basis in $T^+$ and,
while the inclusion of the sum over $k$ in \eqref{e:deffx} may not appear very natural here, it leads
to more natural expressions for the maps $\Delta$ and $\Deltap$ below.}
It then turns out that it is possible to build a linear map $\Delta \colon T \to T \otimes T^+$ such that
if we define $F_x \colon T \to T$ by
\begin{equ}[e:defineaction]
F_x \tau = (\id \otimes f_x)\Delta \tau\;,
\end{equ}
where $I$ denotes the identity operator on $T$, then these maps are invertible and,
given an admissible model $(\Pi,\Gamma)$,
$\Pi_x F_x^{-1}$ is independent of $x$. If this is the case, then one can then recover 
the maps $\Gamma_{xy}$ by
\begin{equ}[e:defGammaxy]
\Gamma_{xy} = F_x^{-1} \circ F_y\;.
\end{equ}
The ``correct'' definition for the map $\Delta$ compatible with  
\eqref{e:defAdmissible} and \eqref{e:deffx} is then given by
\begin{equ}
\Delta \1 = \1 \otimes \one\;,\qquad \Delta \sXi = \sXi \otimes \one\;,\qquad \Delta \X = \X \otimes X\;,
\end{equ}
and then recursively by
\begin{equ}
\Delta(\tau \bar \tau) = (\Delta \tau)(\Delta \bar \tau)\;,\qquad 
\Delta \CI(\tau) = (\CI \otimes \id) \Delta \tau + \sum_k {\X^k\over k!} \otimes \II_k(\tau)\;.
\end{equ}

\begin{remark}
This definition shows that the convention $\CI(\X^k) = 0$
is compatible with the convention $\II_\ell(\X^k) = 0$.
\end{remark}

The identities \eqref{e:defineaction} and \eqref{e:defGammaxy} and the definition of $f_x$ suggest
that one should take for $G$ the set of all linear maps of the type
\begin{equ}
\Gamma_f \tau = (\id \otimes f)\Delta \tau\;,
\end{equ}
where $f$ is a multiplicative linear functional on $T^+$.
Before we show that $G$ does indeed form a group, we argue that if $\Pi_x$ is defined
recursively as above and $F_x$ is given by \eqref{e:defineaction}, then there exists
a single linear map $\PPi\colon T \to \CS'$ such that
\begin{equ}[e:idenPPi]
\Pi_x\tau  = \PPi F_x \tau \qquad \forall \tau \in T\;.
\end{equ}

We can simply exhibit $\PPi$ explicitly. Set
\begin{equ}[e:defPPi1]
\bigl(\PPi \sXi\bigr)(x) = \xi(x)\;, \qquad
\bigl(\PPi \X^k\bigr)(x) = x^k\;,
\end{equ}
and then recursively
\begin{equ}[e:recPPi]
\PPi \tau \bar \tau = \PPi \tau \cdot \PPi \bar \tau\;,\qquad \PPi \CI \tau = K * \PPi \tau\;.
\end{equ}
Note that this is very similar to the definition of $\LL(\xi)$, with the notable exception
that \eref{e:defAdmissible} is replaced by the more ``natural'' identity $\PPi \CI \tau = K * \PPi \tau$.
It is then a simple exercise in binomial identities to show that indeed
$\Pi_x \X^k = \PPi F_x \X^k$ and that, assuming that \eqref{e:idenPPi} holds for some $\tau$ and that 
$f_x$ is given by \eqref{e:deffx}, one also
has $\Pi_x \CI(\tau) = \PPi F_x \CI(\tau)$. Finally, since all relevant objects are multiplicative,
one can see that if \eqref{e:idenPPi} holds for symbols $\tau$ and $\bar \tau$, then it must also 
hold for their product $\tau \bar \tau$.

We now argue that $G$ as defined above actually forms a group, so that in particular the maps
$F_x$ are invertible. Indeed, if we define a map $\Deltap \colon T^+ \to T^+\otimes T^+$
very similarly to $\Delta$ by
\begin{equ}
\Deltap \one = \one \otimes \one\;,\qquad \Deltap X = X \otimes X\;,
\end{equ} 
and then recursively by
\begin{equ}[e:defDeltap]
\Deltap(\sigma \bar \sigma) = (\Deltap \sigma)(\Deltap \bar \sigma)\;,\qquad 
\Deltap \II_\ell(\tau) = (\II_\ell \otimes \id) \Delta \tau + \sum_k {X^k\over k!} \otimes \II_{\ell+k}(\tau)\;,
\end{equ}
then it can be verified that this map intertwines with $\Delta$ via the relations
\begin{equ}[e:propDelta]
(\Delta \otimes I)\Delta = (I \otimes \Deltap)\Delta\;,\qquad 
(\Deltap \otimes I)\Deltap = (I \otimes \Deltap)\Deltap\;.
\end{equ}
We then define a product $\circ$ on the space of linear functionals $f \colon T^+ \to \R$ by
\begin{equ}
(f \circ g)(\sigma) = (f \otimes g)\Deltap \sigma\;.
\end{equ}
If we furthermore denote by $\Gamma_f$ the operator $T$ associated to any such linear functional as in 
\eref{e:defineaction}, the first identity of \eref{e:propDelta} 
yields the identity $\Gamma_f \Gamma_g = \Gamma_{f\circ g}$. The first identity of \eref{e:defDeltap} furthermore
ensures that if $f$ and $g$ are both multiplicative in the sense that $f(\sigma \bar \sigma) = f(\sigma) f(\bar \sigma)$,
then $f\circ g$ is again multiplicative. It also turns out that every multiplicative linear functional $f$ admits a 
unique inverse $f^{-1} = \CA f$ for some linear map $\CA \colon T \to T$ 
such that $f^{-1} \circ f = f\circ f^{-1} = e$, where $e \colon T^+ \to \R$
maps every basis vector of the form \eref{e:genPolynom} to zero, except for $e(\one) = 1$.
The element $e$ is neutral in the sense that $\Gamma_e$ is the identity operator.

It is now natural to define the structure group $G$ associated to $T$ as the set of all multiplicative
linear functionals on $T^+$, acting on $T$ via \eref{e:defineaction}. Furthermore, for any admissible model,
one has the identity
\begin{equ}
\Gamma_{xy} = F_x^{-1} F_y = \Gamma_{\gamma_{xy}}\;,\qquad \gamma_{xy} = f_x^{-1} \circ f_y\;.
\end{equ}

Returning to the relation between $\Pi_x$ and $\PPi$, we
showed actually more, namely that the knowledge of $\PPi$ and the 
knowledge of $(\Pi,\Gamma)$ are equivalent. Inedeed, one the  one hand one has
$\PPi = \Pi_x F_x^{-1}$ and the map $F_x$ can be recovered from $\Pi_x$ by \eref{e:deffx} and \eqref{e:defineaction}. 
On the other hand however, one also has of course $\Pi_x = \PPi F_x^{-1}$
and, if we equip $T$ with an adequate recursive structure (denote by 
$|\tau|_I$ the number
of times $\CI$ appears in the symbol $\tau$), then 
it is possible to show that the determination of $F_x^{-1} \tau$ (and therefore of $\Pi_x \tau$)
only requires knowledge of $f_x(\II_k(\sigma))$ (and therefore of $\Pi_x\sigma$) for
symbols $\sigma$ with $|\sigma|_I \le |\tau|_I -1$.

Furthermore, the translation $(\Pi,\Gamma) \leftrightarrow \PPi$ outlined above works for 
\textit{any} admissible model and does not at all rely on the fact that it was built by 
lifting a continuous function. In particular, it does \textit{not} rely on 
the fact that $\Pi_x$ and $\PPi$ are multiplicative. In the general case,
the first identity in \eref{e:recPPi} may then of course fail to be true, even
if $\PPi\tau$ happens to be a continuous function for every $\tau \in T$. 
The only reason why our definition of an admissible model does not simply consist of the
single map $\PPi$ is that there seems to be no simple way of describing the topology
given by Definition~\ref{def:model} in terms of $\PPi$.

\subsection{Renormalisation of the dynamical \texorpdfstring{$\Phi^4_3$}{Phi43} model}
\label{sec:Phi}

Combining Theorem~\ref{theo:mult} and Theorem~\ref{theo:Schauder}, we see that the 
map 
\begin{equ}
\Phi \mapsto \CP \bigl(\sXi - \Phi^3\bigr)\;,
\end{equ}
is a continuous map from $\CD^\gamma$ into itself, provided that $\gamma > 1+2\kappa$.
Actually, one obtains a regularity improvement in the sense that it maps
$\CD^\gamma$ into $\CD^{\bar \gamma}$ for some $\bar \gamma > \gamma$. In way
reminiscent of usual parabolic PDE techniques, one can leverage this regularity
improvement to build spaces $\CD^{\gamma,\eta}$ (essentially weighted versions of
the spaces $\CD^\gamma$ which allow coefficients to become singular near
times $0$) such that the following holds.

\begin{theorem}\label{theo:localSol}
For every $\alpha > -{2\over 3}$ there exist exponents $\gamma$ and $\eta$ such that,
for every $\Phi_0 \in \CC^\alpha$ and every admissible model, \eref{e:abstractFull} 
admits a unique local solution in $\CD^{\gamma,\eta}$. Furthermore, this solution
depends continuously on both $\Phi_0$ and the underlying model.
\end{theorem}

\begin{remark}
One hits some minor technical difficulties in the actual definition of local solutions
to \eqref{e:abstractFull}, but these are resolved in \cite{Regular}. One can also show
that solutions can be continued either forever or up to a time where the $\CC^\alpha$
solution of the solution explodes. In this particular instance, one can actually show
that solutions do not blow up, as shown in \cite{GlobalSols}.
\end{remark}

Given a model of the type $\LL(\xi)$ constructed as in the previous subsection
for a continuous space-time function $\xi$, 
it then  follows from \eref{e:product}
and the admissibility of $\LL(\xi)$ that the associated reconstruction operator satisfies the properties
\begin{equ}
\CR \CK f = K * \CR f \;,\qquad \CR (fg) = \CR f \cdot \CR g\;,
\end{equ} 
as long of course as all the functions to which $\CR$ is applied belong to $\D^\gamma$ for some $\gamma > 0$
so that $\CR f$ is uniquely defined by \eqref{e:boundRf}.
As a consequence, applying the reconstruction operator $\CR$ to both sides of  \eref{e:abstractFull},
we see that if $\Phi$ solves \eref{e:abstractFull} then 
$\CR \Phi$ solves the integral equation \eqref{e:intPhi4},
which thus yields the unique classical solution to \eref{e:Phi4}.
Note however that this identification of solutions to \eqref{e:abstractFull} with
solutions to \eqref{e:Phi4} relies crucially on the identity
$\CR(\Phi^3) = (\CR \Phi)^3$. This identity has no reason whatsoever 
to hold for a generic admissible model! In particular, while \eqref{e:abstractFull} makes sense
for \textit{any} admissible model, we do not know in general whether its solutions still solve
a local PDE.

At this stage, the situation is as follows. For any \textit{continuous} realisation $\xi$ of the driving noise,
we have factored the solution map $(\Phi_0,\xi) \to \Phi$ associated to \eref{e:Phi4} into maps
\begin{equ}
(\Phi_0,\xi) \to \bigl(\Phi_0,\LL(\xi)\bigr) \to \Phi \to \CR \Phi \;,
\end{equ}
where the middle arrow corresponds to the solution to \eref{e:abstractFull} 
built in Theorem~\ref{theo:localSol}.
The advantage of such a factorisation is that the last two arrows yield \textit{continuous} maps, even in topologies
sufficiently weak to be able to describe driving noise having the lack of regularity of space-time white noise.
The only arrow that isn't continuous in such a weak topology is the first one. We also hope that the
reader is convinced
that a similar construction can be performed for a very large class of semilinear stochastic PDEs.
In particular, the KPZ equation can also be analysed in this framework. 

Given this construction, one is lead naturally to the following 
question: given a sequence $\xi_\eps$ of ``natural'' regularisations
of space-time white noise, do the lifts $\LL(\xi_\eps)$ converge in 
probably in a suitable space of admissible models? Unfortunately, unlike in the case of the theory of rough paths
where this is very often the case (but see \cite{KPZ} for an example where it fails there too), 
the answer to this question in the context of SPDEs is often an emphatic \textit{no}.
Indeed, if it were the case for the dynamical $\Phi^4_3$ model,
 then one could have chosen the constant $C_\eps$ to be independent of $\eps$ in 
\eref{e:Phi43Regular}, which is certainly not the case.

The way in which we are able to circumvent the fact that $\LL(\xi_\eps)$ does not converge to a limiting model
as $\eps \to 0$ is to consider instead a sequence of \textit{renormalised} models. The main idea
is to exploit the fact that our definition of a model does not impose the identity
\eref{e:product}, even in situations where $\xi$ itself happens to be a continuous function. 
One question that then imposes itself is: what are the natural ways of ``deforming'' the usual
product which still lead to an admissible model? It turns out that the regularity structure
whose construction was sketched above comes equipped with a natural \textit{finite-dimensional} 
group of continuous transformations $\RR$ on its space of admissible models (henceforth called the 
``renormalisation group''), which essentially amounts
to the space of all natural deformations of the product. It then turns out that even though 
$\LL(\xi_\eps)$ does not converge, it is possible to find a sequence $M_\eps$ of elements in $\RR$ such that 
the sequence $M_\eps \LL(\xi_\eps)$ converges to a limiting model $(\hat \Pi, \hat \Gamma)$.
Unfortunately, the elements $M_\eps$ no \textit{not} preserve the image of $\LL$ in the space of
admissible models. As a consequence, when solving the fixed point map \eref{e:abstractFull} 
with respect to the model $M_\eps \LL(\xi_\eps)$ and inserting the solution into the reconstruction operator,
it is not clear \textit{a priori} that the resulting function (or distribution) can again be interpreted
as the solution to some modified PDE. It turns out however that
this is again the case and the modified equation is precisely given by \eref{e:Phi43Regular}, 
where $C_\eps$ is 
some linear combination of the two constants appearing in the description of $M_\eps$.

There are now three questions that remain to be answered:
\begin{enumerate}
\item How does one construct the renormalisation group $\RR$?
\item How does one derive the new equation obtained when renormalising a model?
\item What is the right choice of $M_\eps$ ensuring that the renormalised models converge?
\end{enumerate}

\subsection{The renormalisation group}

In order to build the group $\RR$, it turns out to be appropriate to describe its action 
first at the level of $\PPi$ rather than at the level of $(\Pi,\Gamma)$. 
At this stage we note that if $\xi$ happens to be a stationary stochastic process and $\PPi$ is built
from $\xi$ by \eqref{e:defPPi1} and \eqref{e:recPPi}, 
then $\PPi \tau$ is also a stationary stochastic process for every
$\tau \in T$. In order to define $\RR$, it is natural to consider only transformations of the space of admissible
models that preserve this property. Since we are not in general allowed to multiply components of $\PPi$,
the only remaining operation is to form linear combinations between them. 
It is therefore natural to describe elements of $\RR$ 
by linear maps $M \colon T \to T$ and to postulate their action on admissible models by 
$\PPi \mapsto \PPi^M$ with
\begin{equ}
\PPi^M \tau = \PPi M \tau\;.
\end{equ}
It is not clear \textit{a priori} whether given such a map $M$ and an admissible model
$(\Pi,\Gamma)$ there is a coherent way of building a new model $(\Pi^M, \Gamma^M)$ such that 
$\PPi^M$ is the map associated to $(\Pi^M, \Gamma^M)$ as above. It turns out that one has the following statement:

\begin{proposition}\label{prop:transform}
In the above context, for every linear map $M \colon T \to T$ commuting with $\CI$ and multiplication
by $X^k$, there exist \textit{unique}
linear maps $\Delta^M \colon T \to T \otimes T^+$ and $\hat \Delta^M \colon T^+ \to T^+ \otimes T^+$ such that if we set
\begin{equ}
\Pi_x^M \tau = \bigl(\Pi_x \otimes f_x\bigr)\Delta^M \tau \;,\qquad \gamma_{xy}^M(\sigma) = (\gamma_{xy} \otimes f_x)\hat \Delta^M\sigma\;,
\end{equ}
then $\Pi_x^M$ satisfies again \eref{e:defAdmissible} and the identity $\Pi_x^M \Gamma_{xy}^M = \Pi_y^M$.
\end{proposition}

At this stage it may look like \textit{any} linear map $M \colon T \to T$ commuting with $\CI$ and multiplication by $X^k$
yields a transformation on the space
of admissible models by Proposition~\ref{prop:transform}. This however is not true since we have completely
disregarded the \textit{analytical} bounds that every model has to satisfy.
It is clear from Definition~\ref{def:model} that these are satisfied in general 
if and only if, for every symbol $\tau$, $\Pi_x^M \tau$ is a linear combination
of the $\Pi_x \bar\tau$ only involving symbols $\bar \tau$ with $|\bar\tau| \ge |\tau|$. 
This suggests the following definition. 

\begin{definition}
The renormalisation group $\RR$ consists of the set of linear maps $M \colon T \to T$ commuting with $\CI$
and with multiplication by $X^k$,
such that for $\tau \in T_\alpha$ and $\sigma \in T_\alpha^+$, one has 
\begin{equ}[e:defRenorm]
\Delta^M \tau - \tau \otimes \one \in  T_{>\alpha} \otimes T^+\;,\qquad
\hat \Delta^M \sigma - \sigma \otimes \one \in T_{>\alpha}^+ \otimes T^+\;.
\end{equ}
Its action on the space of admissible models is given by Proposition~\ref{prop:transform}.
\end{definition}

\begin{remark}
It turns out that the second condition of \eqref{e:defRenorm} is actually a consequence
of the first one, see the appendix in \cite{Jeremy}.
\end{remark}

\subsection{The renormalised equations}
\label{sec:renorm}

In the case of the dynamical $\Phi^4_3$ model considered in this article, 
it turns out that in general (i.e.\ if we also want to cover the situation mentioned
in the introduction where $\xi$ is approximated by a rescaled non-Gaussian stationary process)
we need a five-parameter subgroup of $\RR$
to renormalise the equations. More precisely, we consider elements $M \in \RR$ of the form
$M = \exp(- \sum_{i=1}^5 C_i L_i)$, where the generators $L_i$ are 
determined by the substitution rules
\begin{equ}
L_1 \colon \<2> \mapsto \1\;,\quad L_2 \colon \<3> \mapsto \1\;,\quad  L_3 \colon \<22> \mapsto \1\;,\quad
L_4 \colon \<32> \mapsto \1\;,\quad L_5 \colon \<31> \mapsto \1\;.
\end{equ}
This should be understood in the sense that if $\tau$ is an arbitrary formal expression,
then $L_1 \tau$ is the sum of all formal expressions obtained from $\tau$ by performing
a substitution of the type $\<2> \mapsto 1$, and similarly for the other $L_i$. For example, one has
\begin{equ}
L_1 \<3> = 3 \<1>\;,\qquad L_1 \<12> = \<10> \;,\qquad L_3 \<32> = 3 \<1>\;,\qquad L_1 \<1> = 0\;.
\end{equ}
The convention $\CI(\1) = 0$ also suggests that one should set for example
$L_2 \<32> = 0$, since the symbol \<32b>, which is what the substitution $ \<3> \mapsto \1$ creates, 
would contain a factor of $\CI(\1)$.
It is easy to see that all of the $L_i$ commute with each other so that, as an abstract group,
the transformations we consider form a copy of $\R^5$ with addition as its group operation.
This is not a general fact however.
One then has the following result:

\begin{proposition}\label{prop:renorm}
The linear maps $M$ of the type just described belong to $\RR$. Furthermore, if $(\Pi,\Gamma)$ is an
admissible model such that $\Pi_x \tau$ is a continuous function for every $\tau \in T$, then one has the
identity
\begin{equ}[e:actionRenorm]
\bigl(\Pi_x^M \tau\bigr)(x) = \bigl(\Pi_x M \tau\bigr)(x)\;.
\end{equ}
\end{proposition}

\begin{remark}
Note that it it is the same value $x$ that appears twice on each side of \eref{e:actionRenorm}.
It is in fact \textit{not} the case that one has $\Pi_x^M \tau = \Pi_x M \tau$!
However, the identity \eref{e:actionRenorm} is all we need to derive the renormalised equations. 
\end{remark}

It is now rather straightforward to show the following:

\begin{proposition}\label{prop:renormEquation}
Let $M$ be as above and let $(\Pi^M,\Gamma^M)$ be the model obtained by acting 
with $M$ on the canonical model $(\Pi,\Gamma) = \LL(\xi)$ for some smooth function $\xi$.
Let furthermore $\Phi$ be the solution to \eref{e:abstractFull} with respect to the model $(\Pi^M,\Gamma^M)$. Then, the function
$u(t,x) = \bigl(\CR^M \Phi\bigr)(t,x)$ solves the equation
\begin{equ}
\d_t u = \Delta u - u^3 + (3C_1 - 9 C_3 - 6C_5)u - (C_2 + 3C_4) + \xi\;.
\end{equ}
\end{proposition}

\begin{proof}
By Theorem~\ref{theo:mult}, it turns out that \eref{e:abstractFull} can be solved locally in $\D^\gamma$ 
(or rather a weighted version of this space taking into account possible blowup near time $0$)
as soon as $\gamma$ is a
little bit greater than $1$. Therefore, we only need to keep track of its solution $\Phi$ up to terms of homogeneity $1$.
By repeatedly applying the identity \eref{e:abstract}, we see that the solution $\Phi$ is necessarily of the form
\begin{equ}[e:decompuPhi]
\Phi = \<1> + \phi\, \1 - \<30> - 3 \phi\, \<20> + \scal{\nabla \phi, \X}\;,
\end{equ}
for some real-valued function $\phi$ and some $\R^3$-valued function $\nabla \phi$. We
emphasise that $\nabla \phi$ is treated
as an independent function here, we certainly do not suggest that the function $\phi$ is 
differentiable! Our notation
is only by analogy with the classical Taylor expansion. Similarly, the right hand side of 
the equation is given up to order $0$ by
\begin{equ}
\sXi - \Phi^3 = \sXi - \<3> - 3\phi\, \<2> + 3 \<32> - 3 \phi^2\,\<1> + 6 \phi\, \<31> 
 + 9\phi\, \<22> - 3 \scal{\nabla\phi, \<2> \X}
- \phi^3\,\1 \;. \label{e:RHS}
\end{equ}
Combining this with the definition of $M$, it is straightforward to see that, modulo terms of strictly
positive homogeneity, one has
\begin{equs}
M (\sXi - \Phi^3) &= \sXi - (M\Phi)^3 + 3C_1 \<1> + C_2\1 + 3C_1 \phi \1 - 9 C_3 \<1> - 3 C_4\1\\
&\qquad  - 6C_5\<1> 
- 6C_5\phi\1 - 9C_3 \phi \1\\
& = \sXi - (M\Phi)^3 + (3C_1 - 9C_3 - 6C_5) M\Phi + (C_2 - 3C_4)\1\;.
\end{equs}
Combining this with \eref{e:actionRenorm} and applying as before the reconstruction operator
$\CR^M$ to both sides of \eqref{e:abstractFull}, the claim now follows at once. 
\end{proof}

\subsection{Convergence of the renormalised models for Gaussian approximations}
\label{sec:convGauss}

We now argue that if $\xi_\eps = \rho_\eps * \xi$ as in Theorem~\ref{theo:construction},
then one expects to be able to find constants $C_1^{(\eps)}$ and $C_3^{(\eps)}$ (and set 
$C_i^{(\eps)} =0$ for $i \in \{2,4,5\}$) such that the 
sequence of renormalised models $M^\eps \LL(\xi_\eps)$ defined as in the previous subsection
converges to a limiting model. 
Instead of considering the actual sequence of models, we only consider the sequence of
stationary processes $\hat \PPi^\eps \tau := \PPi^\eps M^\eps \tau$, where $\PPi^\eps$ is associated
to $(\Pi^\eps, \Gamma^\eps) = \LL(\xi_\eps)$ as before. 
Since there are general arguments available to deal
with all the expressions $\tau$ of positive homogeneity, we restrict ourselves to those of negative
homogeneity which, leaving out $\sXi$ which is easy to treat, are given by
\begin{equ}
\<3>,\; \<2>,\; \<32>,\;\<1>,\; \<31>,\; \<22>,\; \X_i \<2>\;.
\end{equ}

\begin{remark}
Even if we can show that $\hat \PPi^\eps \tau$ converges weakly to some limit as $\eps \to 0$,
this does not necessarily imply that the corresponding sequence of models converges in their
natural topology. We will not try to address convergence in the correct topology here, although this
is a non-trivial problem. 
\end{remark}

For this section, some elementary notions from the theory of Wiener chaos expansions are required, but we will
try to hide this as much as possible.
Recall that $\PPi^\eps \<1> = K * \xi_\eps = K_\eps * \xi$,
where the kernel $K_\eps$ is given by $K_\eps = K * \rho_\eps$. This shows that, for $\eps > 0$, one has
\begin{equ}
\bigl(\PPi^\eps \<2>\bigr)(z) = \bigl(K * \xi_\eps\bigr)(z)^2 = \int\!\!\int K_\eps(z-z_1)K_\eps(z-z_2)\, \xi(z_1)\xi(z_2)\,dz_1\,dz_2\;.
\end{equ}
Here, we make use of the fact that, for any two distributions $\zeta$ and $\eta$, the product
$\zeta \cdot \eta$ is well-defined as a distribution on the product space. 
It is only when one tries to define a pointwise product yielding again a distribution
on the \textit{same} space that one runs into trouble.
Similar but more complicated expressions can be found for any formal expression $\tau$. 
This naturally leads to the study of random variables of the type
\begin{equ}[e:defIk]
I_k(f) = \int \!\cdots\! \int f(z_1,\ldots,z_k)\, \xi(z_1)\cdots \xi(z_k)\, dz_1\cdots dz_k\;.
\end{equ}
Ideally, one would hope to have an It\^o isometry of the type $\E I_k(f)I_k(g) = \scal{f^\sym,g^\sym}$,
where $\scal{\cdot,\cdot}$ denotes the $L^2$-scalar product and $f^\sym$ denotes the symmetrisation of $f$.
This would then allow to extend \eqref{e:defIk} from smooth test functions to all functions $f \in L^2$.
It is unfortunately \textit{not} the case that one has such an isometry. 
Instead, one should first replace the products in \eref{e:defIk}
by \textit{Wick products}, which are obtained by considering all possible \textit{contractions} of the type
\begin{equ}
\xi(z_i)\xi(z_j) \mapsto \Wick{\xi(z_i)\, \xi(z_j)} + \delta(z_i - z_j)\;.
\end{equ} 
For example, the distribution $\Wick{\xi(z_i)\, \xi(z_j)\, \xi(z_k)}$ is defined by the identity
\begin{equ}
\xi(z_i)\xi(z_j)\xi(z_k) = \Wick{\xi(z_i)\xi(z_j) \xi(z_k)} + \xi(z_i)\delta(z_j-z_k)
 + \xi(z_j)\delta(z_k-z_i) + \xi(z_k)\delta(z_i-z_j)\;.
\end{equ}
(See Section~\ref{sec:CLT} below for a more general definition which covers Gaussian
random variables as a special case.) If we then set
\begin{equ}
\hat I_k(f) = \int \!\cdots\! \int f(z_1,\ldots,z_k)\, \Wick{\xi(z_1)\cdots\xi(z_k)}\, dz_1\cdots dz_k\;,
\end{equ}
which is well-defined for all smooth test functions $f$, one recovers indeed
the It\^o isometry
\begin{equ}[e:defIto]
\E \hat I_k(f) \hat I_k(g) = \scal{f^\sym,g^\sym}\;.
\end{equ}
We refer to \cite{Nualart} for a more thorough description of this construction, which also
goes under the name of \textit{Wiener chaos}, with random variables of the type $\hat I_k(f)$ said
to belong to the chaos of $k$th order.
While chaoses of different order are not independent, they are orthogonal in the
sense that $\E \hat I_k(f) \hat I_\ell(g) = 0$ if $k \neq \ell$. 

One very nice property is that one has equivalence of moments in the sense that, for every $k>0$ and $p>0$ there exists a constant $C_{k,p}$
such that 
\begin{equ}[e:hypercontract]
\E |\hat I_k(f)|^p \le C_{k,p} \|f^\sym\|_{L^2}^p\le C_{k,p} \|f\|_{L^2}^p\;,
\end{equ}
where the second bound comes from the fact that symmetrisation is a contraction in $L^2$.
In other words, the $p$th moment of a random variable belonging to a Wiener chaos of fixed order can
be bounded by the corresponding power of its second moment.
The reason why such a bound is very useful is the following Kolmogorov-type result:

\begin{theorem}\label{theo:Kolmogorov}
Let $\xi$ be a distribution-valued random variable such that there exists $\alpha \le 0$ such that,
for every $p > 0$, the bound
\begin{equ}[e:boundKolmo]
\E |\xi(\phi_x^\lambda)|^p \lesssim \lambda^{\alpha p}\;,
\end{equ}
holds uniformly over $\phi \in \CB_r$ for some $r > |\alpha|$, all
$\lambda \in (0,1]$, and locally uniformly in $x$.
Then, for every $\beta < \alpha$, there exists a $\CC^\beta$-valued version of $\xi$.
\end{theorem}

\begin{remark}
Here, we say that $\tilde \xi$ is a version of $\xi$ if, for every smooth test function $\phi$,
 $\tilde\xi(\phi) = \xi(\phi)$ almost surely.
\end{remark}

\begin{remark}
There is an analogous result allowing to show the convergence of a sequence of models
to a limiting model in the topology given by Definition~\ref{def:model}, see \cite[Thm~10.7]{Regular} for
more details.
\end{remark}

\begin{remark}
As usual for Kolmogorov-type results, one can trade integrability for regularity in the sense that if
\eqref{e:boundKolmo} only holds for some fixed value $p > 2$, then we can only conclude that
$\xi$ has a version in some $\CC^\beta$ for $\beta < \alpha - {d \over p}$, where $d$ denotes the
scaling dimension of the underlying space. In our particular case one has $d=5$ since
we consider space-time distributions and ``time counts double''.
\end{remark}

Returning to our problem, we first argue that it should be possible to choose $M$ in such a way that 
$\hat \PPi^\eps \<2>$ converges to a limit as $\eps \to 0$.
The above considerations suggest that one should rewrite $\PPi^\eps \<2>$ as
\begin{equ}[e:Psi2]
\bigl(\PPi^\eps \<2>\bigr)(z) = \bigl(K * \xi_\eps\bigr)(z)^2 = \int\!\!\int K_\eps(z-z_1)K_\eps(z-z_2)\, 
\Wick{\xi(z_1)\,\xi(z_2)}\,dz_1\,dz_2 + C_\eps\;,
\end{equ}
where the constant $C_\eps$ is given by
\begin{equ}[e:defCeps]
C_\eps = \int K_\eps^2(z_1)\,dz_1 = \int K_\eps^2(z-z_1)\,dz_1\;.
\end{equ}
At this stage, it is convenient to introduce a graphical notation ``\`a la Feynman'' for these 
multiple integrals since expressions otherwise rapidly became unwieldy.
Similarly to \cite{Regular,Etienne,HaoShen}, we denote dummy integration variables by black dots,
we write \tikz \node[root] {}; for the distinguished variable $z$, 
\tikz \node[var] {}; for a Wick factor of the type
$\xi(z_i)$, \tikz[baseline=-1mm] \draw[keps] (0,0) -- (1,0); 
for the kernel $K_\eps$ evaluated at the difference between the variables
representing its two endpoints, and we implicitly assume that all variables except for 
\tikz \node[root] {}; are integrated out. With this notation, \eqref{e:Psi2} and \eqref{e:defCeps} can be expressed in 
a much friendlier  and more compact way as
\begin{equ}[e:defCepsGraph]
\PPi^\eps \<2> = 
\begin{tikzpicture}[scale=0.35,baseline=0cm]
	\node at (0,-1)  [root] (root) {};
	\node at (-1,1.5)  [var] (left) {};
	\node at (1,1.5)  [var] (right) {};
	
	\draw[keps] (left) to  (root);	
	\draw[keps] (right) to (root);
\end{tikzpicture}\;
 + \;
\begin{tikzpicture}[scale=0.35,baseline=0cm]
	\node at (0,-1)  [root] (root) {};
	\node at (0,1.6)  [int] (top) {};
	
	\draw[keps,bend right = 60] (top) to  (root);
	
	\draw[keps,bend left = 60] (top) to (root);
\end{tikzpicture}\;,\qquad
C_\eps = \begin{tikzpicture}[scale=0.35,baseline=0cm]
	\node at (0,-1)  [root] (root) {};
	\node at (0,1.6)  [int] (top) {};
	
	\draw[keps,bend right = 60] (top) to  (root);
	
	\draw[keps,bend left = 60] (top) to (root);
\end{tikzpicture}\;.
\end{equ}
This notation is also compatible with the graphical notation we are using for the 
various symbols: the graphical notation for $\PPi^\eps \tau$ is essentially the same
as that for $\tau$ but, as a consequence of the definition of Wick products, 
one should sum over all possible graphs obtained by pairwise contractions 
of the noises \tikz \node[var] {};.

Note now that $K_\eps$ is an $\eps$-approximation of the kernel $K$ which has the same singular behaviour
as the heat kernel. In terms of the parabolic distance, the singularity of the heat kernel scales like
$K(z) \sim |z|^{-3}$ for $z \to 0$. (Recall that we consider the parabolic distance $|(t,x)| = \sqrt{|t|} + |x|$,
so that this is consistent with the fact that the heat kernel is bounded by $t^{-3/2}$.)
This suggests that one has $K_\eps^2(z) \sim |z|^{-6}$ for $|z| \gg \eps$. Since parabolic space-time has scaling dimension
$5$ (time counts double!), this is a non-integrable singularity. As a matter of fact, there is a whole power of $z$
missing to make it borderline integrable, which correctly suggests that one has
\begin{equ}
C_\eps \sim {1\over \eps}\;.
\end{equ}
This already shows that one should not expect $\PPi^\eps \<2>$ to converge to a limit as $\eps \to 0$. 
Indeed, as we will see presently,
it turns out that the first term in \eref{e:Psi2} converges to a distribution-valued stationary 
space-time process,
so that one would like to somehow get rid of this diverging constant $C_\eps$. This is exactly where the renormalisation
map $M^\eps$ (in particular the factor $\exp(-C_1^{(\eps)} L_1)$) enters into play. Following the above
definitions, we see that one has
\begin{equ}
\bigl(\hat \PPi^\eps \<2>\bigr)(z) = \bigl(\PPi^\eps M\<2>\bigr)(z) = \bigl(\PPi^\eps \<2>\bigr)(z)  - C_1^{(
\eps)}\;.
\end{equ}
This suggests that if we make the choice $C_1^{(\eps)} = C_\eps$, then 
$\hat \PPi^\eps \<2>$ does indeed converge to a 
non-trivial limit as $\eps \to 0$. Writing \tikz[baseline=-1mm] \draw[kernel] (0,0) -- (1,0);
for the kernel $K$,
this limit is a distribution which can formally be described as
\begin{equ}
\PPi^\eps \<2> = 
\begin{tikzpicture}[scale=0.35,baseline=0cm]
	\node at (0,-1)  [root] (root) {};
	\node at (-1,1.5)  [var] (left) {};
	\node at (1,1.5)  [var] (right) {};
	
	\draw[kernel] (left) to  (root);	
	\draw[kernel] (right) to (root);
\end{tikzpicture}\;,
\end{equ}
with the implicit understanding that $\bigl(\PPi^\eps \<2> \bigr)(z)$ does not make sense
as a random variable, but must first be integrated against a smooth test function in order to produce
a random variable belonging to the second Wiener chaos.
Using the scaling properties of the kernel $K$,
it is not too difficult to show that this procedure works, and that the resulting 
random variables satisfy a
bound of the type considered in Theorem~\ref{theo:Kolmogorov} with $\alpha = -1$. 
Once we know that $\hat \PPi^\eps \<2>$ converges,
it is immediate that $\hat \PPi^\eps X\<2>$ converges as well, since this amounts to just multiplying a distribution
by a smooth function. 

A similar argument to what we did for $\<2>$ allows to take care of $\tau = \<3>$ since one then has
\begin{equs}
\bigl(\PPi^\eps \<3>\bigr)(z) &= 
\begin{tikzpicture}[scale=0.35,baseline=0cm]
	\node at (0,-1)  [root] (root) {};
	\node at (-1.5,1.5)  [var] (left) {};
	\node at (0,1.8)  [var] (middle) {};
	\node at (1.5,1.5)  [var] (right) {};
	
	\draw[keps] (left) to  (root);	
	\draw[keps] (right) to (root);
	\draw[keps] (middle) to (root);
\end{tikzpicture}\;
 + 3\;
\begin{tikzpicture}[scale=0.35,baseline=0cm]
	\node at (0,-1)  [root] (root) {};
	\node at (-1,1.6)  [int] (top) {};
	\node at (1,1.6)  [var] (topr) {};
	
	\draw[keps,bend right = 60] (top) to  (root);	
	\draw[keps,bend left = 60] (top) to (root);
	\draw[keps,bend left = 60] (topr) to (root);
\end{tikzpicture}\;.
\end{equs}
Noting that the second term in this expression is nothing but
$3C_\eps \bigl(\PPi^\eps \<1>\bigr)(z)$,
we see that in this case, provided again that $C_1^{(\eps)} = C_\eps$, 
$\hat \PPi^\eps \<3>$ is given by only the first term in the expression
above, which turns out to converge to a non-degenerate limiting random 
distribution in a similar way to what happened
for $\<2>$, but this time the Kolmogorov-type bound only holds for $\alpha = -{3\over 2}$. 

Going down our list of terms of negative homogeneity, we see that it remains to consider $\<31>$, $\<22>$,
and $\<32>$. It turns out that the last one is the most difficult, so we only discuss that one.
 The explicit expression for $\PPi^\eps \<32>$ is given in our graphical notation by
\begin{equs}
\PPi^\eps \<32> &= 
\begin{tikzpicture}[scale=0.35,baseline=0cm]
	\node at (0,-1)  [root] (root) {};
	\node at (0,1.5)  [int] (middle) {};
	\node at (-1.6,1)  [var] (left) {};
	\node at (1.6,1)  [var] (right) {};
	\node at (-1.5,3.5)  [var] (tl) {};
	\node at (1.5,3.5)  [var] (tr) {};
	\node at (0,4)  [var] (tm) {};
	
	\draw[keps] (left) to  (root);	
	\draw[keps] (right) to (root);
	\draw[keps] (tl) to  (middle);	
	\draw[keps] (tr) to (middle);
	\draw[keps] (tm) to  (middle);	
	\draw[kernel] (middle) to (root);
\end{tikzpicture}\; + 
6\;
\begin{tikzpicture}[scale=0.35,baseline=0cm]
	\node at (0,-1)  [root] (root) {};
	\node at (0,1.5)  [int] (middle) {};
	\node at (-1.6,.25)  [int] (left) {};
	\node at (1.6,1)  [var] (right) {};
	\node at (1.5,3.5)  [var] (tr) {};
	\node at (0,4)  [var] (tm) {};
	
	\draw[keps,bend right=40] (left) to  (root);	
	\draw[keps] (right) to (root);
	\draw[keps,bend left=40] (left) to  (middle);	
	\draw[keps] (tr) to (middle);
	\draw[keps] (tm) to  (middle);	
	\draw[kernel] (middle) to (root);
\end{tikzpicture}\; + 
3\;
\begin{tikzpicture}[scale=0.35,baseline=0cm]
	\node at (0,-1)  [root] (root) {};
	\node at (0,1.5)  [int] (middle) {};
	\node at (-1.6,1)  [var] (left) {};
	\node at (1.6,1)  [var] (right) {};
	\node at (-1,3.6)  [int] (tl) {};
	\node at (1,3.6)  [var] (tr) {};
	
	\draw[keps] (left) to  (root);	
	\draw[keps] (right) to (root);
	\draw[keps,bend right = 60] (tl) to  (middle);	
	\draw[keps,bend left = 60] (tl) to (middle);
	\draw[keps,bend left = 60] (tr) to (middle);
	\draw[kernel] (middle) to (root);
\end{tikzpicture}\; + \;
\begin{tikzpicture}[scale=0.35,baseline=0cm]
	\node at (0,-1)  [root] (root) {};
	\node at (1,1.6)  [int] (middle) {};
	\node at (-1,1.6)  [int] (left) {};
	\node at (-0.5,3.5)  [var] (tl) {};
	\node at (2.5,3.5)  [var] (tr) {};
	\node at (1,4)  [var] (tm) {};

	\draw[keps,bend right = 60] (left) to  (root);	
	\draw[keps,bend left = 60] (left) to (root);
	\draw[kernel,bend left = 60] (middle) to (root);
	
	\draw[keps] (tl) to  (middle);	
	\draw[keps] (tr) to (middle);
	\draw[keps] (tm) to  (middle);	
\end{tikzpicture}\\ 
&\qquad + 
3\;
\begin{tikzpicture}[scale=0.35,baseline=0cm]
	\node at (0,-1)  [root] (root) {};
	\node at (1,1.6)  [int] (middle) {};
	\node at (-1,1.6)  [int] (left) {};
	\node at (0,3.6)  [int] (tl) {};
	\node at (2,3.6)  [var] (tr) {};

	\draw[keps,bend right = 60] (left) to  (root);	
	\draw[keps,bend left = 60] (left) to (root);
	\draw[kernel,bend left = 60] (middle) to (root);
	
	\draw[keps,bend right = 60] (tl) to  (middle);	
	\draw[keps,bend left = 60] (tl) to (middle);
	\draw[keps,bend left = 60] (tr) to (middle);
\end{tikzpicture}\; + 6\;
\begin{tikzpicture}[scale=0.35,baseline=0cm]
	\node at (0,-1)  [root] (root) {};
	\node at (0,1.5)  [int] (middle) {};
	\node at (-1.6,.25)  [int] (left) {};
	\node at (1.6,1)  [var] (right) {};
	\node at (0,3.5)  [int] (top) {};
	
	\draw[keps,bend right=40] (left) to  (root);	
	\draw[keps] (right) to (root);
	\draw[keps,bend left=40] (left) to  (middle);	
	\draw[keps,bend right = 60] (top) to (middle);
	\draw[keps,bend left = 60] (top) to  (middle);	
	\draw[kernel] (middle) to (root);
\end{tikzpicture}\; + 6\;
\begin{tikzpicture}[scale=0.35,baseline=0cm]
	\node at (0,-1)  [root] (root) {};
	\node at (0,1.5)  [int] (middle) {};
	\node at (-1.6,.25)  [int] (left) {};
	\node at (1.6,.25)  [int] (right) {};
	\node at (0,4)  [var] (tm) {};
	
	\draw[keps,bend right=40] (left) to  (root);	
	\draw[keps,bend left=40] (left) to  (middle);	
	\draw[keps,bend left=40] (right) to  (root);	
	\draw[keps,bend right=40] (right) to  (middle);	
	\draw[keps] (tm) to  (middle);	
	\draw[kernel] (middle) to (root);
\end{tikzpicture}\;.
\end{equs}
It turns out that \textit{all} of these terms, except for the first two,
are divergent as $\eps \to 0$, so we have to hope that our definition of
$\hat \PPi^\eps$ creates sufficiently many cancellations to take care of them!
A simple calculation using our definition of $M^\eps$ shows that one has the 
identity
\begin{equ}
\hat \PPi^\eps \<32> = \PPi^\eps\<32> - 3C_1^{(\eps)} \PPi^\eps \<12> - C_1^{(\eps)} \PPi^\eps \<30>
+ 3 (C_1^{(\eps)})^2 \PPi^\eps \<10> - 3 C_3^{(\eps)} \PPi^\eps\<1>\;.
\end{equ}
(Recall that we have set $C_i^{(\eps)} = 0$ for $i \not \in \{1,3\}$.)
Inserting the definition of $\PPi^\eps$ into this expression and making use of
the identity $C_1^{(\eps)} = C_\eps$ with $C_\eps$ given by \eqref{e:defCepsGraph},
we see that many terms do indeed cancel out, finally yielding
\begin{equ}[e:Pihat32]
\hat \PPi^\eps \<32> = 
\begin{tikzpicture}[scale=0.35,baseline=0cm]
	\node at (0,-1)  [root] (root) {};
	\node at (0,1.5)  [int] (middle) {};
	\node at (-1.6,1)  [var] (left) {};
	\node at (1.6,1)  [var] (right) {};
	\node at (-1.5,3.5)  [var] (tl) {};
	\node at (1.5,3.5)  [var] (tr) {};
	\node at (0,4)  [var] (tm) {};
	
	\draw[keps] (left) to  (root);	
	\draw[keps] (right) to (root);
	\draw[keps] (tl) to  (middle);	
	\draw[keps] (tr) to (middle);
	\draw[keps] (tm) to  (middle);	
	\draw[kernel] (middle) to (root);
\end{tikzpicture}\; + 
6\;
\begin{tikzpicture}[scale=0.35,baseline=0cm]
	\node at (0,-1)  [root] (root) {};
	\node at (0,1.5)  [int] (middle) {};
	\node at (-1.6,.25)  [int] (left) {};
	\node at (1.6,1)  [var] (right) {};
	\node at (1.5,3.5)  [var] (tr) {};
	\node at (0,4)  [var] (tm) {};
	
	\draw[keps,bend right=40] (left) to  (root);	
	\draw[keps] (right) to (root);
	\draw[keps,bend left=40] (left) to  (middle);	
	\draw[keps] (tr) to (middle);
	\draw[keps] (tm) to  (middle);	
	\draw[kernel] (middle) to (root);
\end{tikzpicture}\; + 6\;
\begin{tikzpicture}[scale=0.35,baseline=0cm]
	\node at (0,-1)  [root] (root) {};
	\node at (0,1.5)  [int] (middle) {};
	\node at (-1.6,.25)  [int] (left) {};
	\node at (1.6,.25)  [int] (right) {};
	\node at (0,4)  [var] (tm) {};
	
	\draw[keps,bend right=40] (left) to  (root);	
	\draw[keps,bend left=40] (left) to  (middle);	
	\draw[keps,bend left=40] (right) to  (root);	
	\draw[keps,bend right=40] (right) to  (middle);	
	\draw[keps] (tm) to  (middle);	
	\draw[kernel] (middle) to (root);
\end{tikzpicture}
\; - 3C_3^{(\eps)}\;
\begin{tikzpicture}[scale=0.35,baseline=0cm]
	\node at (0,-1)  [root] (root) {};
	\node at (0,1.5)  [var] (middle) {};
	
	\draw[keps] (middle) to (root);
\end{tikzpicture}\;.
\end{equ}
This is still problematic: the penultimate term in this expression contains the kernel
\begin{equ}[e:defQeps]
Q_\eps = 
\begin{tikzpicture}[scale=0.35,baseline=-0.1cm]
	\node at (4,0)  [root] (root) {};
	\node at (0,0)  [root] (middle) {};
	\node at (2,-1.5)  [int] (left) {};
	\node at (2,1.5)  [int] (right) {};
	
	\draw[keps,bend right=30] (left) to  (root);	
	\draw[keps,bend left=30] (left) to  (middle);	
	\draw[keps,bend left=30] (right) to  (root);	
	\draw[keps,bend right=30] (right) to  (middle);	
	\draw[kernel] (middle) to (root);
\end{tikzpicture}\;.
\end{equ}
(This is a slight abuse of our notation: the two variables \tikz \node[root] {};
are fixed and $Q_\eps$ is evaluated at their difference.)
As $\eps \to 0$, $Q_\eps$ converges to a kernel
$Q$, which has a non-integrable singularity at the origin, thus preventing the corresponding
term to converge to a limit.

This is akin to the problem of making sense of integration against a one-dimensional kernel 
with a singularity of type $1/|x|$ at the origin. For the sake of the argument, let us consider a
function $W\colon \R \to \R$ which is compactly supported and smooth everywhere except at the origin,
where it diverges like $W(x) \sim 1/|x|$. It is then natural to associate to $W$ a ``renormalised'' distribution
$\Ren W$ given by
\begin{equ}
\bigl(\Ren W\bigr)(\phi) = \int W(x) \bigl(\phi(x) - \phi(0)\bigr)\,dx\;.
\end{equ}
Note that $\Ren W$ has the property that if $\phi(0) = 0$, then it simply corresponds to integration against $W$,
which is the standard way of associating a distribution to a function. In a way, the extra term can be interpreted
as subtracting a Dirac distribution with an ``infinite mass'' located at the origin, thus cancelling out the
divergence of the non-integrable singularity. It is also straightforward to verify that if $W_\eps$ is a sequence
of smooth approximations to $W$ (say one has $W_\eps(x) = W(x)$ for $|x| > \eps$ and $W_\eps \sim 1/\eps$ otherwise),
then $\Ren W^\eps \to \Ren W$ in a distributional sense, and (using the usual correspondence between functions
and distributions) one has
\begin{equ}
\Ren W^\eps = W^\eps - \hat C_\eps \delta_0\;,\qquad \hat C_\eps = \int W^\eps(x)\,dx\;.
\end{equ}
The cure to the problem we are facing for showing the convergence of $\PPi^\eps \<32>$ is
virtually identical. Indeed, by choosing 
\begin{equ}
C_3^{(\eps)} = 2\;\begin{tikzpicture}[scale=0.35,baseline=-0.1cm]
	\node at (4,0)  [int] (root) {};
	\node at (0,0)  [root] (middle) {};
	\node at (2,-1.5)  [int] (left) {};
	\node at (2,1.5)  [int] (right) {};
	
	\draw[keps,bend right=30] (left) to  (root);	
	\draw[keps,bend left=30] (left) to  (middle);	
	\draw[keps,bend left=30] (right) to  (root);	
	\draw[keps,bend right=30] (right) to  (middle);	
	\draw[kernel] (middle) to (root);
\end{tikzpicture}\;,
\end{equ}
(which is a constant as it does not depend on \tikz \node[root]{}; by translation invariance)
the term in the first homogeneous Wiener chaos for $\hat \PPi^\eps \<32>$ 
is precisely given by
\begin{equ}
6\int \bigl(\Ren \hat Q_\eps * K_\eps\bigr)(z- z_2)\xi(z_2)\,dz_2\;.
\end{equ}
It turns out that the convergence of $\Ren \hat Q_\eps$ to a limiting distribution $\Ren \hat Q$ takes
place in a sufficiently strong topology to allow to conclude that $\hat \PPi^\eps \<32>$ 
does indeed converge to a non-trivial limiting random distribution.

It should be clear from this whole discussion that while the precise values of the constants $C_1^{(\eps)}$
and $C_3^{(\eps)}$ depend on the shape of the mollifier $\rho$, the limiting (random)
model $(\hat \Pi,\hat \Gamma)$ obtained in this way is independent of it. Combining this with the
continuity of the solution to the fixed point map \eref{e:abstractFull} and of the reconstruction
operator $\CR$ with respect to the underlying model, the conclusion of Theorem~\ref{theo:construction}
follows.

\section{Convergence of other smooth models to \texorpdfstring{$\Phi^4_3$}{Phi43}}

In this final section, we give a short overview of the ideas involved in showing
that different kinds of smooth models (and not just the simplest one obtained by
hitting the noise with a smooth mollifier) also converge to the same process. 

\subsection{Non-Gaussian approximations}
\label{sec:CLT}

We first discuss how the argument sketched in the previous section 
can be modified to deal with the 
case when $\xi_\eps$ does no longer denote some simple mollification of $\xi$, but instead
denotes a suitable rescaling of an arbitrary stationary space-time stochastic process, i.e.\ 
one has
\begin{equ}
\xi_\eps(t,x) = \eps^{-{5\over 2}} \eta(t/\eps^2, x/\eps)\;,
\end{equ}
for a smooth stationary process $\eta$ on $\R\times \R^3$. As already mentioned
in the introduction, it is then possible to find choices of constants $C_\eps^{(i)}$
such that solutions to
\begin{equ}
\d_t \Phi_\eps = \Delta \Phi_\eps + C_\eps^{(1)} + C_\eps^{(2)} \Phi_\eps - \Phi_\eps^3 + \xi_\eps \;,
\end{equ}
converge as $\eps \to 0$ to the process $\Phi$ built in Theorem~\ref{theo:construction}.
In view of Proposition~\ref{prop:renormEquation}, this is again a consequence of the fact that
the renormalised models $\hat \PPi^\eps$ built from $\xi_\eps$ as in Section~\ref{sec:convGauss}
converge to the same limit as before, provided that the renormalisation
constants $C_i^{(\eps)}$ are suitably chosen. The purpose of this section is to show
why all of the constants $C_i^{(\eps)}$ with $i\in \{1,\ldots,5\}$ should in general
be set to non-zero values.

We essentially follow the arguments
of \cite{HaoShen} and we will completely ignore issues related to the fact
that one wishes to consider the limiting equation on a torus rather than the whole space.
We assume that the $\sigma$-fields generated by point evaluations of $\eta$ in any two regions 
separated by a distance of at least $1$ are independent.
The kind of example for $\eta$ one should have in mind is 
\begin{equ}
\eta(z) = \int \phi(z-\bar z) \,\P_\mu(d\bar z, d\phi)\;,
\end{equ}
where $\P_\mu$ denotes a compensated Poisson point process on $\R^4\times \CC_0^\infty$ with intensity 
measure $\mu(d\bar z, d\phi) = d\bar z\,\nu(d\phi)$, where $\nu$ is a finite measure charging
only functions with support contained in the unit ball and satisfying suitable moment conditions
to guarantee that $\eta$ admits moments of all orders.

Since $\xi_\eps$ is no longer Gaussian, we no longer have Wiener chaos decomposition at our disposal. 
However, there exists an analogue to the Wick product for arbitrary collections of random variables,
and this will be sufficient for our purpose. 
In order to introduce the Wick product, we first need to introduce joint cumulants.
Given a collection of random variables $\CX = \{X_\alpha\}_{\alpha \in A}$ for some index set $A$,
and a subset $B \subset A$, we write $X_B \subset \CX$ and $X^B$ as shorthands for
\begin{equ}
X_B = \{X_\alpha\,:\, \alpha \in B\} \;,\qquad X^B = \prod_{\alpha\in B}X_\alpha\;.
\end{equ}
Given a finite set $B$, we furthermore write $\CP(B)$ for the collection of all partitions of $B$,
i.e.\ if $B$ is non-empty, then $\CP(B)$ consists of 
all sets $\pi \subset \powerset(B)$ 
(the set of subsets of $B$) such that $\bigcup \pi = B$, $\emptyset \in \pi$,
and such that any two distinct elements of $\pi$ are disjoint. 
With this definition, we in particular have $\CP(\emptyset) = \{\{\emptyset\}\}$.
The joint cumulant $\Cum \big(X_{B}\big)$ for a collection $X_B$ of random variables
is then given by the following defining property.

\begin{definition} \label{def:cumu}
Given a collection $\CX$ of random variables as above and any non-empty finite set 
$B \subset A$, one has
\begin{equ} [e:mome2cumu]
	\E \big( X^B\big)
	= \sum_{\pi \in \CP(B)} \prod_{\bar B\in\pi} \Cum \big(X_{\bar B}\big)\;.
\end{equ}
\end{definition}

Note that this definition naturally enforces 
the convention that $\Cum \big(X_{\emptyset}\big) = 1$.
We can now define the Wick product of an arbitrary finite collection of random variables:

\begin{definition} \label{def:Wick-product}
The Wick product 
$\Wick{X_B}$ for $B \subset A$ is defined recursively by 
\begin{equ} [e:defWick]
X^A  = 
\sum_{B \subset A}  \Wick{X_B}  \sum_{\pi \in \CP(A \setminus B)} 
	  \prod_{\bar B \in \pi}
		 \Cum \big(X_{\bar B}\big)    \;.
\end{equ}
\end{definition}

Similarly to above, this enforces the natural convention that $\E \Wick{X_{\emptyset}} = 1$.
In other words, to turn a product of random variables into a sum of Wick products, one
now not only considers all possible ways of replacing pairs of random variables by their
covariance, but all possible ways of replacing subsets of random variables by their
joint cumulants. 

Denote now by $\kappa_p^{(\eps)}$ the $p$th joint cumulant of $\xi_\eps$, namely
\begin{equ}
\kappa_p^{(\eps)}(z_1,\ldots,z_p) = \E_c \bigl(\{\xi_\eps(z_1),\ldots,\xi_\eps(z_p)\}\bigr)\;.
\end{equ}
By translation invariance, this function depends only on the differences between
its $p$ arguments. Furthermore, it satisfies the scaling reation
\begin{equ}
\kappa_p^{(\eps)}(z_1,\ldots,z_p)
= \eps^{-{5p\over 2}} \kappa_p(S_\eps z_1,\ldots,S_\eps z_p)\;,
\end{equ}
where $\kappa_p$ denotes the $p$th joint cumulant of $\eta$ and the scaling operator 
$S_\eps$ acts on $\R^4$ as $S_\eps(t,x) = (t/\eps^2,x/\eps)$.
We then use the same graphical notation as in \cite{HaoShen} for these cumulants:
a red $p$-gon for instance 
\begin{tikzpicture}[baseline=-4]
\node		(mid)  	at (0,0) {};
	\node[cumu4]	(mid-cumu) 	at (mid) {};
\node[int] at (mid.north west) {};
\node[int] at (mid.south west) {};
\node[int] at (mid.north east) {};
\node[int] at (mid.south east) {};
\end{tikzpicture}
for $p=4$, 
represents the cumulant function $\kappa_p^{(\eps)}(z_1,\ldots,z_p)$, with $z_i$ given by the
$p$ integration variables represented as before by the $p$ black dots.
This time furthermore, \tikz \node[var]{}; denotes a Wick factor of $\xi_\eps$ instead of $\xi$.

As a consequence of Definitions~\ref{def:cumu} and \ref{def:Wick-product}, we then 
have for example the identity
\begin{equ}
\bigl(\PPi^\eps \<3>\bigr)(z) = 
\begin{tikzpicture}[scale=0.35,baseline=0cm]
	\node at (0,-1)  [root] (root) {};
	\node at (-1.5,1.5)  [var] (left) {};
	\node at (0,1.8)  [var] (middle) {};
	\node at (1.5,1.5)  [var] (right) {};
	
	\draw[kernel] (left) to  (root);	
	\draw[kernel] (right) to (root);
	\draw[kernel] (middle) to (root);
\end{tikzpicture}\;
 + 3\;
\begin{tikzpicture}[scale=0.35,baseline=0cm]
	\node at (0,-1)  [root] (root) {};
	\node[cumu2n] at (-1,1.6) (top) {};
	\node at (1,1.6)  [var] (topr) {};

	\draw[cumu2,rotate=20] (top) ellipse (20pt and 10pt);
	
	\draw[kernel,bend right = 60] ($(top)+(200:3mm)$) node[int]{} to  (root);	
	\draw[kernel,bend left = 60] ($(top)+(20:3mm)$) node[int]{} to (root);
	\draw[kernel,bend left = 60] (topr) to (root);
\end{tikzpicture}
\;+\;
\begin{tikzpicture}  [scale=0.35,baseline=0cm] 
\node[root]	(root) 	at (0,-1) {};
\node at (0,1.5) (a) {};  
	\node[cumu3,rotate=180] at (a) {};
\draw[kernel] ($(a)+(30:4mm)$) node[int]   {} to [bend left = 60] (root);
\draw[kernel] ($(a)+(150:4mm)$) node[int]   {} to [bend right = 60] (root);
\draw[kernel] ($(a)+(-90:4mm)$) node[int]   {} to (root);
\end{tikzpicture}
\;.
\end{equ}
A powercounting argument shows that this additional term also diverges
as $\eps \to 0$, which suggests that in this case we should no longer set
$C_2^{(\eps)} = 0$ as in the Gaussian case but instead we should choose the
relevant renormalisation constants as follows:
\begin{equ}[e:renorm1]
C_1^{(\eps)} = \;
\begin{tikzpicture}  [baseline=10] 
\node[root]	(root) 	at (0,0) {};
\node[cumu2n]	(a)  		at (0,1) {};	
\draw[cumu2] (a) ellipse (8pt and 4pt);
\draw[kernel] (a.east) node[int]   {} to [bend left = 60] (root);
\draw[kernel] (a.west) node[int]   {} to [bend right = 60] (root);
\end{tikzpicture}
\; \sim \eps^{-1}\;,\qquad 
C_2^{(\eps)} = 
\begin{tikzpicture}  [scale=0.35,baseline=0cm] 
\node[root]	(root) 	at (0,-1) {};
\node at (0,1.5) (a) {};  
	\node[cumu3,rotate=180] at (a) {};
\draw[kernel] ($(a)+(30:4mm)$) node[int]   {} to [bend left = 60] (root);
\draw[kernel] ($(a)+(150:4mm)$) node[int]   {} to [bend right = 60] (root);
\draw[kernel] ($(a)+(-90:4mm)$) node[int]   {} to (root);
\end{tikzpicture} \sim \eps^{-3/2}\;.
\end{equ}
This choice of renormalisation constant then yields just as in the Gaussian case 
\begin{equ}
\bigl(\hat \PPi^\eps \<3>\bigr)(z) = 
\begin{tikzpicture}[scale=0.35,baseline=0cm]
	\node at (0,-1)  [root] (root) {};
	\node at (-1.5,1.5)  [var] (left) {};
	\node at (0,1.8)  [var] (middle) {};
	\node at (1.5,1.5)  [var] (right) {};
	
	\draw[kernel] (left) to  (root);	
	\draw[kernel] (right) to (root);
	\draw[kernel] (middle) to (root);
\end{tikzpicture}\;,
\end{equ}
which can then be shown to converge to a finite limiting distribution 
as $\eps \to 0$.

\begin{remark}\label{rem:conv}
The fact that $\hat \PPi^\eps \<3>$ is tight in some space of distributions 
as $\eps \to 0$ can be shown quite easily by first testing against a test
function and then  bounding its second moment. Tightness in the correct topology
and identification of the correct limit is more delicate and requires control
on moments of all orders. Unfortunately, these can no longer be obtained
as a corollary of a second moment bound since the analogue to \eqref{e:hypercontract} 
does not hold anymore in general.

We will henceforth ignore these difficulties and admit without proof the 
fact that if $f$ is an $L^2$ kernel of $k$ variables
and $f_\eps$ is a sequence of ``sufficiently nice'' kernels converging to $f$, 
then the sequence of random variables
\begin{equ}
\int f_\eps(z_1,\ldots ,z_k)\, \Wick{\xi_\eps(z_1)\cdots \xi_\eps(z_k)}dz_1\cdots dz_k\;,
\end{equ}
converges weakly to $I_k(f)$ as $\eps \to 0$. This convergence furthermore
holds jointly for any finite collection of such random variables.
See \cite{HaoShen} for more details.
\end{remark}

The necessity of the remaining renormalisation constants can again be traced 
back to the analysis of $\hat \PPi^\eps \<32>$. This time, one has the 
much lengthier identity
\begin{equs}
\PPi^\eps \<32> &= 
\begin{tikzpicture}[scale=0.35,baseline=0cm]
	\node at (0,-1)  [root] (root) {};
	\node at (0,1.5)  [int] (middle) {};
	\node at (-1.6,1)  [var] (left) {};
	\node at (1.6,1)  [var] (right) {};
	\node at (-1.5,3.5)  [var] (tl) {};
	\node at (1.5,3.5)  [var] (tr) {};
	\node at (0,4)  [var] (tm) {};
	
	\draw[kernel] (left) to  (root);	
	\draw[kernel] (right) to (root);
	\draw[kernel] (tl) to  (middle);	
	\draw[kernel] (tr) to (middle);
	\draw[kernel] (tm) to  (middle);	
	\draw[kernel] (middle) to (root);
\end{tikzpicture}\; + 
6\;
\begin{tikzpicture}[scale=0.35,baseline=0cm]
	\node at (0,-1)  [root] (root) {};
	\node at (0,1.5)  [int] (middle) {};

	\node[cumu2n] at (-1.6,.25) (left) {};
	\draw[cumu2] (left) ellipse (10pt and 20pt);

	\node at (1.6,1)  [var] (right) {};
	\node at (1.5,3.5)  [var] (tr) {};
	\node at (0,4)  [var] (tm) {};
	
	\draw[kernel] (right) to (root);
	\draw[kernel,bend right=40] (left.south) node[int]{} to  (root);	
	\draw[kernel,bend left=40] (left.north) node[int]{} to  (middle);	
	\draw[kernel] (tr) to (middle);
	\draw[kernel] (tm) to  (middle);	
	\draw[kernel] (middle) to (root);
\end{tikzpicture}\; + 
3\;
\begin{tikzpicture}[scale=0.35,baseline=0cm]
	\node at (0,-1)  [root] (root) {};
	\node at (0,1.5)  [int] (middle) {};
	\node at (-1.6,1)  [var] (left) {};
	\node at (1.6,1)  [var] (right) {};
	\node[cumu2n] at (-1,3.6) (tl) {};
	\draw[cumu2,rotate=25] (tl) ellipse (20pt and 10pt);
	\node at (1,3.6)  [var] (tr) {};
	
	\draw[kernel] (left) to  (root);	
	\draw[kernel] (right) to (root);
	\draw[kernel,bend right = 60] ($(tl)+(205:3mm)$) node[int]{} to  (middle);	
	\draw[kernel,bend left = 60] ($(tl)+(25:3mm)$) node[int]{} to (middle);
	\draw[kernel,bend left = 60] (tr) to (middle);
	\draw[kernel] (middle) to (root);
\end{tikzpicture}\; + \;
\begin{tikzpicture}[scale=0.35,baseline=0cm]
	\node at (0,-1)  [root] (root) {};
	\node at (1,1.6)  [int] (middle) {};
	\node[cumu2n] at (-1,1.6) (left) {};
	\draw[cumu2,rotate=25] (left) ellipse (20pt and 10pt);
	\node at (-0.5,3.5)  [var] (tl) {};
	\node at (2.5,3.5)  [var] (tr) {};
	\node at (1,4)  [var] (tm) {};

	\draw[kernel,bend right = 60] ($(left)+(205:3mm)$) node[int]{} to  (root);	
	\draw[kernel,bend left = 60] ($(left)+(25:3mm)$) node[int]{} to (root);
	\draw[kernel,bend left = 60] (middle) to (root);
	
	\draw[kernel] (tl) to  (middle);	
	\draw[kernel] (tr) to (middle);
	\draw[kernel] (tm) to  (middle);	
\end{tikzpicture}\; + \;
\begin{tikzpicture}[scale=0.35,baseline=0cm]
	\node at (0,-1)  [root] (root) {};
	\node at (0,1.5)  [int] (middle) {};
	\node at (-1.6,1)  [var] (left) {};
	\node at (1.6,1)  [var] (right) {};
    \node at (0,4) (a) {};  
    	\node[cumu3,rotate=180] at (a) {};
    \draw[kernel] ($(a)+(30:4mm)$) node[int]   {} to [bend left = 60] (middle);
    \draw[kernel] ($(a)+(150:4mm)$) node[int]   {} to [bend right = 60] (middle);
    \draw[kernel] ($(a)+(-90:4mm)$) node[int]   {} to (middle);
	
	\draw[kernel] (left) to  (root);	
	\draw[kernel] (right) to (root);
	\draw[kernel] (middle) to (root);
\end{tikzpicture}
\\ 
&\qquad + 6\;
\begin{tikzpicture}[scale=0.35,baseline=0cm]
	\node at (0,-1)  [root] (root) {};
	\node at (0,1.5)  [int] (middle) {};
	\node at (-1.6,1)  [var] (left) {};
	\node at (-1.5,3.5)  [var] (tl) {};

    \node at (2.5,1.5) (a) {};  
    \node[cumu3,rotate=30] at (a) {};
    \draw[kernel] ($(a)+(0:4mm)$) node[int]   {} to [bend left = 60] (root);
    \draw[kernel] ($(a)+(120:4mm)$) node[int]   {} to [bend right = 60] (middle);
    \draw[kernel] ($(a)+(-120:4mm)$) node[int]   {} to [bend left = 60] (middle);
	
	\draw[kernel] (left) to  (root);	
	\draw[kernel] (tl) to  (middle);	
	\draw[kernel] (middle) to (root);
\end{tikzpicture}\; + 3\;
\begin{tikzpicture}[scale=0.35,baseline=0cm]
	\node at (1,-1)  [root] (root) {};
	\node at (0,1.5)  [int] (middle) {};
	\node at (0.5,4)  [var] (tm) {};
	\node at (-1.5,3.5)  [var] (tl) {};

    \node at (2.5,2) (a) {};  
    \node[cumu3,rotate=-30] at (a) {};
    \draw[kernel] ($(a)+(-60:4mm)$) node[int]   {} to [bend left = 60] (root);
    \draw[kernel] ($(a)+(180:4mm)$) node[int]   {} to [bend right = 60] (root);
    \draw[kernel] ($(a)+(60:4mm)$) node[int]   {} to [bend right = 60] (middle);
	
	\draw[kernel,bend right=30] (tl) to  (middle);	
	\draw[kernel,bend right=35] (tm) to  (middle);	
	\draw[kernel,bend right=40] (middle) to (root);
\end{tikzpicture}
\; + 6\;
\begin{tikzpicture}[scale=0.35,baseline=0cm]
	\node at (0,-1)  [root] (root) {};
	\node at (0,1.5)  [int] (middle) {};
	\node[cumu2n] at (-1.6,.25) (left) {};
	\draw[cumu2] (left) ellipse (10pt and 20pt);

    \node at (2.5,1.5) (a) {};  
    \node[cumu3,rotate=30] at (a) {};
    \draw[kernel] ($(a)+(0:4mm)$) node[int]   {} to [bend left = 60] (root);
    \draw[kernel] ($(a)+(120:4mm)$) node[int]   {} to [bend right = 60] (middle);
    \draw[kernel] ($(a)+(-120:4mm)$) node[int]   {} to [bend left = 60] (middle);
	
	\draw[kernel,bend right=40] (left.south) node[int]{} to  (root);	
	\draw[kernel,bend left=40] (left.north) node[int]{} to  (middle);	
	\draw[kernel,bend right=30] (middle) to (root);
\end{tikzpicture}
\; + 3\;
\begin{tikzpicture}[scale=0.35,baseline=0cm]
	\node at (1,-1)  [root] (root) {};
	\node at (0,1.5)  [int] (middle) {};
	\node[cumu2n] at (-2,3.5) (tl) {};
	\draw[cumu2,rotate=45] (tl) ellipse (20pt and 10pt);

    \node at (2.5,2) (a) {};  
    \node[cumu3,rotate=-30] at (a) {};
    \draw[kernel] ($(a)+(-60:4mm)$) node[int]   {} to [bend left = 60] (root);
    \draw[kernel] ($(a)+(180:4mm)$) node[int]   {} to [bend right = 60] (root);
    \draw[kernel] ($(a)+(60:4mm)$) node[int]   {} to [bend right = 60] (middle);
	
	\draw[kernel,bend left = 60] ($(tl)+(45:3mm)$) node[int]{} to  (middle);	
	\draw[kernel,bend right = 60] ($(tl)+(225:3mm)$) node[int]{} to (middle);
	\draw[kernel,bend right=40] (middle) to (root);
\end{tikzpicture}
\\ &\qquad
 + \;
\begin{tikzpicture}[scale=0.35,baseline=0cm]
	\node at (-1,-1)  [root] (root) {};
	\node at (0,1.5)  [int] (middle) {};
	\node[cumu2n] at (-2,1.6) (left) {};
	\draw[cumu2,rotate=25] (left) ellipse (20pt and 10pt);
    \node at (0,4) (a) {};  
    	\node[cumu3,rotate=180] at (a) {};
    \draw[kernel] ($(a)+(30:4mm)$) node[int]   {} to [bend left = 60] (middle);
    \draw[kernel] ($(a)+(150:4mm)$) node[int]   {} to [bend right = 60] (middle);
    \draw[kernel] ($(a)+(-90:4mm)$) node[int]   {} to (middle);
	
	\draw[kernel,bend right = 60] ($(left)+(205:3mm)$) node[int]{} to  (root);	
	\draw[kernel,bend left = 60] ($(left)+(25:3mm)$) node[int]{} to (root);
	\draw[kernel,bend left=60] (middle) to (root);
\end{tikzpicture}
\; + 3\;
\begin{tikzpicture}[scale=0.35,baseline=0cm]
	\node at (0,-1)  [root] (root) {};
	\node at (1,1.6)  [int] (middle) {};
	\node[cumu2n] at (-1,1.6) (left) {};
	\draw[cumu2,rotate=25] (left) ellipse (20pt and 10pt);
	\node[cumu2n] at (0,3.6) (tl) {};
	\draw[cumu2,rotate=25] (tl) ellipse (20pt and 10pt);
	\node at (2,3.6)  [var] (tr) {};

	\draw[kernel,bend right = 60] ($(left)+(205:3mm)$) node[int]{} to  (root);	
	\draw[kernel,bend left = 60] ($(left)+(25:3mm)$) node[int]{} to (root);
	\draw[kernel,bend left = 60] (middle) to (root);
	
	\draw[kernel,bend right = 60] ($(tl)+(205:3mm)$) node[int]{} to  (middle);	
	\draw[kernel,bend left = 60] ($(tl)+(25:3mm)$) node[int]{} to (middle);
	\draw[kernel,bend left = 60] (tr) to (middle);
\end{tikzpicture}\; + 6\;
\begin{tikzpicture}[scale=0.35,baseline=0cm]
	\node at (0,-1)  [root] (root) {};
	\node at (0,1.5)  [int] (middle) {};
	\node[cumu2n] at (-1.6,.25) (left) {};
	\draw[cumu2] (left) ellipse (10pt and 20pt);
	\node at (1.6,1)  [var] (right) {};
	\node[cumu2n] at (0,3.5) (top) {};
	\draw[cumu2] (top) ellipse (20pt and 10pt);
	
	\draw[kernel,bend right=40] (left.south) node[int]{} to  (root);	
	\draw[kernel] (right) to (root);
	\draw[kernel,bend left=40] (left.north) node[int]{} to  (middle);	
	\draw[kernel,bend right = 60] (top.west) node[int]{} to (middle);
	\draw[kernel,bend left = 60] (top.east) node[int]{} to  (middle);	
	\draw[kernel] (middle) to (root);
\end{tikzpicture}\; + 6\;
\begin{tikzpicture}[scale=0.35,baseline=0cm]
	\node at (0,-1)  [root] (root) {};
	\node at (0,1.5)  [int] (middle) {};
	\node[cumu2n] at (-1.6,.25) (left) {};
	\draw[cumu2] (left) ellipse (10pt and 20pt);
	\node[cumu2n] at (1.6,.25) (right) {};
	\draw[cumu2] (right) ellipse (10pt and 20pt);
	\node at (0,4)  [var] (tm) {};
	
	\draw[kernel,bend right=40] (left.south) node[int]{} to  (root);	
	\draw[kernel,bend left=40] (left.north) node[int]{} to  (middle);	
	\draw[kernel,bend left=40] (right.south) node[int]{} to  (root);	
	\draw[kernel,bend right=40] (right.north) node[int]{} to  (middle);	
	\draw[kernel] (tm) to  (middle);	
	\draw[kernel] (middle) to (root);
\end{tikzpicture}
\\ &\qquad 
 + 2\;
\begin{tikzpicture}[scale=0.35,baseline=0cm]
	\node at (0,-1)  [root] (root) {};
	\node at (0,1.5)  [int] (middle) {};
	\node at (-1.6,1)  [var] (left) {};

    \node at (1.5,3) (a) {};  
    \node[cumu4] at (a) {};
    \draw[kernel] (a.north east) node[int]   {} to [out=0,in=0] (root);
    \draw[kernel] (a.south west) node[int]   {} to (middle);
    \draw[kernel] (a.north west) node[int]   {} to [bend right = 60] (middle);
    \draw[kernel] (a.south east) node[int]   {} to [bend left = 60] (middle);
	
	\draw[kernel] (left) to  (root);	
	\draw[kernel] (middle) to (root);
\end{tikzpicture}
\;+ 3\;
\begin{tikzpicture}[scale=0.35,baseline=0cm]
	\node at (0,-1)  [root] (root) {};
	\node at (0,1.5)  [int] (middle) {};

    \node at (1.5,.25) (a) {};  
    \node[cumu4] at (a) {};
	\node at (0,4)  [var] (tm) {};
	
	\draw[kernel,bend right=40] (a.south west) node[int]{} to  (root);	
	\draw[kernel,bend left=40] (a.north west) node[int]{} to  (middle);	
	\draw[kernel,bend left=40] (a.south east) node[int]{} to  (root);	
	\draw[kernel,bend right=40] (a.north east) node[int]{} to  (middle);	
	\draw[kernel] (tm) to  (middle);	
	\draw[kernel, bend right=60] (middle) to (root);
\end{tikzpicture}
\; + \;
\begin{tikzpicture}[scale=0.35,baseline=0cm]
	\node at (0,-1)  [root] (root) {};
	\node at (0,1)  [int] (middle) {};

    \node at (0,3) (a) {};  
    \node[cumu5,rotate=36] at (a) {};
    \draw[kernel] ($(a)+(-90:5mm)$) node[int]   {} to (middle);
    \draw[kernel] ($(a)+(-18:5mm)$) node[int]   {} to [bend left = 40] (middle);
    \draw[kernel] ($(a)+(-162:5mm)$) node[int]   {} to [bend right = 40] (middle);
    \draw[kernel] ($(a)+(54:5mm)$) node[int]   {} to [bend left = 70] (root);
    \draw[kernel] ($(a)+(126:5mm)$) node[int]   {} to [bend right = 70] (root);
	
	\draw[kernel] (middle) to (root);
\end{tikzpicture}
\;.
\end{equs}
At this stage, it starts to become clear that
on should in principle start to use a more systematic approach,
but for the sake of this article, and in order to give a feeling for the kind
of bounds involved in the proofs, our more ``brutal'' approach is sufficient.

In order to bound $\PPi^\eps \<32>$ we first note that as a consequence of the
fact that $K$ integrates to $0$, both the last term on the first line and
the first term on the third line of the above expression vanish.
This time, taking into account the fact that all of the $C_i^{(\eps)}$ can be chosen
to be non-zero, we obtain the identity
\begin{equ}
\hat \PPi^\eps \<32> = \PPi^\eps\<32> - 3C_1^{(\eps)} \PPi^\eps \<12> - C_1^{(\eps)} \PPi^\eps \<30>
+ 3 (C_1^{(\eps)})^2 \PPi^\eps \<10> - 3 C_3^{(\eps)} \PPi^\eps\<1>
- C_4^{(\eps)}\PPi^\eps\1 - 2C_5^{(\eps)}\PPi^\eps\<1> \;.
\end{equ}
We claim that besides \eqref{e:renorm1}, the correct choice of renormalisation constants
in order to obtain the same limiting model as in Section~\ref{sec:convGauss} is given by
\begin{equs}
C_3^{(\eps)} 
= 
2\;
\begin{tikzpicture}[scale=0.35,baseline=0cm]
	\node at (0,-1)  [root] (root) {};
	\node at (0,3)  [int] (middle) {};
	\node[cumu2n] at (-1.6,1) (left) {};
	\draw[cumu2] (left) ellipse (10pt and 20pt);
	\node[cumu2n] at (1.6,1) (right) {};
	\draw[cumu2] (right) ellipse (10pt and 20pt);
	
	\draw[kernel,bend right=40] (left.south) node[int]{} to  (root);	
	\draw[kernel,bend left=40] (left.north) node[int]{} to  (middle);	
	\draw[kernel,bend left=40] (right.south) node[int]{} to  (root);	
	\draw[kernel,bend right=40] (right.north) node[int]{} to  (middle);	
	\draw[kernel] (middle) to (root);
\end{tikzpicture}
\;+
\begin{tikzpicture}[scale=0.35,baseline=0cm]
	\node at (0,-1)  [root] (root) {};
	\node at (0,3)  [int] (middle) {};

    \node at (1.5,1) (a) {};  
    \node[cumu4] at (a) {};
	
	\draw[kernel,bend right=40] (a.south west) node[int]{} to  (root);	
	\draw[kernel,bend left=40] (a.north west) node[int]{} to  (middle);	
	\draw[kernel,bend left=40] (a.south east) node[int]{} to  (root);	
	\draw[kernel,bend right=40] (a.north east) node[int]{} to  (middle);	
	\draw[kernel, bend right=60] (middle) to (root);
\end{tikzpicture}
\;,\quad
C_4^{(\eps)} 
= 
\begin{tikzpicture}[scale=0.35,baseline=0cm]
	\node at (0,-1)  [root] (root) {};
	\node at (0,1)  [int] (middle) {};

    \node at (0,3) (a) {};  
    \node[cumu5,rotate=36] at (a) {};
    \draw[kernel] ($(a)+(-90:5mm)$) node[int]   {} to (middle);
    \draw[kernel] ($(a)+(-18:5mm)$) node[int]   {} to [bend left = 40] (middle);
    \draw[kernel] ($(a)+(-162:5mm)$) node[int]   {} to [bend right = 40] (middle);
    \draw[kernel] ($(a)+(54:5mm)$) node[int]   {} to [bend left = 70] (root);
    \draw[kernel] ($(a)+(126:5mm)$) node[int]   {} to [bend right = 70] (root);
	
	\draw[kernel] (middle) to (root);
\end{tikzpicture}
 + 6\;
\begin{tikzpicture}[scale=0.35,baseline=0cm]
	\node at (0,-1)  [root] (root) {};
	\node at (0,1.5)  [int] (middle) {};
	\node[cumu2n] at (-1.6,.25) (left) {};
	\draw[cumu2] (left) ellipse (10pt and 20pt);

    \node at (2.5,1.5) (a) {};  
    \node[cumu3,rotate=30] at (a) {};
    \draw[kernel] ($(a)+(0:4mm)$) node[int]   {} to [bend left = 60] (root);
    \draw[kernel] ($(a)+(120:4mm)$) node[int]   {} to [bend right = 60] (middle);
    \draw[kernel] ($(a)+(-120:4mm)$) node[int]   {} to [bend left = 60] (middle);
	
	\draw[kernel,bend right=40] (left.south) node[int]{} to  (root);	
	\draw[kernel,bend left=40] (left.north) node[int]{} to  (middle);	
	\draw[kernel,bend right=30] (middle) to (root);
\end{tikzpicture}
\;,\quad
C_5^{(\eps)}
=
\begin{tikzpicture}[scale=0.35,baseline=0cm]
	\node at (0,-1)  [root] (root) {};
	\node at (0,1.5)  [int] (middle) {};

    \node at (1.5,3) (a) {};  
    \node[cumu4] at (a) {};
    \draw[kernel] (a.north east) node[int]   {} to [out=0,in=0] (root);
    \draw[kernel] (a.south west) node[int]   {} to (middle);
    \draw[kernel] (a.north west) node[int]   {} to [bend right = 60] (middle);
    \draw[kernel] (a.south east) node[int]   {} to [bend left = 60] (middle);
	
	\draw[kernel,bend right=30] (middle) to (root);
\end{tikzpicture}
\;.
\end{equs}
With this choice, a tedious but straightforward graphical calculation yields the
identity
\begin{equs}
\hat \PPi^\eps \<32> &= 
\begin{tikzpicture}[scale=0.35,baseline=0cm]
	\node at (0,-1)  [root] (root) {};
	\node at (0,1.5)  [int] (middle) {};
	\node at (-1.6,1)  [var] (left) {};
	\node at (1.6,1)  [var] (right) {};
	\node at (-1.5,3.5)  [var] (tl) {};
	\node at (1.5,3.5)  [var] (tr) {};
	\node at (0,4)  [var] (tm) {};
	
	\draw[kernel] (left) to  (root);	
	\draw[kernel] (right) to (root);
	\draw[kernel] (tl) to  (middle);	
	\draw[kernel] (tr) to (middle);
	\draw[kernel] (tm) to  (middle);	
	\draw[kernel] (middle) to (root);
\end{tikzpicture}\; + 
6\;
\begin{tikzpicture}[scale=0.35,baseline=0cm]
	\node at (0,-1)  [root] (root) {};
	\node at (0,1.5)  [int] (middle) {};

	\node[cumu2n] at (-1.6,.25) (left) {};
	\draw[cumu2] (left) ellipse (10pt and 20pt);

	\node at (1.6,1)  [var] (right) {};
	\node at (1.5,3.5)  [var] (tr) {};
	\node at (0,4)  [var] (tm) {};
	
	\draw[kernel] (right) to (root);
	\draw[kernel,bend right=40] (left.south) node[int]{} to  (root);	
	\draw[kernel,bend left=40] (left.north) node[int]{} to  (middle);	
	\draw[kernel] (tr) to (middle);
	\draw[kernel] (tm) to  (middle);	
	\draw[kernel] (middle) to (root);
\end{tikzpicture}
\;
 + 6\;
\begin{tikzpicture}[scale=0.35,baseline=0cm]
	\node at (0,-1)  [root] (root) {};
	\node at (0,1.5)  [int] (middle) {};
	\node[cumu2n] at (-1.6,.25) (left) {};
	\draw[cumu2] (left) ellipse (10pt and 20pt);
	\node[cumu2n] at (1.6,.25) (right) {};
	\draw[cumu2] (right) ellipse (10pt and 20pt);
	\node at (0,4)  [var] (tm) {};
	
	\draw[kernel,bend right=40] (left.south) node[int]{} to  (root);	
	\draw[kernel,bend left=40] (left.north) node[int]{} to  (middle);	
	\draw[kernel,bend left=40] (right.south) node[int]{} to  (root);	
	\draw[kernel,bend right=40] (right.north) node[int]{} to  (middle);	
	\draw[kernel] (tm) to  (middle);	
	\draw[kernel] (middle) to (root);
\end{tikzpicture}
\; - 6\;
\begin{tikzpicture}[scale=0.35,baseline=0cm]
	\node at (0,-1)  [root] (root) {};
	\node at (0,1.5)  [int] (middle) {};
	\node[cumu2n] at (-1.6,.25) (left) {};
	\draw[cumu2] (left) ellipse (10pt and 20pt);
	\node[cumu2n] at (1.6,.25) (right) {};
	\draw[cumu2] (right) ellipse (10pt and 20pt);
	\node at (0,4)  [var] (tm) {};
	
	\draw[kernel,bend right=40] (left.south) node[int]{} to  (root);	
	\draw[kernel,bend left=40] (left.north) node[int]{} to  (middle);	
	\draw[kernel,bend left=40] (right.south) node[int]{} to  (root);	
	\draw[kernel,bend right=40] (right.north) node[int]{} to  (middle);	
	\draw[kernel] (tm) .. controls +(right:2.8cm) and +(-30:3.5cm) .. (root);	
	\draw[kernel] (middle) to (root);
\end{tikzpicture}\label{e:Pihateps32}
\\ &\qquad
 + 6\;
\begin{tikzpicture}[scale=0.35,baseline=0cm]
	\node at (0,-1)  [root] (root) {};
	\node at (0,1.5)  [int] (middle) {};
	\node at (-1.6,1)  [var] (left) {};
	\node at (-1.5,3.5)  [var] (tl) {};

    \node at (2.5,1.5) (a) {};  
    \node[cumu3,rotate=30] at (a) {};
    \draw[kernel] ($(a)+(0:4mm)$) node[int]   {} to [bend left = 60] (root);
    \draw[kernel] ($(a)+(120:4mm)$) node[int]   {} to [bend right = 60] (middle);
    \draw[kernel] ($(a)+(-120:4mm)$) node[int]   {} to [bend left = 60] (middle);
	
	\draw[kernel] (left) to  (root);	
	\draw[kernel] (tl) to  (middle);	
	\draw[kernel] (middle) to (root);
\end{tikzpicture}
\; + 3\;
\begin{tikzpicture}[scale=0.35,baseline=0cm]
	\node at (1,-1)  [root] (root) {};
	\node at (0,1.5)  [int] (middle) {};
	\node at (0.5,4)  [var] (tm) {};
	\node at (-1.5,3.5)  [var] (tl) {};

    \node at (2.5,2) (a) {};  
    \node[cumu3,rotate=-30] at (a) {};
    \draw[kernel] ($(a)+(-60:4mm)$) node[int]   {} to [bend left = 60] (root);
    \draw[kernel] ($(a)+(180:4mm)$) node[int]   {} to [bend right = 60] (root);
    \draw[kernel] ($(a)+(60:4mm)$) node[int]   {} to [bend right = 60] (middle);
	
	\draw[kernel,bend right=30] (tl) to  (middle);	
	\draw[kernel,bend right=35] (tm) to  (middle);	
	\draw[kernel,bend right=40] (middle) to (root);
\end{tikzpicture}
\; + 3\;
\begin{tikzpicture}[scale=0.35,baseline=0cm]
	\node at (0,-1)  [root] (root) {};
	\node at (0,1.5)  [int] (middle) {};

    \node at (1.5,.25) (a) {};  
    \node[cumu4] at (a) {};
	\node at (0,4)  [var] (tm) {};
	
	\draw[kernel,bend right=40] (a.south west) node[int]{} to  (root);	
	\draw[kernel,bend left=40] (a.north west) node[int]{} to  (middle);	
	\draw[kernel,bend left=40] (a.south east) node[int]{} to  (root);	
	\draw[kernel,bend right=40] (a.north east) node[int]{} to  (middle);	
	\draw[kernel] (tm) to  (middle);	
	\draw[kernel, bend right=60] (middle) to (root);
\end{tikzpicture}
\; - 3\;
\begin{tikzpicture}[scale=0.35,baseline=0cm]
	\node at (0,-1)  [root] (root) {};
	\node at (0,1.5)  [int] (middle) {};

    \node at (1.5,.25) (a) {};  
    \node[cumu4] at (a) {};
	\node at (0,4)  [var] (tm) {};
	
	\draw[kernel,bend right=40] (a.south west) node[int]{} to  (root);	
	\draw[kernel,bend left=40] (a.north west) node[int]{} to  (middle);	
	\draw[kernel,bend left=40] (a.south east) node[int]{} to  (root);	
	\draw[kernel,bend right=40] (a.north east) node[int]{} to  (middle);	
	\draw[kernel] (tm) to[bend right=90]   (root);	
	\draw[kernel, bend right=60] (middle) to (root);
\end{tikzpicture}
\;.
\end{equs}
The four terms appearing on the first line are direct analogues of the terms
appearing in the Gaussian case \eqref{e:Pihat32}. In view of Remark~\ref{rem:conv}
it is therefore at least very plausible that they converge to the same limit.
It remains to argue that the terms appearing on the last line vanish in the
limit $\eps \to 0$. For this, the following simple lemma (see \cite[Lem.~4.7]{HaoShen}
for a slightly more general version)
is useful:

\begin{lemma}  \label{lem:collapse}
One has the bound
\begin{equs} 
\int
\prod_{i=1}^n |K(z_i - \bar z_i)|
|\kappa^{(\eps)}_n (\bar z_1, \ldots,  \bar z_n)|
  \,d\bar z_1\cdots d\bar z_n
\lesssim \eps^{5(n/2-1)}
\int_{\R^4}
 	\prod_{i=1}^n  
	\big(|z_i - \bar z|+\eps\big)^{-3}
\, d\bar z \;,
\end{equs}
uniformly over $z_1,\ldots,z_n \in\R^4$ and $\eps>0$.
\end{lemma}

Recall that since we are in space dimension $3$,
both the heat kernel and its truncation $K$ satisfy the bounds
$|K(z)| \lesssim |z|^{-3}$, so that each factor appearing on the right hand
side of this bound behaves essentially like a factor $K_\eps$. 
In particular, writing
\begin{equ}
Q^{(1)}_\eps = 
\begin{tikzpicture}[scale=0.35,baseline=0cm]
	\node at (0,-1)  [root] (root) {};
	\node at (0,3)  [root] (middle) {};

    \node at (2,1) (a) {};  
    \node[cumu3,rotate=-30] at (a) {};
    \draw[kernel] ($(a)+(-60:4mm)$) node[int]   {} to [bend left = 40] (root);
    \draw[kernel] ($(a)+(180:4mm)$) node[int]   {} to [bend left = 40] (middle);
    \draw[kernel] ($(a)+(60:4mm)$) node[int]   {} to [bend right = 40] (middle);
	
	\draw[kernel,bend right=40] (middle) to (root);
\end{tikzpicture}
\;,\qquad
Q^{(2)}_\eps = 
\begin{tikzpicture}[scale=0.35,baseline=0cm]
	\node at (0,-1)  [root] (root) {};
	\node at (0,3)  [root] (middle) {};

    \node at (2,1) (a) {};  
    \node[cumu3,rotate=-30] at (a) {};
    \draw[kernel] ($(a)+(-60:4mm)$) node[int]   {} to [bend left = 40] (root);
    \draw[kernel] ($(a)+(180:4mm)$) node[int]   {} to [bend right = 40] (root);
    \draw[kernel] ($(a)+(60:4mm)$) node[int]   {} to [bend right = 40] (middle);
	
	\draw[kernel,bend right=40] (middle) to (root);
\end{tikzpicture}
\;,\qquad
Q^{(3)}_\eps = 
\begin{tikzpicture}[scale=0.35,baseline=0cm]
	\node at (0,-1)  [root] (root) {};
	\node at (0,3)  [root] (middle) {};

    \node at (1.5,1) (a) {};  
    \node[cumu4] at (a) {};
	
	\draw[kernel,bend right=40] (a.south west) node[int]{} to  (root);	
	\draw[kernel,bend left=40] (a.north west) node[int]{} to  (middle);	
	\draw[kernel,bend left=40] (a.south east) node[int]{} to  (root);	
	\draw[kernel,bend right=40] (a.north east) node[int]{} to  (middle);	
	\draw[kernel, bend right=60] (middle) to (root);
\end{tikzpicture}\;,
\end{equ}
with the same conventions as in \eqref{e:defQeps}, it follows
that the kernels $Q^{(i)}_\eps$ satisfy the bounds
\begin{equ}
|Q^{(1)}_\eps(z)| \lesssim \eps^\kappa |z|^{-{9\over 2}-\kappa}\;,\qquad
|Q^{(2)}_\eps(z)| \lesssim \eps^\kappa |z|^{-{9\over 2}-\kappa}\;,\qquad
|Q^{(3)}_\eps(z)| \lesssim \eps^\kappa (|z|+\eps)^{-5-\kappa}\;,
\end{equ}
provided that one chooses $\kappa\ge 0$ sufficiently small.
In particular, both $Q^{(1)}_\eps$ and $Q^{(2)}_\eps$ converge to $0$ in
$L^p$ for $p$ sufficiently small (but still greater than $1$), which allows to
show that the first two terms on the second line of \eqref{e:Pihateps32}
vanish in the limit.
To show that the last two terms also vanish in the limit, one needs to exploit 
cancellations between them: taken separately they converge to finite non-vanishing 
limits. However, since $Q^{(3)}_\eps$ is integrable with its integral remaining uniformly
bounded as $\eps \to 0$, it is easy to show that 
$Q^{(3)}_\eps(\cdot) - C_{3,2}^{(\eps)}\delta(\cdot)$ converges weakly to $0$, where $C_{3,2}^{(\eps)}$
denotes the second term appearing in the above definition of $C_3^{(\eps)}$.
While this in itself is not sufficient to guarantee that the sum of the last two terms
in \eqref{e:Pihateps32} indeed vanishes in the limit, it is a strong indication that
it does, and is in fact not very difficult to show. 

In a similar way, it is possible to show that $\hat \PPi^\eps \tau$ converges to the
same limit as in Section~\ref{sec:convGauss} for every symbol $\tau \in \CW$.
With some additional effort, one can then show that this convergence actually takes
place in the topology of the space of models for $(T,\CG)$ which, when combined
with the continuity statement of Theorem~\ref{theo:localSol}, the identification of
the renormalised solutions of Proposition~\ref{prop:renormEquation}, and the explicit expressions of the 
renormalisation constants, immediately implies
that the solutions to \eqref{e:solNonGauss} with constants
as in \eqref{e:constNonGauss} do indeed converge to the same
limiting object as constructed in Theorem~\ref{theo:construction}.

\subsection{Non-cubic approximations}

In this last section, we give a very short sketch of the main ideas appearing in the
proof of Theorem~\ref{theo:universal}. Recall that we are interested in the study
of \eqref{e:modelV} as the effective potential $\scal{V_\theta}$ undergoes a 
pitchfork bifurcation. In order to 
concentrate on the main ideas, we restrict ourselves to the case where $V$ is of degree $6$.
Then, normalising the equation such that
$\d_\theta\d_x^2 \scal{V_\theta}(0) = -{1\over 2}$ and $\scal{V_0}^{(4)}(0) = 6$, 
the derivative of the effective potential $\scal{V_\theta}$ is given to lowest order in $\theta$ by
\begin{equ}[e:potV]
- \scal{V_\theta}'(\Phi) = \theta \Phi - \Phi^3 - a\Phi^5\;,
\end{equ} 
for some constant $a$. From now on we neglect terms of higher order in $\theta$ and we 
consider the case of a potential with derivative equal to \eqref{e:potV}.

Note that \eqref{e:potV} is the derivative of the \textit{effective} potential. 
After rescaling, the model \eqref{e:modelV} we consider is therefore given by
\begin{equs}
\d_t \Phi_{\eps,\theta} &= \Delta \Phi_{\eps,\theta} + \theta \eps^{-1} \Phi - H_3(\Phi,C/\eps) - a \eps H_5(\Phi,C/\eps) + \xi_\eps \label{e:fullEqu}\\
&= \Delta \Phi_{\eps,\theta} + \eps^{-1}(\theta + 3C - 15aC^2) \Phi - (1+10aC)\Phi^3 - a \eps \Phi^5 + \xi_\eps \;,
\end{equs}
where $C$ is the constant appearing in \eqref{e:defCC}. In this equation, 
$H_n(x,c)$ denotes the $n$th Hermite polynomial with variance $c$, i.e.\ 
$H_3(x,c) = x^3 - 3cx$, etc. A special case of Theorem~\ref{theo:universal} is then given
by, the following

\begin{theorem}\label{theo:universalWeak}
There exist values $a > 0$ and $b \in \R$, such that if one sets 
$\theta = a \eps \log \eps + b \eps$, then the solution to \eqref{e:fullEqu} converges
as $\eps \to 0$ to the process $\Phi$ built in Theorem~\ref{theo:construction}.
\end{theorem}

The remainder of this section is devoted to a sketch of the proof of Theorem~\ref{theo:universalWeak}.
It is not clear a priori how to fit \eqref{e:fullEqu} into the 
framework we developed in this article. One might think that one could try to 
change \eqref{e:abstractFull} into something like
\begin{equ}[e:abstractFullPotential]
\Phi = \CP \one_{t>0}\bigl(\sXi + c_1\Phi - c_2\Phi^3 - a\eps \Phi^5\bigr) + P \Phi_0\;,
\end{equ}
for suitable choices of constants $c_i$, but one then immediately runs into two
(related) problems:
\begin{claim}
\item In order to interpret \eqref{e:abstractFullPotential} as an equation in $\CD^\gamma$ for
some regularity structure, we would naturally have to replace the recursive step \eqref{e:induction} by
something like 
\begin{equ}
\tau_i \in \CU \quad\Rightarrow\quad \CI(\tau_1\tau_2\tau_3\tau_4\tau_5) \in \CU\;.
\end{equ}
The problem then is that the resulting collection of symbols no longer has the property that,
for every $\gamma \in \R$, there are only finitely many symbols of homogeneity at most $\gamma$.
This pretty much destroys the rest of the argument.
\item The whole point of our construction was to build a solution map that no longer depends on
$\eps$ and to hide all the singular $\eps$-dependency of the solution in the convergence of a suitable
sequence of ``models'' to a limit. In the case of \eqref{e:abstractFullPotential}, we still have 
an $\eps$-dependent fixed point problem. This is of course not a problem if its solution depends continuously
on $\eps$ as $\eps \to 0$,  but this is precisely not expected here: Theorem~\ref{theo:universal} states that
the effect of the quintic term is still felt in the limit since the constant appearing in front of the cubic
term in the limit is given by the cubic term of the effective potential $\scal{V_0}$ and not by that
of $V_0$.  
\end{claim}
In order to circumvent these problems, the idea introduced in \cite{Jeremy} 
is to build a regularity structure which encodes
the operation ``multiplication by $\eps$'' as a non-trivial operation in $\CD^\gamma$, thus allowing
us again to set up an $\eps$-independent fixed point problem and to shift all of the 
difficulties to that of proving that a suitable model converges to a limit.

For this, besides the symbols $\X$, $\sXi$ and the operator $\CI$ used before, we introduce
a new operator $\CE$ and we consider instead of \eqref{e:abstractFullPotential} a fixed point problem of the type
\begin{equ}[e:FPeps]
\Phi = \CP \one_{t>0}\bigl(\sXi - \Phi^3 - a\hat \CE \Phi^5\bigr) + P \Phi_0\;,
\end{equ}
where $\hat \CE$ will be an operator acting on $\CD^\gamma$ built from $\CE$ 
(in a way reminiscent of how $\CP$ and $\CK$ were built from $\CI$).
The recursive step \eqref{e:induction} in the construction of the regularity structure is then
replaced by  
\begin{equ}
\tau_i \in \CU \quad\Rightarrow\quad \CI(\tau_1\tau_2\tau_3) \in \CU\quad\&\quad \CI(\CE(\tau_1\tau_2\tau_3\tau_4\tau_5)) \in \CU\;.
\end{equ}
and the definition \eqref{e:defCW} of the collection of symbols $\CW$ is replaced by
\begin{equ}
\CW = \{\sXi\} \cup \{\tau_1\tau_2\tau_3\,:\, \tau_i \in \CU\} \cup \{\CE(\tau_1\tau_2\tau_3\tau_4\tau_5)\,:\, \tau_i \in \CU\}\;.
\end{equ}
We also define a slightly larger collection of symbols
\begin{equ}
\CW_\ex = \CW \cup \{\tau_1\tau_2\tau_3\tau_4\tau_5\,:\, \tau_i \in \CU\}\;,
\end{equ}
and we call similarly to before $T$ and $T_\ex$ the linear spans of $\CW$ and $\CE_\ex$ respectively.
In this construction, we consider the symbol $\CE$ as an ``integration operator of order $1$'' 
in the sense that 
we set $|\CE(\tau)| = |\tau| + 1$. 
This might seem strange at first sight: after all $\CE$ is 
supposed to represent multiplication by the number $\eps$, while the homogeneity is supposed to be
associated to some kind of regularity, but we will see that this is actually quite natural.

The structure group $\CG$ associated to the two graded
spaces $T$ and $T_\ex$ is then built in an analogous way to the construction of Section~\ref{sec:structure}.
The difference is that elements of $T_+$ are now not just polynomials in $X$ and $\II_k(\tau)$
as in \eqref{e:genPolynom}, but are also allowed to contain factors $\EE_k(\tau)$, with 
the maps $\Delta$ and $\Deltap$ defined as above, but with the additional rule
\begin{equ}
\Delta \CE(\tau) = (\CE \otimes \id)\Delta \tau + \sum_k {\X^k \over k!} \otimes \EE_k(\tau)\;,
\end{equ}
and similarly for $\Deltap$.
Given a smooth space-time function $\xi$, we now have a \textit{family} of canonical lifts
$\LL_\eps(\xi) = (\Pi^{(\eps)}, \Gamma^{(\eps)})$ of $\xi$ to models on $\TT_\ex$ given as before by 
\eqref{e:canonical} and \eqref{e:defAdmissible}, together with
\begin{equ}[e:canonicalEps]
\bigl(\Pi_z^{(\eps)} \CE(\tau)\bigr)(\bar z) = \eps \bigl(\Pi_z^{(\eps)} \tau\bigr)(\bar z)
- \eps\sum_{|k| < |\tau|+1} {(\bar z - z)^k \over k!} \bigl(D^k \Pi_z^{(\eps)} \tau\bigr)(z)\;.
\end{equ}
The operators $\Gamma_{z\bar z}^{(\eps)}$ are also built in exactly the same way as
in Section~\ref{sec:structure}, with the additional definition
\begin{equ}[e:deffeps]
f_z^{(\eps)}(\EE_\ell\tau) = - \eps \sum_{|k+\ell| < |\tau|+1} {(-z)^k \over k!} \bigl(D^{k+\ell} \Pi_z^{(\eps)} \tau\bigr)(z)\;,
\end{equ}
which in particular allows us to rewrite \eqref{e:canonicalEps} as
\begin{equ}[e:canonicalEps2]
\Pi_z^{(\eps)} \CE(\tau) = \eps \Pi_z^{(\eps)} \tau
+ \sum_{|k| < |\tau|+1} \Pi_z^{(\eps)} {(\X + z)^k \over k!} f_z^{(\eps)}(\EE_k\tau)\;.
\end{equ}
These definitions guarantee that if we define $\PPi^{(\eps)}$ as in Section~\ref{sec:structure},
one has
\begin{equ}[e:defPPieps]
\PPi^{(\eps)} \CE(\tau) = \eps \PPi^{(\eps)}\tau\;,
\end{equ}
and it is in this sense that $\CE$ represents multiplication by $\eps$.

\begin{remark}
These definitions deviate slightly from those given in \cite{Jeremy}. They are however
equivalent and amount to a simple change of basis in $T^+$. The reason for our
choice here is that it leads as in \eqref{e:defDeltap} to the nicer expression
\begin{equ}
\Deltap \EE_\ell(\tau) = (\EE_\ell \otimes \id) \Delta \tau + 
\sum_k {X^k\over k!} \otimes \EE_{\ell+k}(\tau)\;,
\end{equ}
for $\Deltap$, where we used similarly to before the convention that $\EE_\ell(\tau) = 0$ 
unless $|\tau| + 1 > |\ell|$.
\end{remark}

\begin{remark}
At this stage, we would like to point out one fundamental difference between
$\CE$ and $\CI$. In the case of $\CI$, the fact that it represents $K$ in
the sense of \eqref{e:defAdmissible} was made part of the definition of our
notion of admissible model, and we only ever considered renormalisation
procedures that preserve the set of admissible models. 
In the case of $\CE$ however, we only impose \eqref{e:deffeps} and
\eqref{e:defPPieps} for the canonical lifts $\LL_\eps(\xi)$. In general,
we do \textit{not} impose that these identities are satisfied for an arbitrary
admissible model. In particular, it is crucial to allow for renormalisation 
procedures that break them. This is reminiscent of the status
of the product which is typically not preserved by renormalisation.
\end{remark}

This construction also suggests a definition for the operator $\hat \CE$: 
given an admissible model $(\PPi, f)$, we set
\begin{equ}
\bigl(\hat \CE U\bigr)(z) = \CE U(z) - \sum_\ell {(\X + z)^\ell \over \ell!} f_z(\EE_\ell(U(z)))\;.
\end{equ}
Combining this with \eqref{e:canonicalEps2}, it follows immediately that for a canonical model
of the type $\LL_\eps(\xi)$, one has the identity
\begin{equ}
\CR \hat \CE U = \eps \CR U\;.
\end{equ}
This however is a particular property of such canonical models. In general, there is
no reason for $\CR \hat \CE U$ to be proportional to $\CR U$!
What is more, the operator $\hat \CE$ has a regularising property in the $\CD^\gamma$-scale
of spaces. More precisely, we have the following result, the proof of which
is given in \cite[Prop.~3.15]{Jeremy}.

\begin{proposition}
For the regularity structure $(T_\ex,\CG)$ just constructed, there exists 
$\delta > 0$ such that the operator $\hat \CE$ is bounded from
$\CD^\gamma$ to $\CD^{(\gamma + 1) \wedge \delta}$.
\end{proposition}

In particular, combining this with Theorem~\ref{theo:mult},
it turns out that if we set $|\sXi| = -{5\over 2} - \kappa$ for $\kappa > 0$ small enough, 
then the map $\Phi \mapsto \CP\hat \CE \Phi^5$ is
bounded from $\CD^\gamma$ into $\CD^\gamma$ as soon as $\gamma > 1+4\kappa$.
Similarly to before one can show that, for any admissible model and for sufficiently regular 
initial conditions, the fixed point problem \eqref{e:FPeps} admits short-time solutions, with
the existence time of these solutions uniformly bounded over bounded balls in the space of admissible 
models, and solutions being continuous not just as a function of their initial condition but 
also as a function of the model. Furthermore, for models of the type $\LL_\eps(\xi)$ (with $\xi$ continuous),
these solutions are such that $\phi = \CR \Phi$ solves
\begin{equ}
\d_t \phi = \Delta \phi - \phi^3 - a\eps \phi^5 + \xi\;,\qquad \phi(0,\cdot) = \Phi_0\;.
\end{equ}

\begin{remark}\label{rem:restrModel}
One crucial remark at this stage is that solutions to \eqref{e:FPeps} do not depend on the 
full model on $(T_\ex,\CG)$, but only on its restriction to the smaller regularity structure $(T,\CG)$.
In particular, these solutions make sense for \textit{any} admissible model on $(T,\CG)$, even if
such a model does not come from a model on $(T_\ex,\CG)$!
\end{remark}

As before, one cannot expect that $\LL_\eps(\xi_\eps)$ converges to a limit if $\xi_\eps$ 
is a mollification of white noise. However, it is possible to perform a renormalisation 
such that the models $M_\eps \LL_\eps(\xi_\eps)$ converge to a limit on the smaller
structure $(T,\CG)$. Even after renormalisation, these models do \textit{not} converge on $(T_\ex,\CG)$,
but thanks to Remark~\ref{rem:restrModel} this does not matter for us.
Let us now describe the renormalisation maps $M$ suitable in this setting. We keep the same
graphical notation for basis vectors of $T / T_\ex$ as previously, with $\CE$ depicted by a
circle. With this notation, one has for example
\begin{equ}
\CE(\CI(\Xi)^5) = \<E5>\;,\qquad
\CI(\Xi)^2\CI\bigl(\CE(\CI(\Xi)^5)\bigr) = \<E52>\;.
\end{equ}
Since we only work with Gaussian approximations, we only keep $L_1$ and $L_3$ from
Section~\ref{sec:renorm}, which we rename $\tilde L_1$ and $\tilde L_2$ to avoid confusion. 
The operator $\tilde L_1$ is extended to all of $T_\ex$ in the same way as before by furthermore
postulating that it commutes with $\CE$, so that one has for example
\begin{equ}
\tilde L_1 \<E5> = \binom{5}{2} \<E3>\;,\qquad 
\tilde L_1 \<E52> = \binom{5}{2} \<E32> + \<E50>\;. 
\end{equ}
In the case of $\tilde L_2$, we extend it to $T_\ex$ by simply postulating that 
it vanishes on all symbols except for \<32> and \<22>, so we do for example
set $\tilde L_2 \<3E2> = 0$ rather than $\tilde L_2 \<3E2> = \CE(\1)$.
Finally, we add two more operators $\tilde L_3$ and $\tilde L_4$ given by
\begin{equ}[e:def34]
\tilde L_3 \<3E3> = \1 \;,\qquad 
\tilde L_3 \<3E4> = 4\<1> \;,\qquad 
\tilde L_4 \<E4E4> = \1 \;,\qquad 
\tilde L_4 \<E5E4> = 5\<1> \;,
\end{equ}
as well as $\tilde L_i \tau = 0$ for all combinations of $i \in \{3,4\}$ and $\tau \in \CW$
not appearing in \eqref{e:def34}.
The rationale here is that $\tilde L_3$ ``contracts'' all instances of \<3E3>,
which appears $4$ times in \<3E4> and nowhere else, and similarly for $\tilde L_4$.
Similarly to before, we then consider renormalisation maps of the form 
$\tilde M = \exp\bigl(-\sum_i C_i \tilde L_i\bigr)$. There is an analogue in this context to 
Proposition~\ref{prop:renorm}, which shows that these renormalisation maps are ``legal''
in the sense that they induce an action on admissible models with the property that 
a model represented by $\PPi$ is mapped to one represented by $\PPi \tilde M$.

By a calculation analogous to that of Proposition~\ref{prop:renormEquation}, one can then show the following.

\begin{proposition}\label{prop:renormBig}
Let $\tilde M$ be as above and let $(\Pi^{\tilde M},\Gamma^{\tilde M})$ be the model obtained by acting 
with $\tilde M$ on $(\Pi,\Gamma) = \LL_\eps(\xi)$ for some smooth function $\xi$ and some $\eps > 0$.
Let furthermore $\Phi$ be the solution to \eref{e:FPeps} with respect to the model 
$(\Pi^{\tilde M},\Gamma^{\tilde M})$. Then, the function
$u(t,x) = \bigl(\CR^{\tilde M} \Phi\bigr)(t,x)$ solves the equation
\begin{equ}
\d_t u = \Delta u - H_3(u,\tilde C_1) - a\eps H_5(u,\tilde C_1) - (9 \tilde C_2 + 20a \tilde C_3 + 25a^2 \tilde C_4)u + \xi\;,
\end{equ}
where $H_n$ denotes the $n$th Hermite polynomial as before.
\end{proposition}

Given $c \in \R$ and $\eps > 0$, we now define 
$\tilde M_{c,\eps} = \exp\bigl(-\sum_i C_i^{(c,\eps)} \tilde L_i\bigr)$ by setting
\begin{equ}
\tilde C_1^{(c,\eps)} = \tilde C_1^{(\eps)}\;,\quad 
\tilde C_2^{(c,\eps)} = \tilde C_2^{(\eps)}\;,\quad
\tilde C_3^{(c,\eps)} = c \tilde C_3^{(\eps)}\;,\quad
\tilde C_4^{(c,\eps)} = c^2 \tilde C_4^{(\eps)}\;,
\end{equ}
with
\begin{equ}
\tilde C_1^{(\eps)} = 
\begin{tikzpicture}  [baseline=10] 
\node[root]	(root) 	at (0,0) {};
\node[cumu2n]	(a)  		at (0,1) {};	
\draw[cumu2] (a) ellipse (8pt and 4pt);
\draw[kernel] (a.east) node[int]   {} to [bend left = 60] (root);
\draw[kernel] (a.west) node[int]   {} to [bend right = 60] (root);
\end{tikzpicture} 
\;,\quad
\tilde C_2^{(\eps)} = 2\;
\begin{tikzpicture}[scale=0.4,baseline=0cm]
	\node at (0,-1)  [root] (root) {};
	\node at (0,1.8)  [int] (middle) {};
	\node[cumu2n] at (-1,0.4) (left) {};
	\draw[cumu2] (left) ellipse (10pt and 20pt);
	\node[cumu2n] at (1,0.4) (right) {};
	\draw[cumu2] (right) ellipse (10pt and 20pt);
	
	\draw[kernel,bend right=40] (left.south) node[int]{} to  (root);	
	\draw[kernel,bend left=40] (left.north) node[int]{} to  (middle);	
	\draw[kernel,bend left=40] (right.south) node[int]{} to  (root);	
	\draw[kernel,bend right=40] (right.north) node[int]{} to  (middle);	
	\draw[kernel] (middle) to (root);
\end{tikzpicture}\;,\quad
\tilde C_3^{(\eps)} = 3! \eps\,
\begin{tikzpicture}[scale=0.4,baseline=0cm]
	\node at (0,-1)  [root] (root) {};
	\node at (0,1.8)  [int] (middle) {};
	\node[cumu2n] at (-1,0.4) (left) {};
	\draw[cumu2] (left) ellipse (10pt and 20pt);
	\node[cumu2n] at (0.8,0.4) (right) {};
	\draw[cumu2] (right) ellipse (10pt and 20pt);
	\node[cumu2n] at (1.8,0.4) (rright) {};
	\draw[cumu2] (rright) ellipse (10pt and 20pt);
	
	\draw[kernel,bend right=40] (left.south) node[int]{} to  (root);	
	\draw[kernel,bend left=40] (left.north) node[int]{} to  (middle);	
	\draw[kernel,bend left=30] (right.south) node[int]{} to  (root);	
	\draw[kernel,bend right=30] (right.north) node[int]{} to  (middle);	
	\draw[kernel,bend left=50] (rright.south) node[int]{} to  (root);	
	\draw[kernel,bend right=50] (rright.north) node[int]{} to  (middle);	
	\draw[kernel] (middle) to (root);
\end{tikzpicture}\;,\quad
\tilde C_4^{(\eps)} = 4! \eps^2\;
\begin{tikzpicture}[scale=0.4,baseline=0cm]
	\node at (0,-1)  [root] (root) {};
	\node at (0,1.8)  [int] (middle) {};
	\node[cumu2n] at (-1.8,0.4) (lleft) {};
	\draw[cumu2] (lleft) ellipse (10pt and 20pt);
	\node[cumu2n] at (-0.8,0.4) (left) {};
	\draw[cumu2] (left) ellipse (10pt and 20pt);
	\node[cumu2n] at (0.8,0.4) (right) {};
	\draw[cumu2] (right) ellipse (10pt and 20pt);
	\node[cumu2n] at (1.8,0.4) (rright) {};
	\draw[cumu2] (rright) ellipse (10pt and 20pt);
	
	\draw[kernel,bend right=50] (lleft.south) node[int]{} to  (root);	
	\draw[kernel,bend left=50] (lleft.north) node[int]{} to  (middle);	
	\draw[kernel,bend right=30] (left.south) node[int]{} to  (root);	
	\draw[kernel,bend left=30] (left.north) node[int]{} to  (middle);	
	\draw[kernel,bend left=30] (right.south) node[int]{} to  (root);	
	\draw[kernel,bend right=30] (right.north) node[int]{} to  (middle);	
	\draw[kernel,bend left=50] (rright.south) node[int]{} to  (root);	
	\draw[kernel,bend right=50] (rright.north) node[int]{} to  (middle);	
	\draw[kernel] (middle) to (root);
\end{tikzpicture}\;,
\end{equ}
with the same graphical notations as in the previous section.
Here, $\tilde C_1$ diverges like $\eps^{-1}$ and $\tilde C_2$ diverges logarithmically,
but the remaining two constants actually converge to finite limits as $\eps \to 0$.
The reason for these definitions is the following result, which is the crucial
point on which the proof of Theorem~\ref{theo:universalWeak} hinges.

\begin{theorem}\label{theo:convIndep}
Let $\xi_\eps = \rho_\eps * \xi$ for some space-time white noise $\xi$ and,
for $c \in \R$ and $\eps > 0$, let
\begin{equ}
\fM_{c,\eps} = \tilde M_{c,\eps} \LL_{c\eps}(\xi_\eps)\;,
\end{equ} 
with $\tilde M_{c,\eps}$ as above. Then, the limit $\lim_{\eps \to 0} \fM_{c,\eps}$
exists in the space of admissible models on $(T,\CG)$ and is independent of $c$.
\end{theorem}

\begin{remark}
It is \textit{not} true that the models $\fM_{c,\eps}$ converge to a limit on
the extended regularity structure $(T_\ex,\CG)$. Fortunately, this is not needed
for our argument.
\end{remark}

The proof of this result follows the same line of argument as in Section~\ref{sec:convGauss}.
The reason why the limit turns out to be independent of $c$ is that the value of $c$ only 
affects $\Pi_z \tau$ for symbols $\tau$ containing at least one instance of $\CE$.
Our choice of renormalisation constants is then precisely such that, in the limit $\eps \to 0$,
all of these terms vanish.

Combining Theorem~\ref{theo:convIndep} with Proposition~\ref{prop:renormBig} and the fact that 
solutions to \eqref{e:FPeps}
depend continuously on the underlying model, we conclude that the solutions to
\begin{equ}
\d_t u = \Delta u - H_3(u,\tilde C_1^{(\eps)}) - a\eps H_5(u,\tilde C_1^{(\eps)}) - (9 \tilde C_2^{(\eps)} + 20a \tilde C_3^{(\eps)} + 25a^2 \tilde C_4^{(\eps)})u + \xi_\eps\;,
\end{equ}
as well as those to
\begin{equ}
\d_t u = \Delta u - u^3 +(3\tilde C_1^{(\eps)} - 9 \tilde C_2^{(\eps)})u + \xi_\eps\;,
\end{equ}
both converge to the same limit as $\eps \to 0$.
Since the limit of the latter is the $\Phi^4_3$ model by definition, the claim of Theorem~\ref{theo:universalWeak} 
follows at once.

\bibliographystyle{Martin}
\bibliography{./refs}

\end{document}